\documentclass[english,1p]{amsart}
\usepackage[latin9]{inputenc}
\usepackage{geometry}
\usepackage{amsthm}
\usepackage{amsmath}
\usepackage{amssymb}
\usepackage{esint}
\usepackage{tikz}
\usepackage{hyperref}

\makeatletter
\numberwithin{equation}{section}
\numberwithin{figure}{section}

\theoremstyle{plain}
\newtheorem{thm}{Theorem}
  \theoremstyle{plain}
  \numberwithin{thm}{section}
  \newtheorem{cor}[thm]{Corollary}
  \theoremstyle{plain}
  \newtheorem{lem}[thm]{Lemma}
  \theoremstyle{plain}
   \newtheorem{proposition}[thm]{Proposition}
  \theoremstyle{remark}
  \newtheorem{rem}[thm]{Remark}
  \newtheorem{ex}[thm]{Example}
 
  \def\Ddots{\mathinner{\mkern1mu\raise\p@
\vbox{\kern7\p@\hbox{.}}\mkern2mu
\raise4\p@\hbox{.}\mkern2mu\raise7\p@\hbox{.}\mkern1mu}}
\makeatother

\usepackage{babel}

\newcommand{\eps}{\varepsilon}

\newcommand{\phw}{\widetilde \Phi^{wk}}

\newcommand{\norm}[1]{\left\| #1 \right\|}
\newcommand{\mklm}[1]{\left\{ #1 \right\}}
\newcommand{\eklm}[1]{\left\langle #1 \right\rangle}

\renewcommand{\d}{\,d}
\newcommand{\N}{{\mathbb N}}
\newcommand{\Z}{{\mathbb Z}}
\newcommand{\C}{{\mathbb C}}
\newcommand{\Ccal}{{\mathcal C}}
\newcommand{\R}{{\mathbb R}}

\newcommand{\E}{{\mathcal E}}
\newcommand{\F}{{\mathcal F}}

\renewcommand{\H}{{\mathcal H}}

\newcommand{\J}{{\mathcal J }}
\newcommand{\Jbb}{{\mathbb J }}
\newcommand{\M}{{\mathcal M}}
\newcommand{\T}{{\rm T}}

\renewcommand{\O}{{\mathcal O}}

\newcommand{\Xbf}{{\mathbf X}}

\newcommand{\0}{{\rm 0}}
\newcommand{\1}{{\bf 1}}
\renewcommand{\epsilon}{\varepsilon}

\renewcommand{\rho}{\varrho}

\newcommand{\Cinft}{{\rm C^{\infty}}}
\newcommand{\CT}{{\rm C^{\infty}_c}}

\renewcommand{\L}{{\rm L}}
\newcommand{\Lcal}{{\mathcal L}}
\newcommand{\Ncal}{{\mathcal N}}

\renewcommand{\S}{{\mathcal S}}
\newcommand{\G}{{\mathcal G}}

\newcommand{\SO}{\mathrm{SO}}

\newcommand{\g}{{\bf \mathfrak g}}

\renewcommand{\t}{{\bf \mathfrak t}}

\newcommand{\U}{{\mathcal U}}

\newcommand{\Ad}{\mathrm{Ad}\,}

\newcommand{\id}{\mathrm{id}\,}

\renewcommand{\det}{\mathrm{det}\,}

\renewcommand{\Re}{\mathrm{Re}\,}
\newcommand{\vol}{\text{vol}\,}
\newcommand{\dist}{\text{dist}\,}

\newcommand{\Crit}{\mathrm{Crit}}

\DeclareMathOperator{\supp}{supp\,}

\DeclareMathOperator{\tr}{tr}
\DeclareMathOperator{\gd}{\partial}

\DeclareMathOperator{\grad}{grad}

\newcommand{\e}[1]{\,{\mathrm e}^{#1}\,}
\newcommand{\dbar}{{\,\raisebox{-.1ex}{\={}}\!\!\!\!d}}

\newcommand{\bdm}{\begin{displaymath}}
\newcommand{\edm}{\end{displaymath}}
\newcommand{\bq}{\begin{equation}}
\newcommand{\eq}{\end{equation}}
\newcommand{\bqn}{\begin{equation*}}
\newcommand{\eqn}{\end{equation*}}

\begin{document}
\author{Pablo Ramacher}
\email[Pablo Ramacher]{\href{mailto:ramacher@mathematik.uni-marburg.de}{ramacher@mathematik.uni-marburg.de}}
\address{Philipps-Universit\"at Marburg, FB  12 Mathematik und Informatik, Hans-Meerwein-Str., 35032 Marburg}
\title[The equivariant spectral function of an invariant elliptic operator]{The equivariant spectral function of an invariant elliptic operator. $\L^p$-bounds, caustics, and concentration of  eigenfunctions}
\date{September 17, 2017}

\begin{abstract} Let $M$ be a compact boundaryless Riemannian manifold, carrying an  effective and isometric action of a compact   Lie group $G$,  and $P_0$ an invariant elliptic classical pseudodifferential operator on $M$. Using Fourier integral operator techniques, we prove a local Weyl law  with remainder estimate  for the equivariant (or reduced) spectral function  of $P_0$ for each isotpyic component in the Peter-Weyl decomposition of $\L^2(M)$, generalizing work of Avacumovi\v{c}, Levitan, and H\"ormander. From this we deduce  a generalized Kuznecov sum formula for periods of G-orbits, and recover the local Weyl law for orbifolds shown by Stanhope and Uribe. Relying on recent results on singular equivariant asymptotics of oscillatory integrals, we  further characterize  the  caustic behaviour of the reduced spectral function near singular  orbits, which allows us to give corresponding point-wise bounds  for  clusters of eigenfunctions in specific isotypic components.  In  case that $G$ acts on $M$ without singular orbits, we are able to deduce hybrid   $\L^p$-bounds  for  $2 \leq p \leq \infty$ in the eigenvalue and isotypic aspect that improve on the classical estimates of Seeger and Sogge for generic eigenfunctions. Our results are sharp in the eigenvalue aspect, but not in the isotypic aspect, and reduce to the classical ones in the case $G=\mklm{e}$. 
\end{abstract}

\maketitle

\setcounter{tocdepth}{1}
\tableofcontents{}

\section{Introduction}

In this paper, we derive an asymptotic formula  with remainder estimate for the equivariant (or reduced) spectral function of  an invariant elliptic operator on a compact Riemannian manifold with an effective and isometric action of a compact Lie group $G$, generalizing previous work of Avacumovi\v{c} \cite{avacumovic}, Levitan\cite{levitan52}, H\"ormander \cite{hoermander68}, and,  more recently, Stanhope and Uribe \cite{stanhope-uribe}. If $G$ acts on $M$ with orbits of the same dimension,   we obtain hybrid $\L^p$-bounds for eigenfunctions in the eigenvalue and isotypic aspect that improve on  the classical estimates for generic eigenfunctions proved by Seeger and Sogge \cite{sogge88, seeger-sogge}, but cannot hold when singular orbits are present. In the latter  case, we are able to describe the caustic   behaviour of the reduced spectral function as one approaches  orbits of singular type, relying on our recent work  \cite{ramacher10} on singular equivariant asymptotics  obtained via desingularization techniques.  As an application, we are able to prove  point-wise bounds  for isotypic clusters of eigenfunctions, showing that they tend to concentrate on singular orbits. Since very little can be said about the shape of eigenfunctions in general, this result is rather striking. In particular, this gives a new interpretation of the classical bounds for  spherical harmonics in terms of caustics of the equivariant spectral function, generalizing them to eigenfunctions on arbitrary compact manifolds with symmetries.  The concentration of eigenfunctions along singular orbits was already observed in \cite{kuester-ramacher15} for Schr\"odinger operators in the context of equivariant quantum ergodicity under the additional assumption that the reduced Hamiltonian flow is ergodic. Our results can be viewed as part of  the more general problem of studying  the eigenfunctions  of  a commuting family of differential operators on  a  general compact manifold  that are independent in some sense  \cite{marshall16}.

To explain our results, consider  a closed\footnote{By a closed manifold we will understand a compact manifold without boundary.} connected Riemannian manifold $M$ of dimension $n$, together with an elliptic classical pseudodifferential operator 
\bqn 
P_0:\Cinft(M) \, \longrightarrow \, \L^2(M)
\eqn
of degree $m$, where $\Cinft(M)$ denotes the space of smooth functions on $M$ and $\L^2(M)$ the Hilbert space of square integrable functions  with respect to the Riemannian  volume density $dM$ on $M$. We assume that $P_0$ is positive and symmetric, so that it has a unique self-adjoint extension $P$. Furthermore, the compactness of $M$ implies that $P$ has discrete spectrum. Let $\mklm{E_\lambda}$ be a spectral resolution of $P$, and denote by $e(x,y,\lambda)$  the Schwartz kernel of $E_\lambda$, which is called  the \emph{spectral function} of $P$. Within the theory of Fourier integral operators one can then show the following \emph{local Weyl formula} \cite{avacumovic, levitan52, hoermander68}
 \bq
\label{eq:1.1}
\Big |e(x,x,\lambda)-\frac{\lambda^{\frac nm}}{(2\pi)^n} \int_{\{p(x,\xi)< 1\}} \d \xi \Big | \leq C \lambda^{\frac{n-1}{m}}, \qquad   x \in M, \, \lambda \to +\infty ,
\eq
 for some constant $C>0$ independent of $x$ and $\lambda$,  $p$ being the principal symbol of $P_0$. By integrating over $M$ one  deduces from this  for  the \emph{spectral counting function} $  N(\lambda):=\sum_{t\leq \lambda} \dim \E_t=\int_M e(x,x,\lambda) \d M(x)$  the 
\emph{global Weyl formula}
\bqn
  N(\lambda)=\frac{\vol S^\ast M}{n(2\pi)^n} \lambda^{\frac nm} + O(\lambda^{\frac{n-1}m}),
\eqn
 where $\E_t$ denotes the eigenspace of $P$ belonging to the eigenvalue $t$ and $S^\ast M$ the co-sphere bundle $\mklm{(x,\xi) \in T^\ast M\mid p(x,\xi)=1}$. In order to show the stronger point-wise formula \eqref{eq:1.1} one first proves the estimate
 \bq
\label{eq:1.2}
|e(x,x,\lambda+1)- e(x,x,\lambda)|\leq C \cdot \lambda^{\frac{n-1}{m}}, \qquad x \in M,
\eq
which describes   the order of magnitude of the discontinuities of $N(\lambda)$ or, more generally, the amount of eigenvalues in the interval $(\lambda, \lambda+1]$ as $\lambda \to +\infty$, yielding the asymptotics $N(\lambda+1)-N(\lambda)=O(\lambda^{\frac{n-1}m})$. The  bound \eqref{eq:1.2} is equivalent to 
\bq
\sum_{{\lambda_j \in (\lambda,\lambda+1]}} |e_j(x)|^2 \leq C \cdot \lambda^{\frac{n-1}{m}}, \qquad x \in M,
\eq
where $\mklm{e_j}_{j\geq 0}$ denotes an arbitrary  orthonormal basis of eigenfunctions of $P$  in $\L^2(M)$  with corresponding eigenvalues $\mklm{\lambda_j}_{j \geq 0}$, and actually implies  the bound
\bq
\label{eq:1.4}
\norm{\chi_\lambda u}_{\L^\infty(M)} \leq C(1+\lambda)^{\frac{n-1}{2m}} \norm{u}_{\L^2(M)}, \qquad u \in \L^2(M),
\eq
where $\chi_\lambda$ denotes the spectral projection onto the sum of eigenspaces with eigenvalues in the interval  $(\lambda, \lambda+1]$ with Schwartz kernel $\chi_\lambda(x,y)=e(x,y,\lambda+1) - e(x,y,\lambda)$, since $\norm {\chi_\lambda}_{\L^2 \to \L^\infty}^2\equiv \sup_{x\in M} \chi_\lambda(x,x)$. From this  the estimate for $N(\lambda+1)-N(\lambda)$ immediately follows by taking the trace of $\chi_\lambda$. In particular,  one deduces from \eqref{eq:1.4}  the bound for eigenfunctions
\bq
\label{eq:1.4a}
\norm{u}_{\L^\infty(M)} \leq C \, \lambda^{\frac{n-1}{2m}}, \qquad u \in \E_\lambda, \, \norm{u}_{\L^2}=1.
\eq
Under the additional assumption that the co-spheres $S_x^\ast M$ are strictly convex, Seeger and Sogge \cite{seeger-sogge} were also able to prove upper bounds for  $\L^p$-norms of eigenfunctions via analytic interpolation techniques, generalizing previous work of Sogge for second order elliptic differential operators \cite{sogge88}. More precisely, let
\bqn 
\delta_n(p):=\max \left ( n \left |\frac 1 2 -\frac 1p \right | -\frac 12,0 \right ).
\eqn  
Then, for $u \in \E_\lambda, \, \norm{u}_{\L^2}=1$ one has 
\bq
\label{eq:27.12.2015}
\norm{u}_{\L^p(M)} \leq \begin{cases} C \lambda ^{\frac{\delta_n(p)}{m}}, & \frac{2(n+1)}{n-1} \leq p \leq \infty, \\ C \lambda^{\frac{(n-1)(2-p')}{4mp'}}, & 2 \leq p \leq \frac{2(n+1)}{n-1}, \end{cases}
\eq
where $\frac1p+\frac 1{p'}=1$.

In this paper, we shall sharpen the bounds \eqref{eq:1.1}--\eqref{eq:27.12.2015} in the presence of  symmetries. To explain our results, assume that  $M$ carries an effective and isometric action of a compact Lie group $G$ with Lie algebra $\g$ and orbits of dimension less or equal $n-1$. The group $G$ might be disconnected or even finite, though the case of interest is when $G$ is continuous. Suppose that $P$ commutes with the {left-regular representation} $(\pi,\L^2(M))$ of $G$ in $\L^2(M)$ given by 
\bqn 
\pi(g) u (x) = u(g^{-1}\cdot x), \qquad u \in \L^2(M),
\eqn
 so that each eigenspace of $P$ becomes a unitary $G$-module. If $\widehat G$ denotes the set of equivalence classes of irreducible unitary representations of $G$, which we shall  identify with the set of characters of $G$, the Peter-Weyl theorem asserts that 
 \bq
\label{eq:PW} 
\L^2(M)=\bigoplus_{\gamma \in \widehat G} \L^2_\gamma(M),
\eq
a Hilbert sum decomposition, where $\L^2_\gamma(M):=\Pi_\gamma \L^2(M)$ denotes the $\gamma$-isotypic component, and $\Pi_\gamma$  the corresponding projection. Assume that the orthonormal basis $\mklm{e_j}_{j \geq 0}$ has been chosen such that it is compatible with the decomposition \eqref{eq:PW}, and let $e_\gamma(x,y,\lambda)$ be the spectral function of the operator $P_\gamma:=\Pi_\gamma \circ P\circ \Pi_\gamma=P\circ \Pi_\gamma=\Pi_\gamma \circ P$, which is also called the \emph{reduced spectral function} of $P$. Further, let $\Jbb:T^\ast M \to \g^\ast$ denote the momentum map of the Hamiltonian $G$-action on $T^\ast M$, induced by the action of $G$ on $M$,  and write $\Omega:=\Jbb^{-1}(\mklm{0})$.  
As our first result, we show in  Theorem \ref{thm:main}   the \emph{equivariant local Wey law}
 \bq
 \label{eq:14.05.2015}
\left |e_\gamma(x,x,\lambda)- \lambda^{\frac{n-\kappa_x}{m}} \frac{d_\gamma [\pi_{\gamma|G_x}:\1]}{(2\pi)^{n-\kappa_x}}  \int_{\{ (x,\xi) \in \Omega, \, p(x,\xi)< 1\}} \frac{ \d \xi}{\vol \O_{(x,\xi)}} \right | \leq C_{x,\gamma} \, \lambda^{\frac{n-\kappa_x-1}{m}}, \quad x \in M, 
\eq as $\lambda \to +\infty $, where $\kappa_x:=\dim \O_x$ is the dimension of the $G$-orbit through $x$,  $d_\gamma$ denotes the dimension of an irreducible $G$-representation $\pi_\gamma$ belonging to $\gamma$ and $ [\pi_{\gamma|G_x}:\1]$ the multiplicity of the trivial representation  in the restriction of $\pi_\gamma$ to the isotropy group $G_x$ of $x$, while $C_{x,\gamma}>0 $ is a constant depending on $x$ and $\gamma$.  It should be  emphasized that $\kappa_x$, and therefore also the leading term and the constant $C_{x,\gamma}$, which are independent of $\lambda$,   will in general depend in a highly non-uniform way on $x\in M$. In fact, the description of  $e_\gamma(x,y,\lambda)$ reduces in essence to the study of   oscillatory integrals of the form 
\bq
\label{eq:1.12.2015}
I_{x,y}(\mu):=\int_{G}\int_{S^{\ast}_x Y} e^{i\mu \Phi_{x,y}(\omega,g)} a (x,y,\omega,g) \d (S^{\ast}_x Y)(\omega) \d g, \qquad  \, \mu >0,
\eq
with phase function 
\bqn 
\Phi_{x,y}(\omega,g):=  \eklm{\kappa(x) - \kappa( g\cdot y), \omega},
\eqn 
where $(Y,\kappa)$ is a local chart on $M$ and  $a \in \CT$ an amplitude that might depend on $\mu$ and is such that $(x,y, \omega, g) \in \supp a$ implies $x, g \cdot y \in Y$, while $d(S^\ast Y)$ and $dg$ denote  Liouville and Haar measure, respectively.  Now, when trying to describe the asymptotic behaviour of $I_{x,x}(\mu)$ as $\mu \to +\infty $ uniformly in $x$ via the stationary phase principle, one encounters the phenomenon that the critical set of $\Phi_{x,x}$ changes abruptly its dimension when $x$ passes through points of  singular orbits, leading to a drastic change in the asymptotics of $I_{x,x}(\mu)$. Such points are called \emph{caustics} \cite{varadarajan97}, and are ultimately responsible for the qualitatively very different asymptotic behaviour of  the reduced spectral function  as $x$ approaches  such points. A precise description of the asymptotics of the integrals \eqref{eq:1.12.2015} is given in Theorems \ref{thm:12.05.2015}, \ref{thm:14.05.2017}, and Proposition \ref{prop:10.08.2017}.

Though the leading coefficient in the asymptotic formula \eqref{eq:14.05.2015}  for $e_\gamma(x,x,\lambda)$  is explicit, and has a clear geometric meaning, it does not unveil the caustic nature of $e_\gamma(x,x,\lambda)$ when singular orbits are present, and   blows up  in an unknown way  as $x$ approaches  such orbits. To obtain a precise description of this  caustic behaviour it is  necessary to examine the integrals \eqref{eq:1.12.2015}  more carefully. For this, we shall  rely on our recent results \cite{ramacher10} on singular equivariant asymptotics  obtained via resolution of singularities,  from which we will be able to deduce  a uniform description of the integrals $I_{x,x}(\mu)$ and  the behaviour of $e_\gamma(x,x,\lambda)$ near singular   orbits.  More precisely, consider the stratification  $M=M(H_1) \, \dot \cup \dots \dot \cup \, M(H_L)$ of $M$ into orbit types, arranged in such a way that 
$(H_i) \leq (H_j)$ implies $i \geq j$, and let $\Lambda$ be the maximal length that a maximal totally ordered subset of isotropy types can have. Write  $M_\mathrm{prin}:=M(H_L)$, $M_\mathrm{except}$, and $M_\mathrm{sing}$ for the union of all orbits of principal, exceptional, and singular type, respectively, so that 
\bq
\label{eq:15.08.2016}
M= M_\mathrm{prin}\, \dot \cup \, M_\mathrm{except}\, \dot \cup \, M_\mathrm{sing},
\eq
  and denote by $\kappa:=\dim G/H_L$ the dimension of an orbit of principal type.   Then, by   Theorem \ref{thm:15.11.2015} one has  for $x \in M_\mathrm{prin}\cup M_\mathrm{except}$  and $\lambda \to +\infty $ the \emph{singular equivariant local Weyl law}
   \begin{align}
 \label{eq:17.06.2016}
 \begin{split}
 \Big |e_\gamma(x,x,\lambda)&- \frac{d_\gamma \lambda^{\frac{n-\kappa}{m}}}{(2\pi)^{n-\kappa}} \sum_{N=1}^{\Lambda-1} \,  \sum_{{i_1<\dots< i_{N} }} \, \prod_{l=1}^{N}   |\tau_{i_l}|^{\dim G- \dim H_{i_l}-\kappa}  \mathcal{L}^{0,0}_{i_1\dots i_{N} }(x,\gamma) \Big  |\\ 
&\leq C_\gamma \lambda^{\frac{n-\kappa-1}m} \sum_{N=1}^{\Lambda-1}\, \sum_{{i_1<\dots< i_{N}}}   \prod_{l=1}^N     |\tau_{i_l}|^{\dim G- \dim H_{i_l}-\kappa-1}, 
\end{split}
\end{align}
 where the multiple sums run over all possible totally ordered subsets $\mklm{(H_{i_1}),\dots, (H_{i_N})}$ of singular isotropy types, the  coefficients $\mathcal{L}^{0,0}_{i_1\dots i_{N}}$  are explicitly given and bounded functions  in $x$, and  $\tau_{i_j} =\tau_{i_j}(x)\in (-1,1)$ are desingularization parameters that arise in the resolution process satisfying  $|\tau_{i_j}|\approx \dist (x, M(H_{i_j}))$, while  $C_\gamma>0$ is a constant independent of $x$ and $\lambda$. Thus, {the combinatorial complexity of the underlying group action is reflected in  the asymptotic shape of the equivariant spectral function}. By integrating the asymptotic formulae \eqref{eq:14.05.2015} and \eqref{eq:17.06.2016}  over $x\in M$, one obtains for the equivariant counting function $N_\gamma(\lambda):= \int_M e_\gamma(x,x,\lambda) \d M(x)$ the \emph{equivariant Weyl law}
 \bq
 \label{eq:weylglobal}
N_\gamma(\lambda)= \frac{d_\gamma [{\pi_\chi}_{|H_L}:\1]}{(n-\kappa)(2\pi)^{n-\kappa}}   \mathrm{vol} \, [(\Omega \cap S^\ast M)/G]  \,  {\lambda} ^{\frac{n-\kappa}m }   + O_\gamma\big (\lambda^{{(n-\kappa-1)}/m} (\log \lambda)^{\Lambda} \big ).
\eq
This was the main result of \cite{ramacher10}. Notice that in spite of the fact that the desingularization techniques developed there are necessary to establish the remainder estimate in \eqref{eq:weylglobal}, singular and exceptional orbits, being of measure zero, do not contribute to the equivariant Weyl law \eqref{eq:weylglobal}, and remain hidden. It is only in the stronger local Weyl laws  \eqref{eq:14.05.2015} and \eqref{eq:17.06.2016} for the reduced spectral function that the whole orbit structure of the underlying group action becomes manifest.

As a major consequence, Theorems \ref{thm:main} and  \ref{thm:15.11.2015} lead to  refined bounds for  eigenfunctions. In the non-singular case, that is, when only principal and exceptional orbits are present,  and consequently
all $G$-orbits have the same dimension $\kappa$,     the obtained bounds are uniform in $x\in M$, while in the singular case, they show that eigenfunctions tend to concentrate along lower dimensional orbits. Indeed, as in the non-equivariant case, the crucial bound  for obtaining \eqref{eq:14.05.2015} is a bound for $e_\gamma(x,x,\lambda+1) - e_\gamma(x,x,\lambda)$, which is equivalent to the non-uniform bound
\bq
\label{eq:clust}
\sum_{\stackrel{\lambda_j \in (\lambda,\lambda+1],}{ e_j \in \L^2_\gamma(M)}} |e_j(x)|^2\leq C_{x, \gamma} \, \lambda^{\frac{n-\kappa_x-1}{m}}, \quad x \in M,
\eq
see Corollary \ref{cor:21.06.2016}.  
From this one immediately deduces in the non-singular case  the hybrid $\L^\infty$-estimate in the eigenvalue and isotypic aspect
\bqn
\norm{(\chi_\lambda\circ \Pi_\gamma) u}_{\L^\infty(M)} \leq C_{\gamma}\,  (1+ \lambda)^{\frac{n-\kappa-1}{2m}} \norm{u}_{\L^2(M)}, \qquad u \in \L^2(M),
\eqn
 where $C_{\gamma}>0$ is a constant independent of $\lambda$ satisfying the estimate
 \bq
 \label{eq:24.07.2017}
 C_\gamma \ll \sqrt{d_\gamma \sup_{l \leq \lfloor \kappa/2+1\rfloor} \norm{D^l\gamma}_\infty},
 \eq
 see Proposition \ref{thm:bounds} and \eqref{eq:1.6.2017}. In particular,  we obtain the hybrid equivariant bound for eigenfunctions
\bqn 
\norm{u}_{\L^\infty(M)} \ll C_\gamma \, \lambda^{\frac{n-\kappa-1}{2m}}, \qquad  u \in \E_\lambda \cap \L^2_\gamma(M), \quad \norm{u}_{\L^2}=1.
\eqn
 Note that if $n=\kappa+1$, this bound reads   $\norm{u}_\infty \leq C_\gamma$. The proof of $\L^p$-bounds is considerably more envolved, since it no longer suffices to study the integrals $I_{x,y}(\mu)$ restricted to the diagonal. Instead,  it is necessary to estimate their growth as $\mu\to +\infty $ in a neighborhood of the latter, for which we have to assume that  the co-spheres $S_x^\ast M$ are strictly convex. Using complex interpolation techniques, we then prove in Theorem \ref{thm:20.02.2016} the hybrid bounds in the eigenvalue and isotypic aspect
\bqn
\norm{(\chi_\lambda \circ \Pi_\gamma) u}_{\L^q(M)} \leq \begin{cases} C_{\gamma} \,  \lambda^{\frac{\delta_{n-\kappa}(q)}{m}} \norm{u}_{\L^2(M)}, &  \frac{2(n-\kappa+1)}{n-\kappa-1} \leq q \leq \infty, \vspace{2mm} \\ C_{\gamma} \, \lambda^{\frac{(n-\kappa-1)(2-q')}{4m q'}} \norm{u}_{\L^2(M)}, &  2 \leq q \leq \frac{2(n-\kappa+1)}{n-\kappa-1}, \end{cases} 
\eqn
 where $\frac 1q+\frac 1{q'}=1$, 
 and $C_\gamma$ is as in \eqref{eq:24.07.2017}.   In particular, we have the hybrid equivariant bound
\bqn 
\norm{u}_{\L^q(M)} \leq \begin{cases} C_\gamma \,  \lambda^{\frac{\delta_{n-\kappa}(q)}{m}}, &  \frac{2(n-\kappa+1)}{n-\kappa-1} \leq q \leq \infty, \vspace{2mm} \\ C_\gamma \, \lambda^{\frac{(n-\kappa-1)(2-q')}{4m q'}}, &  2 \leq q \leq \frac{2(n-\kappa+1)}{n-\kappa-1}, \end{cases} 
\eqn
for any eigenfunction of $P$  belonging to  $u \in \E_\lambda \cap \L^2_\gamma(M)$ and satisfying $\norm{u}_{\L^2}=1$, 
provided that $G$ acts on $M$ with  orbits of the same dimension $\kappa$. Nevertheless, the $\L^p$-bounds  above cannot hold when singular orbits are present, and the situation in this case is described by Corollary \ref{cor:2.12.2015}, by which   one has the uniform bound
\bq
\label{eq:4.12.2015}
\sum_{\stackrel{\lambda_j \in (\lambda,\lambda+1],}{ e_j \in \L^2_\gamma(M)}} |e_j(x)|^2  \leq \begin{cases}  C \, \lambda^{\frac{n-1}m}, & \hspace{-.0cm} x\in M_\mathrm{sing}, \\
& \\
C_\gamma \, \lambda^{\frac{n-\kappa-1}m} \sum\limits_{N=1}^{\Lambda-1}\, \sum\limits_{{i_1<\dots< i_{N}}}   \prod\limits_{l=1}^N  |\tau_{i_l}|^{\dim G- \dim H_{i_l}-\kappa-1}, & x\in M- M_\mathrm{sing},  \end{cases}
\eq
for a constant $C_\gamma>0$ independent of $x$ and $\lambda$, and $C>0$ even independent of $\gamma$.  In comparison with the bound \eqref{eq:clust}, where the dependency of the constant $C_{x,\gamma}$ on $x$ remains unspecified, the bound \eqref{eq:4.12.2015} gives a rather precise description of the growth of eigenfunctions near singular orbits. 

To illustrate our results, consider the classical  case where $M=S^2$, and $G=\SO(2)$ acts on $M$ by rotations around  the symmetry axis through the poles. The eigenfunctions of the Laplace-Beltrami operator on $M=S^2$  are given by the spherical functions
\begin{equation*}
Y_{k,m}(\phi,\theta)=\sqrt{\frac{2k+1}{4\pi}\frac{(k-m)!}{(k+m)!}} P_{k,m}(\cos \theta)e^{im\phi}, \qquad 0\leq \phi<2\pi, \, 0 \leq \theta < \pi, 
\end{equation*}
with corresponding eigenvalues $k(k+1)$, where $k \in \N$, $|m|\leq k$,  and  $P_{k,m}$ are the associated Legendre polynomials. Furthermore, with the identification $\widehat{\SO(2)} \simeq \Z$ the spherical function $Y_{k,m}$ belongs to the isotypic component $\L^2_m(S^2)$. The Legendre polynomials $P_k(\cos \theta):=P_{k,0}(\cos \theta)$ satisfy $P_k(1)=1$, and for $k \sin \theta >1$ obey the classical asymptotics
\bqn
P_k(\cos \theta)=\sqrt{\frac 2{ \pi k \sin \theta}} \cos \left ( \left (k+\frac 12 \right ) \theta -\frac \pi 4 \right ) +O\left (\frac 1{(k \sin \theta)^{3/2}}\right ), \qquad \theta \in (0, \pi), 
\eqn
where the remainder is uniform in $\theta$ on any interval  $[\eps,\pi - \eps]$ with $0<\eps$ small, see \cite[p. 303]{hobson}. From this one concludes in the limit $k \to \infty $ that 
\bq
\label{eq:3.12.2015}
 |Y_{k,0}(\phi,\theta)|^2= {\frac{2k+1}{4\pi}} |P_{k}(\cos \theta)|^2 \approx \begin{cases} k, & \theta=0,\pi, \\
 \frac{1}{{\sin \theta}},
 & \theta \in (0,\pi). \end{cases}
\eq
 Thus, as $k \to \infty $ the  eigenfuntions $Y_{k,0}$  concentrate on the poles, which are precisely the fixed points of the $\SO(2)$-action on $S^2$, and maximize the bound \eqref{eq:1.4a}.  The bounds \eqref{eq:4.12.2015} are precisely of the type \eqref{eq:3.12.2015}, and provide an interpretation of the latter in terms  of the caustic behaviour of the equivariant spectral function, compare also Example \ref{ex:7.12.2015}. Furthermore, as discussed  in Section \ref{sec:sharpness}, the bounds \eqref{eq:3.12.2015} show that the point-wise bounds  \eqref{eq:4.12.2015} are sharp in the spectral parameter. \\

Collecting everything, the main conclusions to be drawn from this work are that 

\medskip

\begin{itemize}
\item \emph{asymptotics for the equivariant spectral function of an invariant elliptic operator are  determined by the orbit structure of the underlying group action}; \\
\item \emph{symmetries lead to refined $\L^p$-estimates for eigenfunctions of invariant elliptic operators, provided that all orbits of the underlying group action have the same dimension};\\
\item \emph{lower dimensional  orbits  are responsible for concentration of eigenfunctions, and this concentration is due to the caustic behaviour of the equivariant spectral function. In other words, the orbit structure is reflected in the shape of eigenfunctions}. 
\end{itemize}

\medskip

We would like to close this introduction by making some final comments.  In the particular case that $\gamma=\gamma_\text{triv}$ is the trivial representation,  \eqref{eq:14.05.2015} actually implies in passing a generalized Kuznecov sum formula for periods of  $G$-orbits, see  Corollary \ref{cor:kuznecov}, which generalizes previous results of Zelditch \cite{zelditch92} on periods of  closed geodesics.  In case that $G$ acts with finite isotropy groups on $M$, that is, when $\widetilde M:=M/G$ is an orbifold, an asymptotic formula for the spectral function of an elliptic operator on $\widetilde M$ was given by Stanhope and Uribe in \cite{stanhope-uribe}, and we recover their result in Corollary \ref{cor:orbifold}. If $G=\mklm{e}$, our results just reduce to the classical ones. Finally, let us mention that  one can deduce also bounds for the spectral function $e(x,y,\lambda)$ of an elliptic operator of the form
\bqn
|e(x,y,\lambda)| \leq C \cdot \lambda^{n/m},  \qquad x, y \in M, 
\eqn
by using heat-equation or, equivalently, zeta-function methods. Nevertheless, bounds of the form \eqref{eq:1.2}, which are necessary for proving the local Weyl  law \eqref{eq:1.1},  are not accessible via these techniques, and can only be obtained within the theory of Fourier integral operators, see \cite{hoermander68} and \cite[Sections 15 and 21, in particular Problem 15.1 and Lemma 21.4]{shubin}. In the equivariant case, bounds of the 
form 
\bqn
|e_\gamma (x,y,\lambda)| \leq C_\gamma \cdot \lambda^{\frac{n-\kappa}m},  \qquad x, y \in M, 
\eqn
could in principle  be deduced from work of Donnelly \cite{donnelly78} and Br\"uning-Heintze \cite{bruening-heintze79}, at least when $G$ acts on $M$ with orbits of the same dimension $\kappa$. But they would not be sufficient to imply our results, and the  desingularization techniques developed in  \cite{ramacher10} are necessary  in order  to describe the precise nature of the reduced spectral function of an invariant elliptic operator.

$\L^p$-bounds for spectral clusters for elliptic second-order differential operators on  $2$-dimensional compact manifolds with boundary and either Dirichlet or Neumann conditions were shown in \cite{smith-sogge}, while manifolds with maximal eigenfunction growth were studied in \cite{sogge-toth-zelditch}. For locally symmetric spaces of higher rank, improved $\L^p$-bounds have been shown by Sarnak and Marshall in \cite{sarnak_letter, marshall16}.   They also derived corresponding subconvex $\L^\infty$-bounds based on the presence of an additional family of commuting operators given by the Hecke algebra \cite{iwaniec-sarnak95,Marshall}. In a forthcoming article \cite{ramacher-wakatsuki17} we shall extend their results to compact arithmetic quotients of semisimple algebraic groups relying on the asymptotic description  of the integrals \eqref{eq:1.12.2015} given in Theorems \ref{thm:12.05.2015} and \ref{thm:14.05.2017}.  For a general overview on eigenfunctions on Riemannian manifolds, we refer to the survey articles \cite{zelditch08,zelditch13}. 

Through the whole document, the notation $O(\mu^{k}), k \in \R \cup \mklm{\pm \infty},$ will mean an  upper bound of the form $C \mu^k$  with a constant $C>0$ that is uniform in all relevant variables, while $O_\aleph(\mu^{k})$ will denote an upper bound of the form $C_\aleph \,  \mu^k$ with a constant $C_\aleph> 0$ that depends on the indicated variable $\aleph$. In the same way,  we shall write $a\ll_\aleph b$ for two real numbers $a$ and $b$, if there exists a constant $C_\aleph>0$ depending only on $\aleph$ such that $|a| \leq C_\aleph b$, and similarly $a \ll b$, if the bound is uniform in all relevant variables. Finally, $\N$ will denote the set of natural numbers $0,1,2,3,\dots$. \\

{\bf Acknowledgements.} I am grateful to  Simon Marshall and Christopher Sogge for conversations about the subject. Also, I would like to thank  Panagiotis Konstantis and Benjamin K\"uster  for helpful discussions concerning the proof of Lemma \ref{lem:21.04.2015}. 

\section{The reduced spectral function of an invariant elliptic operator}

Let $M$ be  a closed connected Riemannian manifold  of dimension $n$ with Riemannian volume density $dM$, and $P_0$ an elliptic classical pseudodifferential operator on $M$ 
of degree $m$ which  is positive and symmetric. Let further $T^\ast M$ be the cotangent bundle  of $M$. The  principal symbol $p(x,\xi)$ of $P_0$  constitutes a strictly positive function on $T^\ast M\setminus\mklm{0}$,  and is homogeneous in $\xi$ of degree $m$. Denote by $P$ the unique self-adjoint extension of $P_0$, its domain being the $m$-th Sobolev space $H^m(M)$, and let $\mklm{e_j}_{j\geq 0}$ be an orthonormal basis of $\L^2(M)$ consisting of eigenfunctions of $P$ with eigenvalues $\mklm{\lambda_j}_{j \geq 0}$ repeated according to their multiplicity. In order to deal with a hyperbolic problem, consider the $m$-th root $Q:=\sqrt[m]{P}$ of $P$ given by the spectral theorem. It is well  known that $Q$ is a classical pseudodifferential operator of order $1$ with principal symbol $q(x,\xi):=\sqrt[m]{p(x,\xi)}$ and domain  $H^1(M)$. Again, $Q$ has discrete spectrum, and its eigenvalues  are given by $\mu_j:=\sqrt[m]{\lambda_j}$.  The spectral function $e(x,y,\lambda)$ of $P$ can then be described by  studying the spectral function of $Q$, which in terms of the basis $\mklm{e_j}$ is given by
\bqn 
e(x,y,\mu):=\sum_{\mu_j\leq \mu} e_j(x) \overline{e_j(y)}, \qquad \mu\in \R, 
\eqn
and belongs to $\Cinft(M \times M)$ as a function of $x$ and $y$. Let $\chi_\mu$ be the spectral projection onto the sum of eigenspaces of $Q$ with eigenvalues in the interval  $(\mu, \mu+1]$, and denote its Schwartz kernel by $\chi_\mu(x,y):=e(x,y,\mu+1) - e(x,y,\mu)$. To obtain an asymptotic description of the spectral function of $Q$, one first derives a description of $\chi_\mu(x,y)$ by approximating $\chi_\mu$ by Fourier integral operators. To do so, let $\rho \in \S(\R,\R_+)$ be such that $\rho(0)=1$ and $\supp \hat \rho\in (-\delta/2,\delta/2)$ for a given $\delta>0$, and define the {approximate spectral projection operator} 
\bq
\label{eq:2.1} 
\widetilde \chi_\mu u := \sum_{j=0}^\infty \rho(\mu-\mu_j) E_{j} u, \qquad u \in \L^2(M),
\eq
where $E_j$ denotes the orthogonal projection onto the subspace spanned by $e_j$. Clearly, $K_{\widetilde \chi_\mu}(x,y):=\sum_{j=0}^\infty \rho(\mu-\mu_j) e_j(x) \overline{e_j(y)}\in \Cinft(M\times M)$ constitutes the kernel of $\widetilde \chi_\mu$. Now, notice that for $\mu,\tau\in \R$ one has
\bqn 
\rho(\mu -\tau) = \frac 1 {2\pi} \intop_\R \hat \rho(t) e^{-it\tau} e^{it\mu} \d t, 
\eqn
where $\hat \rho(t)$ denotes the Fourier transform of $\rho$, 
so that for $u \in \L^2(M)$ we obtain
\begin{align*}
\widetilde \chi_\mu u=\frac 1 {2\pi} \sum_{j=0}^\infty \intop_\R \hat \rho(t)  e^{it\mu}  e^{-it\mu_j}\d t \, E_j u= \frac 1 {2\pi}  \intop_\R \hat \rho(t)  e^{it\mu}  U(t) u \d t, 
\end{align*}
where $U(t)$ denotes the one-parameter group of unitary operators in $\L^2(M)$ 
  \bqn 
 U(t)=\int e^{-it\mu} dE_\mu^Q=e^{-itQ}, \qquad t \in \R,
 \eqn
 given by the Fourier transform of the spectral measure, 
 $\{E_\mu^Q\}$ being a spectral resolution of $Q$. The central result of H\"ormander \cite{hoermander68}  then says that  $U(t)=e^{-itQ}:\L^2(M)\rightarrow \L^2(M)$ can be approximated by Fourier integral operators, yielding an asymptotic formula for the kernels of $\widetilde \chi_\mu$ and  $\chi_\mu$, and finally for the spectral function of $Q$ and $P$. 
 
Let us now come back to the problem described in the introduction, and assume that  $M$ carries an effective and isometric action of a compact   Lie group $G$. Let $P$ commute with the left-regular representation $(\pi,\L^2(M))$ of $G$. Consider  the  Peter-Weyl decomposition \eqref{eq:PW} of $\L^2(M)$, and let $\Pi_\gamma$  be the projection onto the isotypic component belonging to $\gamma \in \widehat G$, which is given by the Bochner integral
\bqn 
\Pi_\gamma=d_\gamma \intop_G \overline{\gamma(g)} \pi(g) \d_G(g),
\eqn 
where $d_\gamma$ is the dimension of an unitary irreducible  representation of class $\gamma$, and  $d_G(g) \equiv dg$  Haar measure on $G$, which we assume to be normalized such that $\vol G=1$. If $G$ is finite, $d_G$ is simply the counting measure.  In addition, let us suppose that the orthonormal basis $\mklm{e_j}_{j\geq 0}$ is  compatible with the decomposition \eqref{eq:PW} in the sense that each vector $e_j$ is contained in some isotypic component $\L^2_\gamma(M)$. In order to describe   the spectral function of the operator $Q_\gamma:=\Pi_\gamma \circ Q\circ \Pi_\gamma=Q\circ \Pi_\gamma=\Pi_\gamma \circ Q$ given by
\bq
\label{eq:24.09.2015}
e_\gamma (x,y,\mu):=\sum_{\mu_j\leq \mu,\, e_j \in \L^2_\gamma(M)} e_j(x) \overline{e_j(y)},
\eq
we consider   the composition
\bqn 
( \chi_\mu\circ \Pi_\gamma) u = \sum_{\mu_j \in (\mu, \mu+1]} (E_j \circ \Pi_\gamma) u= \sum_{\mu_j \in (\mu, \mu+1], \, e_j \in \L^2_\gamma(M)} E_j u, \qquad u \in \L^2(M),
\eqn
with kernel 
$
K_{\chi_\mu \circ \Pi_\gamma}(x,y)=e_\gamma(x,y,\lambda+1)-e_\gamma(x,y,\lambda)
$, 
together with  the corresponding equivariant approximate spectral projection
\begin{align}
\label{eq:1004}
(\widetilde \chi_\mu \circ \Pi_\gamma) u = \sum_{j\geq 0,\, e_j \in \L^2_\gamma(M)} \rho(\mu-\mu_j) E_{j} u= \frac {d_\gamma} {2\pi} \intop_G  \intop_\R \hat \rho(t)  e^{it\mu}  \overline{\gamma(g)} \big ( U(t)\circ \pi(g) \big ) u \d t \d g.
\end{align}
Its kernel can be written as 
\bqn 
K_{\widetilde \chi_\mu \circ \Pi_\gamma}(x,y):=\sum_{j\geq 0, e_j \in \L^2_\gamma(M)} \rho(\mu-\mu_j) e_j(x) \overline{e_j(y)}\in \Cinft(M\times M).
\eqn
Put  $m_\gamma(\mu_j):=d_\gamma \text{mult}_\gamma(\mu_j)/\dim \E_{\mu_j}$, where $\text{mult}_\gamma(\mu_j)$ denotes the multiplicity of an unitary irreducible representation of class $\gamma$ in the eigenspace $\E_{\mu_j}$. 
In \cite{ramacher10},  an asymptotic formula for 
\bqn
\tr \, (\widetilde \chi_\mu \circ \Pi_\gamma) =\int_M K_{\widetilde \chi_\mu \circ \Pi_\gamma}(x,x) \d M(x)=\sum_{j=0}^\infty m_\gamma(\mu_j)  \rho(\mu-\mu_j)
\eqn
was given in order to describe the behaviour of the equivariant counting function as the eigenvalues become large, while now we are interested in the spectral function itself, which makes it necessary to derive  asymptotics for the restriction of  $K_{\widetilde \chi_\mu \circ \Pi_\gamma}$  to the diagonal, or even to a neighborhood of it, and is therefore considerably more subtle than computing the trace. 


 As mentioned before, one can  approximate $U(t)$ by means of Fourier integral operators. More precisely, let $\mklm{(\kappa_\iota, Y_\iota)}_{\iota \in I}$, $\kappa_\iota:Y_\iota \stackrel{\simeq}\to \widetilde Y_\iota \subset \R^n$, be an atlas for $M$,  $\mklm{f_\iota}$ a corresponding partition of unity and  
$
\hat v (\eta) :=\F(v)(\eta):= \int_{\R^n} e^{-i\langle \tilde y, \eta \rangle} v(\tilde y) \, d\tilde y$ the Fourier transform of $v \in \CT(\widetilde Y_\iota)$. Write $\dbar \eta:= \d\eta/(2\pi)^n$, and introduce   the operator 
\bqn 
[\widetilde U_\iota(t)v] (\tilde x):= \int _{\R^n} e ^{i\psi_\iota(t,\tilde  x, \eta)} a_\iota (t, \tilde x, \eta) \hat{v}(\eta) \dbar \eta
\eqn 
on $\widetilde Y_\iota$, where $a_\iota \in S^0_{\mathrm{phg}}$ is a classical polyhomogeneous symbol satisfying $a_\iota(0,\tilde x, \eta)=1$ and $\psi_\iota$  the defining phase function given as the solution of the Hamilton-Jacobi equation
\bqn 
\frac {\gd \psi_\iota} { \gd t } + q \Big (x, \frac{\gd \psi_\iota}{\gd \tilde x}\Big )=0 , \qquad  \psi_\iota( 0, \tilde x, \eta) = \eklm{\tilde x, \eta},
\eqn  
see \cite[ p. 254]{hoermanderIV}.   Let us remark that  $\psi_\iota$ is homogeneous in $\eta$ of degree $1$, so that Taylor expansion for small $t$ gives
\bq
\label{eq:19.04}
\psi_\iota(t,\tilde x, \eta) =\psi_\iota(0,\tilde x, \eta) +t \frac{\gd \psi_\iota}{\gd t} (0, \tilde x , \eta) + O(t^2|\eta|)= \eklm{\tilde x, \eta} -t q_\iota(\tilde x, \eta) +O(t^2|\eta|), 
\eq 
where we wrote $q_\iota(\tilde x, \eta):=q(\kappa_\iota^{-1}(\tilde x),\eta)$. In other words, there exists a smooth function $\zeta_\iota$ which is homogeneous in $\eta$ of degree $1$ and satisfies
\begin{align}
\label{eq:24.12.2015}
\begin{split}
\psi_\iota(t, \tilde x, \eta) &= \eklm {\tilde x, \eta} -t \zeta_\iota(t, \tilde x, \eta), \qquad \qquad \zeta_\iota(0, \tilde x, \eta) = q_\iota(\tilde x, \eta).
\end{split}
\end{align}
Let now $\bar U_\iota(t) u := [ \widetilde U_\iota(t) (u \circ \kappa_\iota^{-1})] \circ \kappa_\iota$, $ u \in \CT( Y_\iota)$. Consider further  test functions $\bar f_\iota \in \CT( Y_\iota)$ satisfying $\bar f_\iota \equiv 1$ on $\supp f_\iota$, and define
\bqn 
\bar U(t) := \sum _\iota F_\iota \, \bar U_\iota(t) \, \bar F_\iota, 
\eqn
where  $F_\iota$, $\bar F_\iota$ denote the multiplication operators corresponding to  $f_\iota$ and $\bar f_\iota$, respectively. Then  H\"ormander showed that for small $|t|$ 
\bq
\label{eq:R(t)}
R(t) := U(t) -\bar U(t) \, \text{is an operator with smooth kernel,}
\eq 
compare \cite[ p. 134]{grigis-sjoestrand} and \cite[Theorem 20.1]{shubin}; in particular, the kernel $R_t(x,y)$ of $R(t)$ is smooth in $t$.  We now have the following

\begin{proposition}
\label{prop:0}
Let $\delta>0$ be sufficiently small and $x,y \in M$. Then, as $\mu \to +\infty$, 
\begin{gather*}
K_{\widetilde \chi_\mu \circ \Pi_\gamma}(x,y)
= \frac {\mu^n  d_\gamma}{(2\pi)^{n+1}} \sum _\iota  \int_{-\infty}^{+\infty}   \int _G  \int_{\R^n} e^{i\mu t[1-\zeta_\iota(t,\kappa_\iota(x),\eta)]} e^{i{ \mu} \eklm{\kappa_\iota( x) - \kappa_\iota( g \cdot y),\eta}}  \hat \rho(t)  \overline{\gamma(g)} f_\iota( x) \\  \cdot  a_\iota(t, \kappa_\iota(x) , \mu \eta)  \bar f _\iota (g \cdot y)   \alpha( q(  x, \eta)) J_\iota(g,y) d\eta \d g \d t, 
\end{gather*}
up to terms of order  $O(\mu^{-\infty})$ which are uniform in $x$ and $y$, where  $0 \leq \alpha \in \CT(1/2, 3/2)$ is a test function such that $\alpha \equiv 1$ in a neighborhood of $1$, $J_\iota(g,y)$ is a Jacobian, and   
$d\eta$ denotes Lebesgue measure on $\R^n$. 
On the other hand,  $K_{\widetilde \chi_\mu \circ \Pi_\gamma}(x,y)$ is rapidly decaying as  $\mu \to -\infty$.
\end{proposition}
\begin{proof}
To obtain an explicit espression for the kernel of $\widetilde \chi_\mu \circ \Pi_\gamma$ let $u \in \Cinft(M)$, and notice that 
\begin{align*}
F_\iota  \bar U_\iota(t) &\bar F_\iota u (x) = f_\iota(x) [ \widetilde U_\iota(t) ( \bar f_\iota u \circ \kappa^{-1}_\iota)] \circ \kappa_\iota(x)\\ 
& = f_\iota(x) \int_{\R^n} e^{i\psi_\iota ( t, \kappa_\iota(x),\eta)} a_\iota(t, \kappa_\iota(x), \eta) \widehat{ (\bar f_\iota u \circ \kappa^{-1}_\iota)} (\eta) \dbar \eta \\
&=\int_{\widetilde Y_\iota} \int_{\R^n} f_\iota(x) e^{i[\psi_\iota(t,\kappa_\iota(x), \eta)- \eklm{ \tilde y,\eta}]} a_\iota(t, \kappa_\iota(x), \eta)  (\bar f_\iota u) (\kappa_\iota^{-1} ( \tilde y ))  d\tilde y \, \dbar \eta \\
&= \int _{Y_\iota} \bigg [  \int _{\R^n}   e^{i[\psi_\iota(t,\kappa_\iota(x), \eta)- \eklm{ \kappa_\iota(y),\eta}]} a_\iota(t, \kappa_\iota(x), \eta) \, \dbar \eta \,  f_\iota(x) \,  \bar f _\iota(y) ( \beta^{-1}_\iota \circ \kappa_\iota)(y)  \bigg ] u(y) \,  dM(y),
\end{align*}
where we wrote $ (\kappa^{-1}_\iota)^\ast (dM)=\beta_\iota d \tilde y$.  The last two expressions are oscillatory integrals with suitable regularizations.
With \eqref{eq:1004} and \eqref{eq:R(t)}  we therefore obtain for $K_{\widetilde \chi_\mu \circ \Pi_\gamma}(x,y) $ the expression
\begin{gather*}
 \frac{d_\gamma}{(2\pi)^{n+1}} \sum _\iota  \int_{-\infty}^{+\infty}  \int _G  \int_{\R^n} \hat \rho(t)  e^{it\mu}\overline{\gamma(g)} f_\iota( x) e^{i[\psi_\iota(t,\kappa_\iota(x), \eta)- \eklm{ \kappa_\iota(g \cdot y),\eta}]} a_\iota(t, \kappa_\iota(x), \eta) \\ \cdot  \bar f _\iota(g \cdot y)  J_\iota(g,y) d\eta \d g \d t + O(|\mu|^{-\infty}),
 \end{gather*}
 since  
 \bqn 
 \frac 1 {2\pi} \int_G\int_{-\infty}^{+\infty} \hat \rho(t) e^{it\mu}  R_t( x, g \cdot y) \d t \, \overline{\gamma(g)} J_\iota(g,y) \d g= \int_G  \F^{-1}\big (\hat \rho (\bullet)  R_\bullet (x, g \cdot y) \big ) (\mu) \, \overline{\gamma(g)} J_\iota(g,y) \d g,
 \eqn
and $\F^{-1}\big (\hat \rho (\bullet)  R_\bullet ( x, g \cdot y)\big )$ is rapidly falling in $\mu$;  in particular, $O(|\mu|^{-\infty})$ is uniform in $x,y$.  
We are interested in the asymptotic behaviour of $K_{\widetilde \chi_\mu \circ \Pi_\gamma}(x,y) $ as $\mu \to \pm \infty$. In order to study it by means of the stationary phase theorem, we  define
\bqn 
\G(\tau, \tilde x, \eta):= \int_{-\infty}^{+\infty} e^{it \tau} \hat \rho(t) a_\iota(t, \tilde x, \eta)e^{iO(t^2 | \eta|) }  dt,
\eqn 
where $O(t^2 | \eta|)$ denotes the remainder in \eqref{eq:19.04}. Clearly, $\G(\tau, \tilde x, \eta)$  is rapidly decaying as a function in $\tau$. On the other hand, there must exist a constant $C>0$ such that 
\bqn 
 C |\eta| \geq q_\iota( \tilde x, \eta) \geq \frac 1 C |\eta| \qquad \forall \tilde x \in \widetilde Y_\iota, \, \eta \in \R^n,
\eqn 
which implies that for fixed $\mu$, $\G(\mu-q_\iota(\tilde x, \eta), \tilde x, \eta)$ is rapidly decaying in $\eta$. This yields a regularization of the oscillatory integral above, and we obtain
\begin{align*}
K_{\widetilde \chi_\mu \circ \Pi_\gamma}(x,y) 
=&  \frac{d_\gamma}{(2\pi)^{n+1}} \sum _\iota    \int _G  \int_{\R^n } e^{i\eklm{\kappa_\iota(x)-\kappa_\iota(g \cdot  y),\eta}}   \overline{\gamma(g)} f_\iota( x) \\  & \cdot \G(\mu -q( x, \eta), \kappa_\iota( x), \eta)   \bar f _\iota (g \cdot y)  J_\iota(g,y) d \eta \d g  + O(|\mu|^{-\infty}).
\end{align*}
 But even more is true. $K_{\widetilde \chi_\mu \circ \Pi_\gamma}(x,y)$  is rapidly decreasing as $\mu \to -\infty$, reflecting the positivity of the spectrum.  Furthermore, assume that $|1- q_\iota (\tilde x, \eta/\mu)| \geq \text{const} >0$. Then
\begin{align*}
|\G(\mu-q_\iota( \tilde x, \eta), \tilde x, \eta)| &\leq C_{N+M} \frac 1 {|\mu|^{N}} \frac 1 { |1 - q_\iota( \tilde x, \eta /\mu)|^N} \frac 1 {| \mu - q_\iota( \tilde x, \eta)|^M}\\
& \leq C'_{N+M} \frac 1 {|\mu|^{N}}  \frac 1 {| \mu- q_\iota( \tilde x, \eta)|^M}
\end{align*}
for arbitrary $N, M \in \N$ and suitable constants. Let therefore $ \alpha \in \CT(1/2, 3/2)$ be as indicated, so that 
\bqn 
1 - \alpha(q_\iota( \tilde x , \eta/\mu)) \not= 0 \quad \Longrightarrow \quad |1-q_\iota( \tilde x, \eta/\mu)| \geq  C >0
\eqn  
for a constant depending only on $\alpha$. 
Substituting $\eta=\mu \eta'$,  we can  re-write  $K_{\widetilde \chi_\mu \circ \Pi_\gamma}(x,y)$ as 
\begin{align*}
K_{\widetilde \chi_\mu \circ \Pi_\gamma}(x,y)
= &\frac {|\mu|^n  d_\gamma}{(2\pi)^{n+1}} \sum _\iota  \int_{-\delta/2}^{+\delta/2}  \int _G  \int_{\R^n} e^{i\mu \big [\psi_\iota(t,\kappa_\iota( x),\eta)- \eklm{ \kappa_\iota(g \cdot y),\eta}+t \big ]}  \hat \rho(t)  \overline{\gamma(g)} f_\iota(x)   \\ & \cdot a_\iota(t, \kappa_\iota( x) , \mu \eta)  \bar f _\iota (g \cdot y)  \alpha( q( x, \eta)) J_\iota(g,y) \d\eta \d g \d t + O(|\mu|^{-\infty}),
\end{align*}
where all integrals are absolutely convergent, and the remainder is uniform in $x,y$. The proposition now follows with \eqref{eq:24.12.2015}.
\end{proof}

Since $\zeta_\iota(0, \tilde x , \eta)= q_\iota(\tilde x, \eta)$, there exists a constant $C>0$ such that for sufficiently small $t \in (-\delta/2, \delta/2)$
\bqn 
 C |\eta| \geq \zeta_\iota(t,  \tilde x, \eta) \geq \frac 1 C |\eta| \qquad \forall \tilde x \in \widetilde Y_\iota, \, \eta \in \R^n.
\eqn 
We can therefore introduce  in $\R^n\setminus\mklm{0}$ the coordinates
$$
\eta = R\, \omega_1, \qquad  R>0, \qquad \zeta_\iota(t,\kappa_\iota( x), \omega_1)=1.
$$
Indeed, since  $\zeta_\iota(t,\kappa_\iota(x),\eta)$ is homogeneous of degree $1$ in $\eta$,  its derivative in radial  direction  reads  
\bqn 
\lim_{s \to 0} s^{-1} (R+s-R)  \zeta_\iota(t,\kappa_\iota(x),\omega_1) =1,
\eqn 
so that for all $\eta=R \, \omega_1$ we have 
\bq
\label{eq:07.04.2016}
\eklm{\grad_\eta \zeta_\iota(t, \tilde x , \eta), \eta}=R>0.
\eq
Consequently, the Jacobian of the coordinate change $\eta=R\, \omega_1$ does not vanish. Re-writing  the expression for the kernel of $\widetilde \chi_\mu \circ \Pi_\gamma$ in Proposition \ref{prop:0} in terms of the coordinates $\eta =R \, \omega_1$ we obtain 
\begin{align}
\label{eq:25.12.2015}
\begin{split}
K_{\widetilde \chi_\mu \circ \Pi_\gamma}(x,y)
=& \frac {\mu^n  d_\gamma}{(2\pi)^{n+1}} \sum _\iota  \int_{\R}\int _{\R} e^{i\mu(t-Rt)}  \int _G  \int_{\Sigma^{R,t}_{\iota,x} } e^{i \mu \eklm{\kappa_\iota( x) - \kappa_\iota( g \cdot y),\omega}}  \hat \rho(t)  \overline{\gamma(g)} f_\iota( x) \\   & \cdot a_\iota(t, \kappa_\iota(x) , \mu  \omega)  \bar f _\iota (g \cdot y)   \alpha( q(  x, \omega)) J_\iota(g,y) {\d\Sigma^{R,t}_{\iota,x}( \omega) \d g}  \d R\d t
\end{split}
\end{align}
up to terms of order  $O(\mu^{-\infty})$ which are uniform in $x$ and $y$, where  we set   
\bq
\label{eq:20.04.2015}
\Sigma^{R,t}_{\iota,x}:=\mklm{\omega \in \R^n \mid  \zeta_\iota (t, \kappa_\iota (x),\omega) = R}.
\eq
 Here ${d\Sigma^{R,t}_{\iota,x}(\omega)}$ denotes the quotient of Lebesgue measure  in $\R^n$ by Lebesgue measure in $\R$ with respect to $\zeta_\iota(t, \tilde x, \omega)$.   Note  that for sufficiently small $\delta>0$ we can assume that the $R$-integration is over a compact set. Furthermore, $R$ and $t$ are close to $1$ and $0$, respectively. To describe the asymptotic behaviour of  $K_{\widetilde \chi_\mu \circ \Pi_\gamma}(x,y)$  as $ \mu \to +\infty$, we shall now first apply the stationary phase theorem to the integral over $R$ and $t$, and then to the integral over $G\times \Sigma^{R,t}_{\iota,x}$.
\begin{cor}
\label{cor:12.05.2015}
Let  $\mu\geq 1$, $ x,y \in M$, and  with the notation of Proposition \ref{prop:0}  set 
\begin{align}
\begin{split}
\label{eq:02.05.2015}
 I^\gamma_\iota(\mu, R, t, x,y):= &   \int _G  \int_{\Sigma^{R,t}_{\iota,x}} e^{i{ \mu}  \Phi_{\iota,x,y}(\omega,g)} \hat \rho(t)  \overline{\gamma(g)} f_\iota( x) \\ &\cdot     a_\iota(t, \kappa_\iota(  x) , \mu \omega)  \bar f _\iota (g \cdot y)   \alpha(q( x, \omega)) J_\iota(g,y) {\d\Sigma^{R,t}_{\iota,x}(\omega) \d g},
 \end{split}
\end{align}
where $\Phi_{\iota,x,y}(\omega,g):=\eklm{\kappa_\iota(  x) - \kappa_\iota(g \cdot  y),\omega}$. Then, for each $\tilde N \in \N$
\begin{align}
\label{eq:13.06.2016}
K_{\widetilde \chi_\mu \circ \Pi_\gamma}(x,y)
= &\Big(\frac{\mu}{2\pi}\Big )^{n-1}  \frac{d_\gamma}{2\pi} \sum _\iota \Big [   \sum_{j=0}^{\tilde N-1} D^{2j}_{R,t} I^\gamma_\iota(\mu, R, t,x,y)_{|(R,t)=(1,0)} \, \mu^{-j} +  \mathcal{R}^\gamma_{\iota}(\mu,x,y) \Big ]
\end{align}
up to terms of order $O(\mu^{-\infty})$ which are uniform in x,y, where $D^{2j}_{R,t}$ are known differential operators of order $2j$ in $R,t$, and 
\begin{align*}
|\mathcal{R}^\gamma_{\iota}(\mu,x,y)| \leq& C\mu^{-\tilde N} \sum_{|\beta| \leq 2\tilde N +3} \sup_{R,t} \big |\gd_{R,t}^\beta  I^\gamma_\iota(\mu,R,t,x,y)  \big |
\end{align*}
for some constant $C>0$. 
\end{cor}
\begin{proof}
Since $(R,t)=(1,0)$ is the only critical point of $t-Rt$, the assertion follows immediately from \eqref{eq:25.12.2015} and the classical stationary phase theorem \cite[Proposition 2.3]{grigis-sjoestrand}. 
\end{proof}

Thus, we are left with the task of describing the asymptotics of the oscillatory integrals $ I^\gamma_\iota(\mu, R, t, x,y)$ as $\mu \to +\infty$, which will occupy us in the next sections. 

\section{Equivariant asymptotics of oscillatory integrals}

Let the notation be as in the previous section. As we have seen there, the question of describing the spectral function  in the equivariant setting reduces to the study of   oscillatory integrals of the form 
\bq
\label{eq:03.05.2015}
I_{x,y}(\mu):=\int_{G}\int_{\Sigma^{R,t}_x} e^{i\mu \Phi_{x,y}(\omega,g)} a (x,y, \omega,g) \d \Sigma^{R,t}_x (\omega) \d g, \qquad  \mu \to + \infty,
\eq
with $\Sigma^{R,t}_x$ as in \eqref{eq:20.04.2015} and phase function 
\bqn
\Phi_{x,y}(\omega,g):=  \eklm{\kappa(x) - \kappa( g\cdot y), \omega},
\eqn 
where we have skipped the index $\iota$ for simplicity of notation, and  $a \in \CT$ is an amplitude that might depend on $\mu$  and other parameters such that $(x,y,\omega, g) \in \supp a$ implies $x, g \cdot y \in Y$. In what follows, we shall also write 
 \bq
 \label{eq:11.9.2017} 
 I_x(\mu) :=I_{x,x}(\mu), \qquad \Phi_x:=\Phi_{x,x}.
 \eq
The asymptotic behaviour of these integrals is related to that  of  oscillatory integrals of the form
\bq
\label{eq:integral} 
I(\mu)= \intop_G\intop_{T^\ast Y} e^{i\mu \Phi(x,\eta,g)} a (x,\eta,g) \d (T^\ast Y)(x,\eta) \d g, \qquad \mu \to + \infty,
\eq
with  phase function
\bq
\label{eq:phase}
\Phi(x,\eta,g):= \eklm{\kappa(  x) - \kappa(g \cdot  x),\eta},
\eq
and suitable amplitude $a$. Let us assume in the following that $G$ is a continuous group. Asymptotics for the integrals \eqref{eq:integral} were given  in \cite{ramacher10} using the stationary phase principle, and we will rely on these results in parts to perform a similar analysis for the integrals $I_{x,y}(\mu)$. Write $\kappa(x)=(\tilde x_1, \dots, \tilde x_n)$ so that the canonical local trivialization of $T^\ast Y$ reads
\bqn 
Y \times \R^n \, \ni (x,\eta) \quad \equiv \quad \sum_{k=1}^n \eta_k (d\tilde x_k)_{x} \in \, T^\ast_xY.
\eqn
With respect to this trivialization, we shall identify  $\Sigma^{R,t}_{x'}$ with a subset in $T^\ast_{x} Y$ for eventually different $x$ and $x'$, if convenient. 
Let $\Omega:=\Jbb^{-1}(\mklm{0})$ be the zero level set of the momentum map $\Jbb: T^\ast M \to \g^\ast$ of the underlying Hamiltonian $G$-action on $T^\ast M$. Since
 \bq
\label{eq:Ann}
(x,\eta) \in \Omega  \cap T^\ast _xM \quad \Longleftrightarrow \quad (x,\eta)  \in \mathrm{Ann}(T_x (G\cdot x)),
\eq
where $\mathrm{Ann} \, (V_x) \subset T_x^\ast M$ denotes the annihilator of a vector subspace $V_x \subset T_xM$,  a simple computation shows that the critical set of $\Phi $ is  given by  
\bq
\label{eq:23.04.2015}
\Crit \, \Phi=\mklm{(x,\eta,g) \in T^\ast Y \times G\mid d(\Phi)_{(x,\eta,g)}=0} =\mklm{(x,\eta,g) \in ( \Omega \cap T^\ast Y) \times G\mid g \in G_{(x,\eta)}}.
\eq
In what follows, we shall compute  the critical set of the phase function $\Phi_{x,y}$, which is much more involved.  Let $\O_x:=G\cdot x$ denote the $G$-orbit and $G_x:=\mklm{g \in G\mid g\cdot x=x}$ the stabilizer or isotropy group of a point  $x\in M$. Throughout the paper, we shall assume that 
\bqn 
 \dim \O_x \leq n-1 \qquad \text{for all } x \in M.
\eqn
Let further $N_y\O_x$ be the normal space to the orbit $\O_x$ at a point $y \in \O_x$, which can be identified with  $\mathrm{Ann}(T_y\O_x)$ via the underlying Riemannian metric. With $M_\text{prin}$, $M_\text{except}$, and $M_\text{sing}$ as in \eqref{eq:15.08.2016} we  now have the following

\begin{lem}\label{lem:21.04.2015}
Let $x\in Y$, $\O_y \cap Y \not=\emptyset$,  and 
\bqn 
\Crit \, \Phi_{x,y}:=\Big \{(\omega,g) \in \Sigma^{R,t}_x  \times \mklm{g \in G\mid g \cdot y \in Y}\mid \, d(\Phi_{x,y})_{(\omega,g)}=0\Big \}
\eqn
be the critical set of $\Phi_{x,y}$. 
\begin{enumerate}
\item[(a)] If  $y\in \mathcal{O}_x$, the set $\Crit \, \Phi_{x,y}$ is clean and given by the smooth submanifold
\bqn
\Ccal_{x,y}:=\big \{(\omega,g)\mid (g \cdot y, \omega) \in \Omega,  \,  x=g\cdot y\big \}
\eqn
of  codimension $2\dim \O_x$.

\item[(b)] If $y \not\in  \mathcal{O}_x$, 
$$\Crit \, \Phi_{x,y}=\Big \{(\omega,g)\mid (g \cdot y,\omega) \in \Omega, \,  \kappa(x)-\kappa(g \cdot y) \in N_\omega \Sigma^{R,t}_x\Big \};$$
furthermore, assume that $G$ acts on $M$ with orbits of the same dimension $\kappa$, that is, $M=M_\mathrm{prin}\,  \cup\,  M_\mathrm{except}$, and that the co-spheres $S_x^\ast M$ are strictly convex. Then, either $\Crit \, \Phi_{x,y}$ is empty, or,  choosing $Y$ sufficiently small,  $\Crit \, \Phi_{x,y}$ is locally diffeomorphic to $G_y$,  clean, and of codimension $n-1+\kappa$. 

\item[(c)] In  case that $x\in Y\cap M_\mathrm{prin}$ one has
\bqn
\Ccal_{x,x}=\mathrm{Crit} \, \Phi \cap (\Sigma^{R,t}_x\times G),
\eqn
a transversal intersection. In particular $\Ccal_{x,x}$ is a smooth submanifold of codimension $2\kappa$. 
\end{enumerate}
\end{lem}
\begin{proof}
Consider a local parametrization 
\bq
\label{eq:paramhyp}
F:\R^{n-1} \supset W \, \longrightarrow \, \Sigma_x^{R,t}\subset \R^n , \quad \alpha \longmapsto F(\alpha)=\omega,
\eq
of the hyerpsurface $\Sigma_x^{R,t}$, where $W$ denotes an open subset. Differentiating $\Phi_{x,y}$ with respect to $\alpha$ and setting the derivatives to zero gives the conditions $\eklm{\kappa(x)-\kappa(g \cdot y), \gd F/\gd \alpha_i}=0$ for $i=1,\dots, n-1$, implying  that $\kappa(x)-\kappa(g \cdot y)$ must be normal to $\Sigma^{R,t}_x$ at $\omega$. On the other side, the derivatives of  $\Phi_{x,y}$ with respect to $g$ read $\sum_{k=1}^n \omega_k(d\tilde x_k)_{g\cdot y} (\widetilde X_j)$, where $\mklm{X_1,\dots, X_d}$ denotes a basis of $\g$ and $\big \{\widetilde X_1,\dots, \widetilde X_d\big \}$ are  the corresponding fundamental vector fields on $M$. Setting them to zero   yields $(g\cdot y,\omega) \in T^\ast_{g\cdot y} Y \, \cap \, \Omega \simeq N_{g \cdot y}\O_y$, and we conclude that
\bq
\label{eq:21.05.2016}
\Crit \, \Phi_{x,y}=\Big \{(\omega,g)\mid (g \cdot y,\omega) \in \Omega, \,  \kappa(x)-\kappa(g \cdot y) \in N_\omega \Sigma^{R,t}_x\Big \}.
\eq
The second condition means that  $\kappa(x)-\kappa(g \cdot y) $ is co-linear to $\grad_\eta \zeta(t,\kappa(x), \omega)$. But in view of \eqref{eq:07.04.2016} we have the equality
\bq
\label{eq:07.04.2016a}
\eklm{\grad_\eta \zeta(t, \tilde x , \omega), \omega}=R>0, \qquad  \omega \in \Sigma^{R,t}_x,
\eq
so that if  $x\not=g \cdot y$ and $\kappa(x)-\kappa(g \cdot y) \in N_\omega\Sigma_x^{R,t}$, we deduce the lower bound 
\bq
\label{eq:22.05.2016}
\bigg |\eklm{\frac{ \kappa(x)-\kappa(g \cdot y)}{\norm{ \kappa(x)-\kappa(g \cdot y)}}, \omega}\bigg | \geq C >0
\eq
for a uniform constant $C>0$.  Since the $G$-action on $M$ is smooth,  there is an invariant tubular neighbourhood  around each $G$-orbit in $M$, and  we may assume that the chart $(\kappa,Y)$ is given in terms of such a neighbourhood around $\O_x$. Thus, let $N\O_x$ be the normal bundle to $\O_x$, and 
\bqn 
\tau:=\exp \circ \, \gamma:N \O_x \longrightarrow M
\eqn 
an equivariant diffeomorphism onto some open neighborhood of $\O_x$, where $\exp$ denotes the exponential map and  $\gamma:N\O_x \rightarrow N\O_x$ is certain contraction  \cite[Theorem VI.2.2]{bredon}. In particular, note that
\bqn 
d(\exp)_z:T_z(N\O_x) \equiv N_z\O_x \oplus T_z\O_x=T_zM \longrightarrow T_zM, \qquad z \in \O_x,
\eqn
is the identity, where $\O_x$ is embedded as the zero section in $N\O_x$.  If we now let $Y \subset \tau(NO_x)$ be small enough, we can  identify $\tau^{-1}(Y)$ with an open neighbourhood of the origin in $T_{x}(N\O_x)$ via the exponential map, and  we put $\kappa:=(\tau^{-1})_{|Y}$.  

To show $(a)$, let us assume that  $y \in \O_x$.  Then the vector $\kappa(x)-\kappa(g \cdot y)$  must  be  approximately normal to $(d\kappa)_x(N_{x} \O_x)\equiv N_{x}\O_x$ for sufficiently small $Y$, so that if $(\omega,g) \in \Crit \,  \Phi_{x,y}$, which in particular means that   $\omega\in N_{g\cdot y} \O_x$, the vector $\kappa(x)-\kappa(g \cdot y)$  must  be  approximately normal to $\omega$, which would be a contradiction to the lower bound  \eqref{eq:22.05.2016}, unless $x=g\cdot y$. Thus, we conclude that  $ \Crit \,  \Phi_{x,y} =\Ccal_{x,y}$. In order to see  that $\Crit \,  \Phi_{x,y}$ is clean, note that with respect to the parametrization  \eqref{eq:paramhyp} of $\Sigma^{R,t}_x$ and canonical coordinates on $G$ the Hessian $\mathrm{Hess} \, \Phi_{x,y}(\omega,g)$  of $\Phi_{x,y}$ at a critical point $(\omega,g) \in \Ccal_{x,y}$, as a symmetric bilinear form on $T_\omega \Sigma^{R,t}_x\times T_gG$, is  given by the matrix
\bq
\label{eq:27.05.2016bis}
\left ( \begin{array}{c|c}  \vspace{3mm} 0 & \sum_{k=1}^n \frac{\gd F_k (\alpha^{-1}(\omega)) }{\gd \alpha_i} (d\tilde x_k)_{x}(\widetilde X_j) \\ \hline \\ \sum_{k=1}^n \frac{\gd F_k (\alpha^{-1}(\omega)) }{\gd \alpha_j}(d\tilde x_k)_{x}(\widetilde X_i) &-\frac 12 \eklm{\widetilde X_{i,x} (\widetilde X_j (\kappa))+\widetilde X_{j,x} (\widetilde X_i (\kappa)),\omega}  \end{array} \right  ).
\eq
The kernel of the corresponding linear transformation is given by those $(\tilde \alpha, \tilde s)\in \R^{n-1} \times \R^d$ satisfying the conditions 
\begin{align}
\label{eq:07.04.2016b}  \sum_{k} \frac{\gd F_k (\alpha^{-1}(\omega)) }{\gd \alpha_i}  (d\tilde x_k)_{x}\big(\widetilde X(\tilde s) \big ) =0 \qquad &\text{ for all } i=1,\dots, n-1,  \\ 
\label{eq:07.04.2016c} \sum_{j,k} \tilde \alpha _j  \frac{\gd F_k (\alpha^{-1}(\omega)) }{\gd \alpha_j}  (d\tilde x_k)_{x}(\widetilde X_i)=0 \qquad &\text{ for all } \, i=1,\dots,d, 
 \end{align}
 where we put $X(\tilde s):=\sum_{j=1}^d \tilde s_j  X_j$.  Indeed, \eqref{eq:07.04.2016b} implies that  
 $$(d\kappa)_x\big(\widetilde X(\tilde s)\big )=\Big ( (d\tilde x_1)_x \big(\widetilde X(\tilde s)\big ), \dots , (d\tilde x_n)_x \big(\widetilde X(\tilde s)\big ) \Big )\in (d\kappa)_x(T_x\O_x)$$ 
 is co-linear to  $\grad_\eta \zeta(t,\kappa_\iota(x), \omega)$. Furthermore, since $(d\kappa)_x(T_x\O_x)\equiv T_x\O_x$ is clearly normal to $(d\kappa)_x(N_x\O_x)\equiv N_x\O_x$, the vector $(d\kappa)_x\big(\widetilde X(\tilde s)\big )$ is normal to $N_x\O_x$.  In view of \eqref{eq:07.04.2016a} and the fact that  $(x,\omega) \in N_x \O_x$ we would obtain a contradiction, unless   $\widetilde X(\tilde s)$ vanishes at   $x$; in particular, this implies that  $\langle  \widetilde X (\tilde s)(\kappa),\omega\rangle$ has a zero of second order at $x$ in orbit direction, so that 
 \bqn 
\widetilde X_{i,x}  \eklm{ \widetilde X(\tilde s) (\kappa),\omega} +\widetilde X(\tilde s)_{x} \eklm{\widetilde X_i (\kappa),\omega}=0.
 \eqn
 Thus, the coefficients in the fourth quadrant of the matrix \eqref{eq:27.05.2016bis} do not contribute to Equations \eqref{eq:07.04.2016c}, and the kernel in question is given by  
 \begin{gather*}  
 \mklm{(\tilde \alpha, \tilde s)\in \R^{n-1} \times \R^d\mid  \sum_{j=1}^d \tilde s_j \widetilde X_{j,x}=0, \, \sum_{j,k} \tilde \alpha _j  \frac{\gd F_k (\alpha^{-1}(\omega)) }{\gd \alpha_j}  (d\tilde x_k)_{x} \in \text{Ann } ( T_{x}\O_x )  }\simeq T_{(\omega,g)} \Ccal_{x,y},
\end{gather*}
which means that $\mathrm{Hess} \, \Phi_{x,y}$ is transversally non-degenerate on $\Ccal_{x,y}$, yielding  (a).

In order to see (b), assume that $y \notin \O_x$, and let the chart $(\kappa,Y)$ be defined as above in terms of the tubular neighbourhood $\tau:N\O_x \rightarrow M$.  Note that without loss of generality we can assume that $y \in S_x\cap Y$, where $S_z:=\tau(N_z\O_x)$.  The first part of (b) is clear from \eqref{eq:21.05.2016}. Now, assume that  the co-spheres $S_x^\ast M$ are strictly convex. For small $|t|$, the hypersurfaces $\Sigma^{R,t}_x$ will  be strictly convex, too. In particular, $\Sigma^{R,t}_x$ is orientable, and the Gauss map 
\bqn 
\Ncal: \Sigma_x^{R,t} \ni \omega \longmapsto \Ncal(\omega) \in N_\omega \Sigma_x^{R,t},
\eqn
which assigns to each point of $\Sigma^{R,t}_x$   the {outer} normal unit vector to $ \Sigma^{R,t}_x$ at that point,  is a global diffeomorphism. Therefore, for each $x \not= \tilde y \in Y$ there is a unique $\omega_{\tilde y} \in \Sigma_x^{R,t}$ such that
\bqn 
\frac{ \kappa(\tilde y)-\kappa(x) }{\norm{ \kappa(\tilde y)-\kappa(x)}}=\Ncal(\omega_{\tilde y}).
\eqn
Consequently, if $(\omega,g) \in \Crit \, \Phi_{x,y}$, the vector $\omega$ is locally uniquely determined by the condition $\Ncal(\omega) =\pm \, \Ncal(\omega_{g\cdot y})$. 
Now, introduce the sets
\bqn 
W_n:= \tau(V_{1/n}), \qquad V_{1/n}:= \mklm{v \in N\O_x\mid \norm{v} < 1/n}, \qquad n \in \N, 
\eqn
and  assume that for each $n \in \N$ there is a $y_n \in W_n \cap Y\cap S_x$ such that $\Crit \, \Phi_{x,y_n}$ is not empty, but $\Crit \, \Phi_{x,y_n}\not \simeq G_{y_n}$ locally. In other words, assume that for each $n \in \N$ there is a smooth curve 
$$
\gamma_n:(-\eps_n,\eps_n) \ni t \longmapsto  (\omega_n(t), g_n(t))\in \Crit \, \Phi_{x,y_n}, \qquad \eps_n >0, 
$$ parametrized such that $\norm{\dot \omega_n(t)}=1.$ In this way, we obtain for each $n \in \N$ a curve $\omega_n(t)$ in $\Sigma_x^{R,t}$ along which the unit normal vector field to $\Sigma_x^{R,t}$ is determined by the direction of $\kappa(x)-\kappa(g_n(t) \cdot y_n)$, so that $\Ncal(\omega_n(t)) = \pm \, \Ncal(\omega_{g_n(t) \cdot y_n})$. In view of \eqref{eq:22.05.2016},  the curves
\bqn 
\mklm{g_n(t) \cdot y_n\mid \, t \in (-\eps_n,\eps_n)}\subset Y  
\eqn
converge to $x$ as  $n \to \infty$, which in particular implies that $\eps_n \to 0$.  Similarly, due to the compactness of $\Sigma_x^{R,t}$ the curves 
\bqn 
\mklm{\omega_n(t)\mid \, t \in (-\eps_n,\eps_n)}\subset \Sigma_x^{R,t}
\eqn
converge to at least one $\omega_\infty \in \Sigma_x^{R,t}\cap N_x \O_x$ after passing to a suitable convergent subsequence $\omega_{n_k}(t)$.  Now, assume that $G$ acts on $M$ with orbits of the same dimension $\kappa$.  If $\O_\mathrm{prin}$ is a principal orbit and $\O$ a principal or exceptional orbit, there is an equivariant covering map $\O_\mathrm{prin} \rightarrow \O$, so that $\O_\mathrm{prin}$ and $\O$ are locally diffeomorphic, compare \cite[p. 181]{bredon}. Therefore, we can assume that all orbits in $Y$ are diffeomorphic, which implies that the more $y_n$ approaches $x$, the faster the direction of $\kappa(x)-\kappa(g_n(t) \cdot y_n)$ changes as $t \in (-\eps_n,\eps_n)$ varies, and the faster $\Ncal(\omega_n(t))$ changes as $t \in (-\eps_n,\eps_n)$ varies. Consequently, the Gaussian curvature of $\Sigma_x^{R,t}$ at $\omega_\infty$, which is given by the product of the principal curvatures,   cannot stay bounded, compare Figure \ref{fig:orbits}.

\medskip

\begin{figure}[h!]
\begin{center}\includegraphics[width=0.4\linewidth]{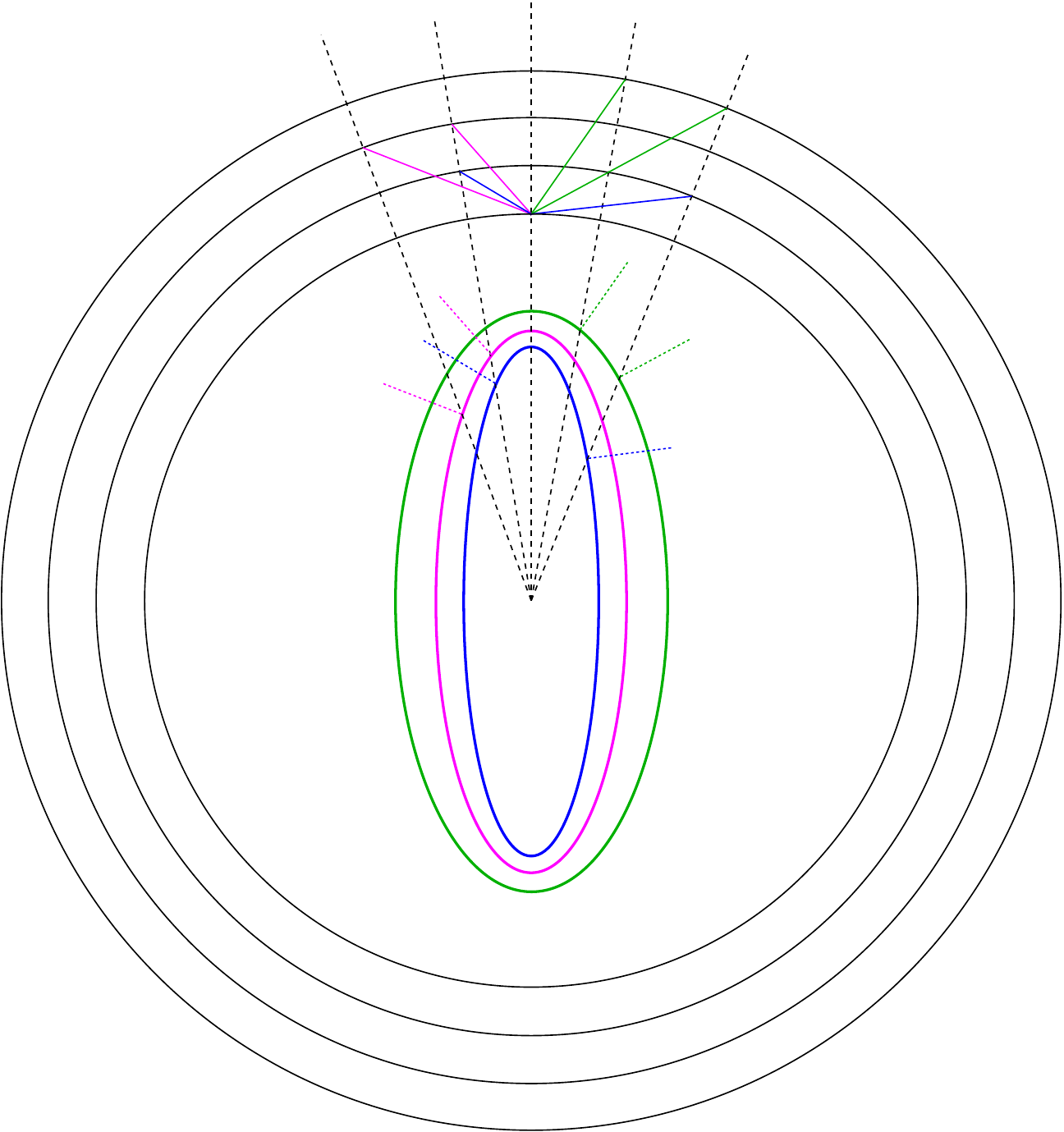}
\end{center}
\vspace{-.3cm}

\caption{\small Concerning the critical set of $\Phi_{x,y}$ in case that $y\not \in \O_x$. The black circle segments represent $G$-orbits in $Y\equiv \kappa(Y)\subset \R^n$, the inner one through $x$ and the outer ones through different points $y$; the black dotted lines represent normal spaces to the orbits. The three coloured ellipse segments depict different hypersurfaces $\Sigma^{R,t}_x\subset \R^n$ whose unit normal at $\omega\in N_{g\cdot y} \O_y\cap \Sigma^{R,t}_x$, depicted by a colored dotted line,  is determined  by the corresponding colored line segments $\kappa(x)-\kappa(g\cdot y)$.}\label{fig:orbits}
\end{figure}

\smallskip

\noindent
Thus, we have shown that for sufficiently small $Y$ we locally have 
$
\Crit \, \Phi_{x,y}\simeq  G_y
$,
 which implies that $\Crit \, \Phi_{x,y} $ is a smooth submanifold of codimension $n-1+\dim \O_y$. We are left with the task of showing that $\mathrm{Hess}\,  \Phi_{x,y}$ is transversally non-degenerate. For this, we are going to show that for each fixed $(\omega,g)\in \Crit \, \Phi_{x,y}$  one has
$
\mathrm{Ker} \,  \mathrm{Hess}\,  \Phi_{x,y}(\omega,g) \simeq T_{(\omega,g)}  \Crit \, \Phi_{x,y}.
$
To do so, note that with respect to the coordinates introduced at the beginning, the Hessian $\mathrm{Hess} \, \Phi_{x,y}(\omega,g) $ of $\Phi_{x,y}$ at a critical point $(\omega,g)$ is  given by the matrix
\begin{gather}
\begin{split}
\label{eq:27.05.2016}
\left ( \begin{array}{c | c} \vspace{3mm}  \eklm{\kappa (x)-\kappa (g \cdot y), \frac{\gd ^2 F}{\gd \alpha_i \gd \alpha_j}(\alpha^{-1}(\omega))}  & \sum_{k=1}^n \frac{\gd F_k (\alpha^{-1}(\omega)) }{\gd \alpha_i} (d\tilde x_k)_{g\cdot y}(\widetilde X_j) \\ \hline \\ \sum_{k=1}^n \frac{\gd F_k (\alpha^{-1}(\omega)) }{\gd \alpha_j}(d\tilde x_k)_{g\cdot y}(\widetilde X_i) &  -\frac 12 \eklm{\widetilde X_{i,g\cdot y} (\widetilde X_j (\kappa))+\widetilde X_{j,g\cdot y} (\widetilde X_i (\kappa)) ,\omega}
 \end{array} \right  ).
\end{split}
\end{gather}
Since $\kappa(g\cdot y)-\kappa(x)\in N_\omega\Sigma^{R,t}_x$, the submatrix in the first quadrant corresponds to a multiple of the second fundamental of $\Sigma^{R,t}_x$
\bqn 
\mathrm{II}: T \Sigma^{R,t}_x \times T \Sigma^{R,t}_x \longrightarrow \Cinft(\Sigma^{R,t}_x), \quad \mathrm{II}(\mathcal{X},\mathcal{Y}):=\eklm{\nabla_\mathcal{X} \mathcal{Y}, \mathcal{N}}=\eklm{\mathcal{X},A \, \mathcal{Y}},
\eqn
where $\nabla_\mathcal{X} \mathcal{Y}\equiv \mathcal{X} (\mathcal{Y})$ denotes the covariant derivative in Euclidean space $\R^n$  and $A: T \Sigma^{R,t}_x \rightarrow T \Sigma^{R,t}_x$ the symmetric endomorphism induced by $\mathrm{II}$ \cite[Chapter VII, Section 3]{kobayashi-nomizuII}. Indeed, assume that $\kappa(x)-\kappa(g\cdot y)$ points in the direction of $-\Ncal (\omega)$, and let $\gd /{\gd \alpha_{i}}_{|\omega}:= \gd F (\alpha^{-1}(\omega))/{\gd \alpha_{i}}$, $1\leq i\leq n-1$,  be the coordinate frame given by the parametrization \eqref{eq:paramhyp}. Then, the entries of the submatrix in the first quadrant of $\eqref{eq:27.05.2016}$ read
\bq
\label{eq:10.04.2016}
-\norm{\kappa(x)-\kappa(g \cdot y)} \, \mathrm{II}\left (\frac {\gd}{\gd \alpha_{i |\omega}},\frac {\gd}{\gd \alpha_{j |\omega}}\right )=-\norm{\kappa(x)-\kappa(g \cdot y)} \, \eklm{\frac {\gd}{\gd \alpha_{i |\omega}},A \frac {\gd}{\gd \alpha_{j |\omega}} }.
\eq
To compute the kernel of the matrix  \eqref{eq:27.05.2016}, assume that  the $X_1,\dots, X_d\in \g$ are  such that the vector fields $\widetilde X_1,\dots, \widetilde X_\kappa$  constitute  an orthonormal basis of  $T_{g\cdot y}\O_y$ at $g \cdot y$, while the vector fields $\widetilde X_{\kappa+1},\dots, \widetilde X_d$ vanish at $g\cdot y$,  
and  consider for $(\tilde \alpha, \tilde s) \in \R^{n-1} \times \R^d$ the system of equations
\bq
\label{eq:28.05.2016a}
\sum_{j=1}^{n-1} \eklm{\kappa (x)-\kappa (g \cdot y), \frac{\gd ^2 F}{\gd \alpha_i \gd \alpha_j}(\alpha^{-1}(\omega))} \tilde \alpha_j +  \sum_{k=1}^n \frac{\gd F_k (\alpha^{-1}(\omega)) }{\gd \alpha_i} (d\tilde x_k)_{g\cdot y}\big (\widetilde X(\tilde s) \big ) =0
\eq
with  $i=1,\dots, n-1$, as well as
\bq
\label{eq:28.05.2016b}
 \sum_{k=1}^n \sum_{j=1}^{n-1}\tilde \alpha_j \frac{\gd F_k (\alpha^{-1}(\omega)) }{\gd \alpha_j}(d\tilde x_k)_{g\cdot y}(\widetilde X_i) - \frac 12 \eklm{\widetilde X_{i,g\cdot y} \big (\widetilde X(\tilde s) (\kappa)\big )+\widetilde X(\tilde s)_{g\cdot y} \big (\widetilde X_i(\kappa)\big ),\omega}=0
\eq
 with $i=1, \dots, d$, where we wrote again $X(\tilde s) := \sum_{j=1}^d \tilde s_j  X_{j}$ for short. We have to show that Equations \eqref{eq:28.05.2016a}--\eqref{eq:28.05.2016b}  are equivalent to   $\tilde \alpha=0$, $\tilde s_1=\dots=\tilde s_\kappa=0$.   Writing $\mathcal{W}_\omega(\tilde \alpha):=\sum_{j=1}^{n-1} \tilde \alpha_j \gd /{\gd \alpha_j}_{|\omega}$ and identifying $Y$ with $\kappa(Y)$, the system of equations \eqref{eq:28.05.2016a}  reads
 \bqn
 -\norm{x- g \cdot y} \, \eklm{\frac {\gd}{\gd \alpha_{i |\omega}},A \,\mathcal{W}_\omega(\tilde \alpha) }+\bigg \langle \frac{\gd }{\gd \alpha_{i|\omega}}, \widetilde X(\tilde s)_{g\cdot y} \bigg \rangle =0,\qquad i=1,\dots, n-1, 
 \eqn
which is equivalent to 
 \bq
 \label{eq:29.05.2016a}
\mathcal{W}_\omega(\tilde \alpha)=\norm{x-g \cdot y}^{-1} A^{-1} \big ( \mathrm{proj}_{|T_\omega \Sigma^{R,t}_x} (\widetilde X(\tilde s)_{g\cdot y} )\big ),
 \eq
compare Figure \ref{fig:nondeg}.  

   \medskip
 
    \begin{figure}[h!]
    \begin{center}
\includegraphics[width=0.4\linewidth]{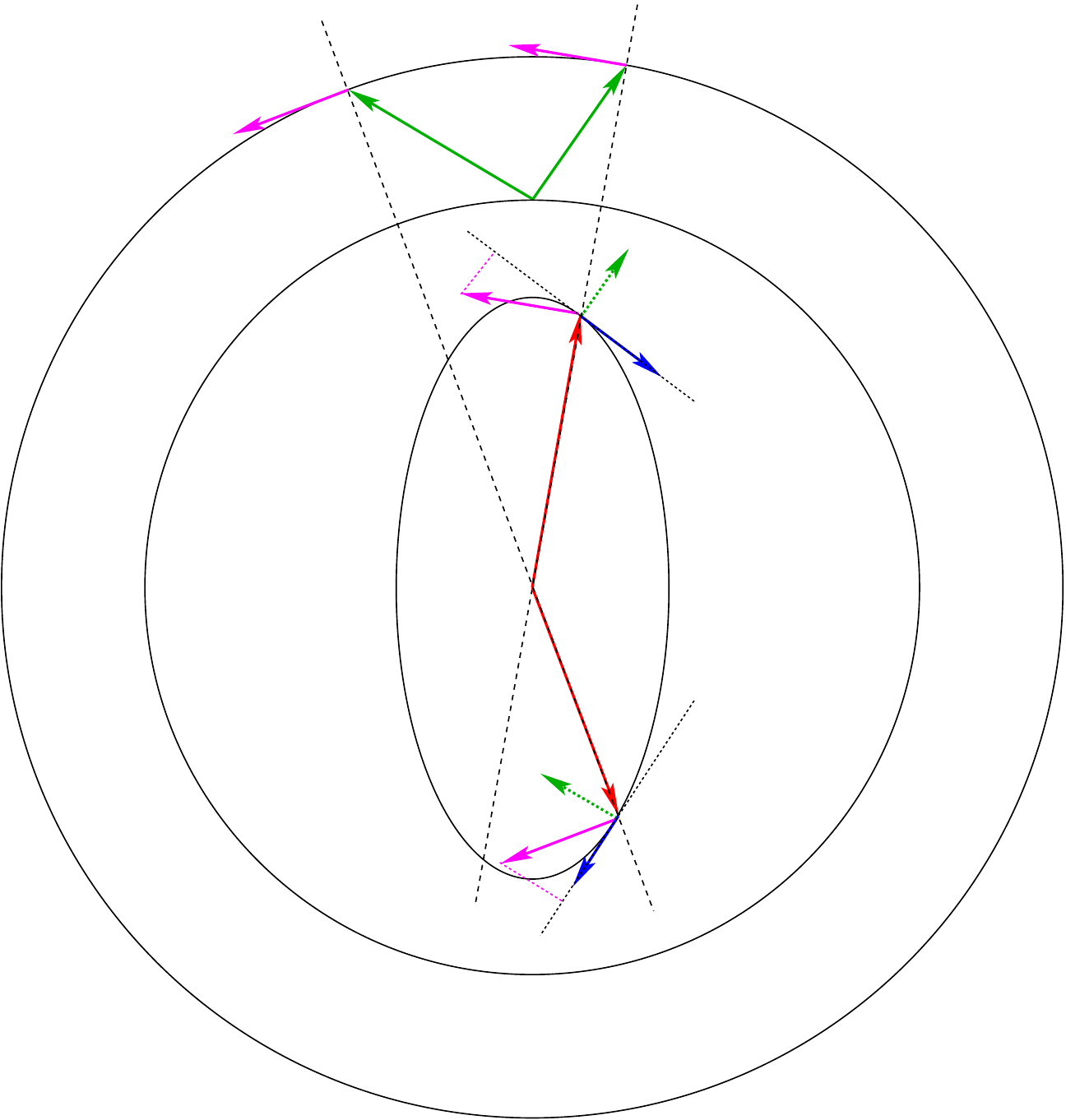}
\end{center}
\vspace{-.3cm}

\caption{\small Concerning the cleanness of the critical set of $\Phi_{x,y}$ in case that $y\not \in \O_x$. Black circles represent $G$-orbits in $Y\equiv \kappa(Y)\subset \R^n$ through $x$ and $y$, respectively; the black dotted lines represent normal spaces to the orbits and  tangent spaces to the hypersurface $\Sigma^{R,t}_x$, respectively, the latter being depicted by an ellipse. The red arrows represent  points $\omega\in \Sigma^{R,t}_x$, the green arrows segments $\kappa(g\cdot y)-\kappa(x) $. The magenta arrows depict vectors $\widetilde X(\tilde s)_{g \cdot y}$ and the blue arrows the corresponding vectors $\mathcal{W}_\omega(\tilde \alpha)$, compare \eqref{eq:29.05.2016a}. }\label{fig:nondeg}
\end{figure}

 \noindent
 Note that $A$ is invertible, since the Gaussian curvature of $\Sigma^{R,t}_x$ does not vanish. Furthermore,  since $\Sigma_x^{R,t}$ is strictly convex, the eigenvalues of $A$, which are given by the principal curvatures of $\Sigma^{R,t}_x$ with respect to the {outer} unit normal vector field, are strictly  {negative}\footnote{Note that  the sign convention used here is such that  if $\Sigma^{R,t}_x$ equals the standard $(n-1)$-sphere $S^{n-1}(R)$ of radius $R$,   then  $A=-\1/R$, where $\1$ represents the identity transformation on $T_\omega S^{n-1}(R)$, see \cite[Chapter VII, Example 4.2]{kobayashi-nomizuII}.}. Hence $A$ defines a non-positive operator on $T_\omega \Sigma_x^{R,t}$.  On the other hand, \eqref{eq:28.05.2016b} amounts to   the equations 
 \bq
 \label{eq:29.05.2016b}
 \langle {\mathcal{W}_\omega(\tilde \alpha),\widetilde X_{i,g\cdot y}}\rangle= \frac 1 2\Big (\widetilde X_{i,g \cdot y} \big (\langle \widetilde X(\tilde s), \omega\rangle \big ) + \widetilde X(\tilde s)_{g \cdot y} \big (\langle \widetilde X_i,\omega \rangle \big )\Big ), \qquad i=1,\dots,d. 
 \eq
Inserting \eqref{eq:29.05.2016a} into \eqref {eq:29.05.2016b} one obtains for all $i=1,\dots, d$
\bq
\label{eq:Koppelow/02.08.2017} 
F^i_{\tilde s} (g\cdot y,\omega)= \norm{x-g \cdot y} \, G^i_{\tilde s} (g\cdot y,\omega),  
 \eq
where we set
\bqn
F^i_{\tilde s} (z,\omega):=\langle {A^{-1} \big ( \mathrm{proj}_{|T_\omega \Sigma^{R,t}_x} (\widetilde X(\tilde s)_{z} )\big ),\widetilde X_{i,z}}\rangle, \quad  G^i_{\tilde s} (z,\omega):= \frac 1 2\Big (\widetilde X_{i,z} \big (\langle \widetilde X(\tilde s), \omega\rangle \big ) + \widetilde X(\tilde s)_{z} \big (\langle \widetilde X_i,\omega \rangle \big )\Big ).
\eqn
As functions on $\mklm{(z,\omega) \in Y \times \Sigma^{R,t}_x\mid \, \omega \in N_z\O_z}$, $F^i_{\tilde s}$ and $G^i_{\tilde s}$ are smooth and bounded from above. Furthermore, the projection from $T_{z} \O_z$ to $T_\omega \Sigma_x^{R,t}$  has a trivial kernel if $\omega$ is normal to $\O_z$ at $z$, since $\omega$ cannot be tangential to $\Sigma_x^{R,t}$ in view of \eqref{eq:07.04.2016a}. Therefore, 
\bq
\label{eq:IHP} 
F^i_{\tilde s} (z,\omega)=0 \quad \text{for all $i=1,\dots, d$} \quad \Longleftrightarrow \quad \widetilde X(\tilde s)_{z}=0,
\eq
$A$ being a non-positive operator. Let us now assume that $\widetilde X(\tilde s)_{g\cdot y}\not=0$. Choosing $Y$ sufficiently small, we deduce from \eqref{eq:IHP} that there is at least one $i$ such that $|F^i_{\tilde s} (z,\omega)| \geq C^i_{\tilde s}$ on $Y$ for some uniform constant $C^i_{\tilde s}>0$. But then, letting $Y$ become even smaller so that $1/\norm{x-g \cdot y}$ becomes large compared to $|G^i_{\tilde s} (g\cdot y,\omega)|/C^i_{\tilde s}$ we arrive at a contradiction in view of  \eqref{eq:Koppelow/02.08.2017}. Thus, we must have $\widetilde X(\tilde s)_{g\cdot y}=0$, which in turn implies $\mathcal{W}_\omega(\tilde \alpha)=0$ by \eqref{eq:29.05.2016a}. We have therefore shown that Equations  \eqref{eq:28.05.2016a}--\eqref{eq:28.05.2016b}  are fulfilled iff $\tilde s_1=\dots=\tilde s_\kappa=0$ and  $\tilde \alpha=0$, so that
 \bqn 
 \mathrm{Ker} \,  \mathrm{Hess}\,  \Phi_{x,y}(\omega,g) \simeq \mklm{0} \times \R^{d-\kappa} \simeq  \T_{(\omega,g)}  \Crit \, \Phi_{x,y},
 \eqn
and we obtain  (b). {An alternative proof of the fact that the Hessian of $\Phi_{x,y}$ is transversally non-degenerate in the cases (a) and (b) will be given in Theorem \ref{thm:12.05.2015} by explicitly computing the transversal Hessian.}  

In order to show (c), let $x \in Y\cap M_\mathrm{prin}$ and 
 $(\omega,g) \in \Ccal_{x,x}$. If $x$ is of principal isotropy type, $G_x$ acts trivially on $N_x(G\cdot x)$ \cite[pp. 308 and 181]{bredon} and, via the identification $T^\ast M\simeq T M$, also on $\mathrm{Ann}(T_x(G\cdot x))$. But  in view of   \eqref{eq:Ann} and (a) we have $\omega \in \mathrm{Ann}(T_x(G\cdot x))$, so  that   $g \cdot \omega =\omega$ in this case, and with \eqref{eq:23.04.2015} we obtain the desired inclusion and therefore (c).  In particular, since $\Crit \, \Phi$ has codimension $2\kappa$, $\Ccal_{x,x}$ has codimension $2\kappa$ as well.
\end{proof}

\begin{rem} \hspace{0cm}
\begin{enumerate}
\item Let   $y  \notin \O_x$. As an example where $\Crit \, \Phi_{x,y}$ is not isomorphic to $G_y$, and does not have codimension $n-1+\kappa$, consider the singular action of $G=\SO(2)$ on the standard $2$-sphere $M=S^2\subset \R^3$ by rotations around the poles $x_N,x_S$, and assume that $\Sigma_x^{R,t}=S^{1}$. Let $(Y,\kappa)$ be an invariant tubular neighborhood around the fixed point $x_N$. Then, for any $y\in Y-\mklm{x_N}$ one has 
\bqn 
\Crit \, \Phi_{x_N,y}=\mklm{(\omega,g)\mid (g \cdot y,\omega) \in N_{g\cdot y}(G\cdot y), \,  \kappa(x_N)-\kappa (g \cdot y) \parallel  \omega }\simeq \SO(2)\times \Z_2\not \simeq G_y=\mklm{e},
\eqn
which has codimension $\kappa=1$ instead of $2$, showing the necessity of the assumption in Lemma \ref{lem:21.04.2015} (b) that all $G$-orbits must have the same  dimension. 

\item Note that Lemma \ref{lem:21.04.2015} (c) cannot hold in general for arbitrary $x \in Y\cap (M_\mathrm{except} \cup M_\mathrm{sing})$. In particular, if $x$ were a fixed point we would have  $\Phi_{x,x}\equiv 0$, so that  $\Crit \, \Phi_{x,x}=\Sigma^{R,t}_x \times G$ in this case. Furthermore, Assertion (c) means that $ \Phi_{x,x}$ does not have secondary critical points for $x \in Y\cap M_\mathrm{prin}$, that is, critical points which do not arise from critical points of $\Phi$.
\end{enumerate}
\end{rem}

From the previous lemma one now deduces

\begin{thm}
\label{thm:12.05.2015}
Assume that $G$ is a continuous compact Lie group acting on $M$ with orbits of dimension less or equal $n-1$, and consider the oscillatory integrals $I_{x,y}(\mu)$  defined  in \eqref{eq:03.05.2015}.
 \begin{enumerate}
 \item[(a)] Let $y \in \mathcal{O}_x$. Then, for every $\tilde N$ one has the asymptotic formula
\bqn
  I_{x,y}(\mu)=(2\pi/\mu)^{\dim \mathcal{O}_x} \left [\sum_{k=0}^{\tilde N-1} \mathcal{Q}_{k}(x,y) \mu^{-k}  +\mathcal{R}_{\tilde N}(x,y,\mu)\right ],  \qquad \mu \to +\infty,
\eqn
with  explicitly known coefficients and remainder. In particular, 
\bqn
\mathcal{Q}_0(x,y)=\int_{ \mathcal{C}_{x,y}} \frac { a(  x , y, \omega,g) }{|\det   \, \Phi_{x,y}''(\omega,g)_{N_{( \omega, g)} \mathcal{C}_{x,y}}|^{1/2}} \d\mathcal{C}_{x,y}(\omega,g),
\eqn
where   $d \mathcal{C}_{x,y}$ denotes the induced volume density. Furthermore, $\mathcal{Q}_k(x,y)$ and $\mathcal{R}_{\tilde N}(x,y, \mu)$ depend smoothly on $R$ and $t$, and satisfy the bounds
\begin{align*}
|\mathcal{Q}_k(x,y)|&\leq  C_{k,\Phi_{x,y}} \vol (\supp a(x,y,\cdot,\cdot)\cap \mathcal{C}_{x,y}) \sup _{l\leq 2k} \norm{D^l a(x,y, \cdot,\cdot)}_{\infty,\mathcal{C}_{x,y}}, \\
|\mathcal{R}_{\tilde N}(x,y,   \mu) | &\leq \widetilde C_{\tilde N,\Phi_{x,y}}  \vol (\supp a  (x,y,\cdot,\cdot)) \sup_{l\leq  2\tilde N+\dim \mathcal{O}_x +1 } \norm{D^l a(x,y,\cdot, \cdot )}_{\infty,  \Sigma^{R,t}_x \times G} \, \mu^{-\tilde N},
\end{align*}
uniformly in $R,t$ for suitable constants $C_{k,\Phi_{x,y}}>0$ and  $\widetilde C_{\tilde N,\Phi_{x,y}}>0$, where $D^l$ denote  differential operators of order $l$ on $\Sigma^{R,t}_x \times G$. Moreover, as functions in $x$ and $y$, $\mathcal{Q}_k(x,y)$ and $\mathcal{R}_{\tilde N}(x,y,\mu)$ are smooth on  $Y \cap M_\mathrm{prin}$, and the constants $C_{k,\Phi_{x,y}}$ and  $\widetilde C_{\tilde N,\Phi_{x,y}}$ are uniformly bounded in $x$ and $y$ if  $M= M_\mathrm{prin} \cup M_\mathrm{except}$. If the amplitude factorizes according to $a(x,y,\omega,g)=a_1(x,y,\omega) \, a_2(x,y,g)$, the remainder can also be estimated by 
\bqn 
|\mathcal{R}_{\tilde N}(x,y,   \mu) | \leq \widetilde C_{\tilde N,\Phi_{x,y}}  \prod_{i=1,2} \vol (\supp a_i  (x,y,\cdot)) \sup_{l\leq  2\tilde N+\lfloor \dim \mathcal{O}_x/2 +1\rfloor } \norm{D^l_i a_i(x,y,\cdot )}_{\infty, \M_i} \, \mu^{-\tilde N},
\eqn
where $D^l_1$ and $D^l_2$ denote differential operators of order $l$ on $\M_1=\Sigma^{R,t}_x$ and $\M_2=G$, respectively. 
\smallskip

\item[(b)] Let $y \not \in \mathcal{O}_x$. Assume that  $M=M_\mathrm{prin}\,  \cup\,  M_\mathrm{except}$ and that the co-spheres $S_x^\ast M$ are strictly convex.  Then, for sufficiently small $Y$ and every $\tilde N \in \N$   one has the asymptotic formula
\bqn
  I_{x,y}(\mu)=\sum_{\J \in \pi_0(\Crit \, \Phi_{x,y})}(2\pi/\mu)^{\frac{n-1+\kappa}{2}} e^{i\mu \,^0\Phi_{x,y}^\J} \left [  \sum_{k=0}^{\tilde N-1} \mathcal{Q}_{\J,k}(x,y) \mu^{-k}  +\mathcal{R}_{\J, \tilde N}(x,y,\mu)\right ]
\eqn
as $\mu \to +\infty$ with explicitly known coefficients and remainder, where $\kappa:=\dim M/G$.  The coefficients $\mathcal{Q}_{\J,k}(x,y)$  and the remainder term $\mathcal{R}_{\J,\tilde N}(x,y,\mu) $ are given by distributions depending smoothly on $R,t$, and $x,y \in Y \cap M_\mathrm{prin}$ with support in $\Crit  \, \Phi_{x,y}$ and $\Sigma^{R,t}_x \times G$, respectively. Furthermore, they satisfy bounds analogous to the ones in (1), where now the constants $C_{k,\Phi_{x,y}}$ and  $\widetilde C_{\tilde N,\Phi_{x,y}}$ are no longer uniformly bounded, but satisfy
\begin{align*}
C_{k,\Phi_{x,y}}& \ll  \dist(y, \O_x)^{-(n-1-\kappa)/2 -k}, \qquad 
\widetilde C_{\tilde N,\Phi_{x,y}} \ll \dist(y, \O_x)^{-(n-1-\kappa)/2-\tilde N}.
\end{align*}
Finally, $^0\Phi_{x,y}^\J$ stands for  the constant values of $\Phi_{x,y}$ on the connected components $\J$ of  its critical set, and is  given by
\bqn
^0\Phi_{x,y}^\J(R,t)= R \, c_{x,g \cdot y}(t), \qquad c_{x,g \cdot y}(t):=\pm \frac{\norm {\kappa(x) - \kappa( g_\J \cdot y )}}{\norm {\grad_\eta \zeta(t,\kappa(x),\omega_\J) }}, \qquad (\omega_\J,g_\J) \in \J.
\eqn
\end{enumerate}
\end{thm}

\begin{proof}
The asymptotic expansions for the integrals $I_{x,y}(\mu)$,  the smoothness of the coefficients $\mathcal{Q}_{k}(x,y)$, $\mathcal{Q}_{\J,k}(x,y)$, and the remainder terms  in the parameters $R,t$, and $x,y \in Y\cap M_\text{prin}$, as well as the bounds satisfied by them  are a direct consequence of Lemma  \ref{lem:21.04.2015}, Theorem \ref{thm:SP}, and Remark \ref{rem:25.07.2017},  together with \cite[Theorem 7.7.6]{hoermanderI}.  
To see that  the constants $C_{k,\Phi_{x,y}}$, $\widetilde C_{\tilde N,\Phi_{x,y}}$ satisfy the specified bounds, we have to compute the transversal Hessian of $\Phi_{x,y}$ and its determinant in the two cases $y \in \O_x$ and $y \not \in \O_x$. Recall the notation and the proof of  Lemma \ref{lem:21.04.2015}, and let $(\omega,g) \in \Crit \, \Phi_{x,y}$ be a critical point.  As a bilinear form on $T_\omega \Sigma^{R,t}_x\times T_gG$, the Hessian of $\Phi_{x,y}$ can be written as the $(n-1+d)\times (n-1+d)$--matrix
\bqn
\mathrm{Hess} \, \Phi_{x,y}(\omega,g)\equiv \left ( \begin{array}{c  | c} \vspace{3mm}  -\norm{x-g \cdot y} \, \eklm{\frac {\gd}{\gd \alpha_{i |\omega}},A \frac {\gd}{\gd \alpha_{j |\omega}} }  & \eklm{\frac {\gd}{\gd \alpha_{i |\omega}},\widetilde X_{j,g \cdot y}} \\  \hline  \\ \eklm{\frac {\gd}{\gd \alpha_{j |\omega}},\widetilde X_{i,g \cdot y}} &  -\frac 12 \Big (\widetilde X_{i,g\cdot y} \big (\langle \widetilde X_j ,\omega\rangle\big )+\widetilde X_{j,g\cdot y} \big (\langle \widetilde X_i ,\omega\rangle\big ) \Big )
 \end{array} \right  ),
\eqn
 compare   \eqref{eq:27.05.2016bis} and  \eqref{eq:27.05.2016}, where we took into account \eqref{eq:10.04.2016}, and made the identification $\kappa(Y)\simeq Y$. Note that for $y \in \O_x$ one has $g\cdot y =x$.  Now, recall that the projection from $T_{g\cdot y} \O_y$ to $T_\omega \Sigma^{R,t}_x$ has a trivial kernel, and choose vectors $\tilde\alpha^1,\dots, \tilde \alpha^{n-1}\in \R^{n-1}$ such that one has the decomposition
\begin{align*}
\begin{split}
T_\omega \Sigma^{R,t}_x &=\mathrm{proj}_{|T_\omega \Sigma^{R,t}_x}(T_{g \cdot y} \O_y) \oplus \mathrm{Span}\mklm{\mathcal{W}_\omega(\tilde \alpha) \in N_{g\cdot y} \O_y}\\
&=\mathrm{Span}\mklm{\mathcal{W}_\omega(\tilde \alpha^1), \dots, \mathcal{W}_\omega(\tilde \alpha^{\dim \O_x})} \oplus \mathrm{Span}\mklm{\mathcal{W}_\omega(\tilde \alpha^{{\dim \O_x}+1}), \dots, \mathcal{W}_\omega(\tilde \alpha^{n-1})},
\end{split}
\end{align*}
where we wrote $\mathcal{W}_\omega(\tilde \alpha):=\sum_{j=1}^{n-1} \tilde \alpha_j \gd /{\gd \alpha_j}_{|\omega}$. Further, suppose that the $X_j\in \g $ have been chosen such that the vector fields $\{\widetilde X_1, \dots, \widetilde X_{\dim \O_x}\}$ constitute an orthonormal basis of $T_{g\cdot y} \O_y$ at $g\cdot y$, while the vector fields $\{\widetilde X_{{\dim \O_x}+1}, \dots, \widetilde X_d\}$ vanish at $g \cdot y$. Then, with respect to the basis $\mklm{\mathcal{W}_\omega(\tilde \alpha^i)}$ the Hessian of $\Phi_{x,y}(\omega,g)$ is essentially given by  the $(n-1+{\dim \O_x})\times (n-1+{\dim \O_x})$--matrix
\bqn
\mathcal{M}_{x,y}(\omega,g):= \left ( \begin{array}{c | c} \vspace{3mm}  -\norm{x-g \cdot y} \, \eklm{\mathcal{W}_\omega(\tilde \alpha^i),A \mathcal{W}_\omega(\tilde \alpha^j) }  & \begin{array}{c} \eklm{\mathcal{W}_\omega(\tilde \alpha^i),\widetilde X_{j,g \cdot y}}   \\ \hline \\ 0 \end{array} \\ \hline \\  \eklm{\mathcal{W}_\omega(\tilde \alpha^j),\widetilde X_{i,g \cdot y}}  \quad \Big {|} \qquad   0 &  -\frac 12 \Big (\widetilde X_{i,g\cdot y} \big (\langle \widetilde X_j ,\omega\rangle\big )+\widetilde X_{j,g\cdot y} \big (\langle \widetilde X_i ,\omega\rangle\big ) \Big )
 \end{array} \right  ),
\eqn
since $\langle \widetilde X_j,\omega \rangle$ has a zero of second order at $g\cdot y$ in orbit direction for $j=\dim \O_x+1, \dots, d$.  If $y \not \in \O_x$, the transversal Hessian of $\Phi_{x,y}(\omega,g)$ is given by  $\mathcal{M}_{x,y}(\omega,g)$; if $y \in \O_x$, it is given by the $(2\, {\dim \O_x} \times 2\, {\dim \O_x})$--matrix
\bqn
\mathcal{M}_{g\cdot y}(\omega):= \left ( \begin{array}{c | c} \vspace{3mm}  0  & \eklm{\mathcal{W}_\omega(\tilde \alpha^i),\widetilde X_{j,g\cdot y}}   \\ \hline \\  \eklm{\mathcal{W}_\omega(\tilde \alpha^j),\widetilde X_{i,g\cdot y}}  & -\frac 12 \Big (\widetilde X_{i,g\cdot y} \big (\langle \widetilde X_j ,\omega\rangle\big )+\widetilde X_{j,g\cdot y} \big (\langle \widetilde X_i ,\omega\rangle\big ) \Big )
 \end{array} \right  ),
\eqn
which is obtained from  $\mathcal{M}_{x,y}(\omega,g)$ by  removing the $({\dim \O_x}+1)$-th, ..., $(n-1)$-th columns and rows. Clearly, $\det \mathcal{M}_{g\cdot y}(\omega)=\det \mathcal{M}'_{g\cdot y}(\omega)$, where $\mathcal{M}'_{g\cdot y}(\omega)$ is the matrix obtained from $\mathcal{M}_{g\cdot y}(\omega)$ by setting the coefficients in the fourth quadrant equal to zero. A computation as in \eqref{eq:07.04.2016b}-\eqref{eq:07.04.2016c} then shows that the kernel of the linear transformation corresponding to $\mathcal{M}'_{g\cdot y}(\omega)$ is trivial, so that $\det \mathcal{M}_{g\cdot y}(\omega)\not =0$. Thus, in the case (a) we have shown again that the transversal Hessian  of $\Phi_{x,y}$ is non-degenerate, as in Lemma \ref{lem:21.04.2015} (a),  and that 
\bqn 
\frac 1{\det \mathrm{Trans \, Hess} \, \Phi_{x,y}(\omega,g)} =\frac 1{\det \mathcal{M}_{x}(\omega)} \ll 1
\eqn
uniformly in $x,y\in M_\mathrm{prin} \cup M_\mathrm{except}$, since principal and exceptional orbits are locally diffeomorphic \cite[p.\ 181]{bredon}, and principal and exceptional isotropy groups infinitesimally isomorphic. On the other hand, in the case (b),  suppose that the matrix $A$ has  diagonal form  with respect to the basis $\mathcal{W}_\omega(\tilde \alpha^i)$. Denote its entries, which correspond to the principal curvatures of $\Sigma^{R,t}_x$, by $(\rho_1,\dots, \rho_{n-1})$. Then
\begin{align}
\label{eq:8.6.2017}
\begin{split}
\det \mathcal{M}_{x,y}(\omega,g)=& \norm{x-g \cdot y}^{n-1-\kappa} \\
& \cdot \Big ( c_0+\norm{x-g \cdot y} c_1+ \dots + \norm{x-g \cdot y}^\kappa c_\kappa \Big ), \qquad c_i \in \R, 
\end{split}
\end{align}
where 
\bqn 
c_0=\pm \rho_{\kappa+1} \dots \rho_{n-1} \, \det \mathcal{M}_{g \cdot y}(\omega).
\eqn
Since $\det \mathcal{M}_{g \cdot y}(\omega)$ is uniformly bounded away from zero,  we have $|c_0| \geq C >0$ for a uniform constant  $C>0$. Taking $Y$ sufficiently small, it is mainly the term $c_0$ that contributes to $\det \M_{x,y}(\omega,g)$ so that  we conclude again that the Hessian of $\Phi_{x,y}$ is transversally non-degenerate in the case (b), compare  Lemma \ref{lem:21.04.2015} (b), and that 
\bqn 
\frac 1{\det \mathrm{Trans \, Hess} \, \Phi_{x,y}(\omega,g)} = \frac 1{\det \mathcal{M}_{x,y}(\omega,g)} \ll  \norm{x-g \cdot y} ^{-(n-1-\kappa)}
\eqn
uniformly in $x,y\in M=M_\mathrm{prin}\,  \cup\,  M_\mathrm{except}$. Summing up, we have shown on $M_\mathrm{prin}\,  \cup\,  M_\mathrm{except}$ the uniform bound
\bqn 
\frac 1{\det \mathrm{Trans \, Hess} \, \Phi_{x,y}(\omega,g)} \ll \begin{cases} 1, & y \in \O_x, \\ \dist (y, \O_x)^{-(n-1-\kappa)},& y \not\in \O_x. \end{cases}
\eqn
 By  \eqref{eq:25.07.2017} it then follows 
that the constants $C_{k,\Phi_{x,y}}$, $\widetilde C_{\tilde N,\Phi_{x,y}}$ satisfy the specified bounds.  Regarding the values of $\Phi_{x,y}$ on its critical set, note that for $(\omega,g) \in \Crit \, \Phi_{x,y}$ one computes with \eqref{eq:07.04.2016}
\bqn
\Phi_{x,y}^0(R,t)=\eklm{\kappa(x) - \kappa( g \cdot y ), \omega}=  \pm c_{x,g \cdot y}(t) \underbrace{\eklm{ \grad_\eta \zeta(t,\kappa(x),\omega), \omega}}_{=R} =\pm R \, c_{x,g \cdot y}(t),
\eqn
since $\kappa(x) - \kappa( g \cdot y )$ must be co-linear to $\grad_\eta \zeta(t,\kappa(x),\omega)$. In particular notice that $c_{x,g \cdot y}(t)$ is independent of $R$ due to the fact that $\zeta(t,\kappa(x),\eta)$ is homogeneous of degree $1$ in $\eta$, so that $\grad_\eta \zeta(t,\kappa(x),\omega)$ only depends on the direction of $\omega$.  
\end{proof}

As  the previous theorem shows, the integrals $I_{x,y}(\mu)$ exhibit a caustic behaviour\footnote{See Appendix \ref{appendix} for a discussion of the terminology.} in their dependence on the variables $x$ and $y$, obeying different asymptotics in the cases $y \in \O_x$ and $y \not \in \O_x$, respectively. In particular, in the latter case, the coefficients in the asymptotic expansion become singular as $y \to \O_x$. In what follows, we shall derive a uniform asymptotic expansion for the integrals $I_{x,y}(\mu)$ that interpolates between these two different asymptotic behaviours. This result will be necessary for deriving the equivariant $\L^p$-bounds in Section \ref{sec:equivLp}.

\begin{thm}
\label{thm:14.05.2017} 
Consider the integrals $I_{x,y}(\mu)$  defined in \eqref{eq:03.05.2015}. Assume that the continuous compact Lie group $G$ acts on $M$ with orbits of the same dimension $\kappa \leq n-1$,  and that the co-spheres $S_x^\ast M$ are strictly convex.  Then, for sufficiently small $Y$ and arbitrary $\tilde N_1, \tilde N_2 \in \N$   one has the asymptotic formula
\begin{gather*}
 I_{x,y}(\mu)\\
 = \sum_{\J \in \pi_0(\Crit \, \Phi_{x,y})}   \frac{e^{i \mu  \,^0\Phi_{ x,y}^\J}}{\mu^\kappa (\mu \norm{ \kappa(x)-\kappa(g_\J \cdot y)}+1 )^{\frac{n-1-\kappa}2}} \left [ \sum_{k_1,k_2 =0}^{\tilde N_1-1,\tilde N_2-1}  \frac {   \mathcal{Q}_{\J, k_1,k_2} (x,y)}{\mu^{k_1} (\mu \norm{ \kappa(x)-\kappa(g_\J \cdot y)}+1)^{k_2}}\right. \\  \left. + \mathcal{R}_{\J,\tilde N_1, \tilde N_2}(x,y,\mu) \right ]
\end{gather*}
as $\mu \to +\infty$.  The coefficients  and the remainder term 
\bqn
\mathcal{R}_{\J,\tilde N_1, \tilde N_2}(x,y,\mu)=O\Big (\mu^{-\tilde N_1}  (\mu \norm{ \kappa(x)-\kappa(g_\J \cdot y)}+1)^{-\tilde N_2}\Big )
\eqn
 are given by distributions depending smoothly on $R,t$ with support in each of the components $\J$ of $\Crit  \, \Phi_{x,y}$ and $\Sigma^{R,t}_{x} \times G$, respectively. Furthermore, they are uniformly bounded in $x$ and $y$ by derivatives of $a$ with respect to $g$  up to order $2k_1$ and $2\tilde N_1+\kappa+1$, respectively, while  
$^0\Phi_{x,y}^\J:= R \,  c_{x,g_\J\cdot y} (t)$ denotes the constant value of $\Phi_{x,y}$ on $\J$. If the amplitude factorizes according to $a(x,y,\omega,g)=a_1(x,y,\omega) \, a_2(x,y,g)$, the remainder can also be estimated by derivatives of $a$ with respect to $g$  up to order  $2\tilde N_1+\lfloor \kappa/2+1\rfloor$.
\end{thm}

\begin{proof}
Let the notation be as before, and recall from Lemma \ref{lem:21.04.2015} the description of the critical set of $\Phi_{x,y}$ in the two cases $y \in \O_x$ and $y\not \in \O_x$. For simplicity, let us assume that $G_y$ is connected. In the first case, $\Crit \, \Phi_{x,y}$ is given by the set
\bqn
\J= (\Sigma^{R,t}_x\cap N_{x} \O_x) \times \mklm{g_\J \cdot G_y},
\eqn
where $g_\J\in G$ is determined by the condition $g_\J \cdot y=x$, and is connected if $\kappa < n-1$. In the second case, each of the connected components of $\Crit \, \Phi_{x,y}$ has the form 
\bqn 
\J= \mklm{\omega_\J} \times \mklm{g_\J \cdot G_y},
\eqn
where $g_\J \in G$, $\omega_\J \in \Sigma^{R,t}_x\cap N_{g_\J\cdot y} \O_y$ are determined by the condition $\kappa(x)-\kappa(g_\J \cdot y) \in N_{\omega_\J}\Sigma^{R,t}_x$. 
In both cases, each of the  $\J \in \pi_0(\Crit \, \Phi_{x,y})$ is contained in a set of the form
\begin{align*}
 \Big \{(\omega,g)\mid \omega  \in \Sigma^{R,t}_x\cap N_{g_\J\cdot y} \O_y, \,   g \in g_\J \cdot G_y \Big \}.
\end{align*}
Let therefore $\U_\J \subset \Sigma^{R,t}_x$ be a sufficiently small neighbourhood of $\Sigma^{R,t}_x\cap N_{x} \O_x$ and $\omega_\J$, respectively, so  that on $\U_\J$ one has local coordinates of the form 
\bqn 
(-\eps,\eps)^\kappa \times (N_{g_\J \cdot y} \O_y\cap \U_\J) \ni (\alpha',\omega') \quad \longmapsto \quad F_\J(\alpha',\omega')=\omega \in \Sigma^{R,t}_x, \qquad \eps>0,
\eqn 
with $N_{g_\J \cdot y} \O_y\cap \U_\J \equiv \mklm{\alpha'=0}$, where we took into account that $\kappa \leq n-1$. Consider next a partition of unity $\mklm{\chi_\J,\chi_0}$ subordinated to the covering of $\Sigma^{R,t}_x$ by the  $\U_\J$ and a small neighbourhood $\U_0$ of their complement. Then, by the non-stationary phase theorem we have 
\bqn
I_{x,y}(\mu)=\sum_\J I^\J_{x,y}(\mu)+O(\mu^{-\infty}),
\eqn
 where we wrote
\begin{align}
\nonumber I^\J_{x,y}(\mu)&:=\int_{G}\int_{\U_\J} e^{i\mu \Phi_{x,y}(\omega,g)} \chi_\J(\omega) a (x,y, \omega,g) \d \Sigma^{R,t}_x (\omega) \d g \\
&=: \int_{N_{g_\J \cdot y} \O_y\cap \, \U_\J} \left [ \int_G \int_{(-\eps,\eps)^\kappa} e^{i\mu \Phi_{x,y}(F_\J(\alpha',\omega'),g)}  a_\J (x,y, \alpha',\omega',g) \d \alpha' \d g  \right ]  \d\omega'.
\label{eq:9.8.2017}
\end{align}
If $\kappa=n-1$, the intersection $N_{g_\J \cdot y} \O_y\cap \U_\J$ consists of isolated points, and $\d\omega'$ corresponds to the counting measure. 
Our intention is to apply the stationary phase principle first to the inner integral and after this to the outer integral. For this, let $\J$ and $\omega'$ be fixed, and introduce the phase function $\Phi^\J_{x,y,\omega'}(\alpha',g):=\Phi_{x,y}(F_\J(\alpha',\omega'),g)$.  We clearly have 
\bqn
\Crit \, \Phi^\J_{x,y,\omega'}=\mklm{(\alpha',g) \in (-\eps,\eps)^\kappa \times G\mid \, F_\J(\alpha',\omega') \in N_{g\cdot y}\O_y,  \, \kappa(x)-\kappa(g\cdot y) \perp V_{F_\J(\alpha',\omega')}},
\eqn
where we put $V_\omega:=\mathrm{Span}\big \{ \frac \gd{\gd \alpha'_i}_{|\omega}\big \}$.  The  two conditions imply that the vector $\kappa(x)-\kappa(g\cdot y)$ cannot be  roughly tangential to $T_{g\cdot y}\O_y$, unless it is zero. Now, if $y \in \O_x$, the mentioned vector becomes  almost  tangential  to $T_{g\cdot y}\O_y$ for small $Y$, so that we must have $x=g \cdot y$ and consequently 
\bq 
\label{eq:9.6.2017a}
\Crit \, \Phi^\J_{x,y,\omega'}=\mklm{(\alpha',g) \in (-\eps,\eps)^\kappa \times G\mid \, \alpha'=0,  \, g \in g_\J \cdot G_y},
\eq
since  $x=g_\J \cdot y$ in this case.  Let us now consider the case $y \not \in \O_x$. We then assert that if $Y$ is chosen sufficiently small, the critical set $\Crit \, \Phi^\J_{x,y,\omega'}$ is again given by \eqref{eq:9.6.2017a}. Indeed,  if $\kappa=n-1$, the condition $\kappa(x)-\kappa(g \cdot y)\perp V_{F_\J(\alpha',\omega')}$ implies that $\kappa(x)-\kappa(g \cdot y)$ is normal to $\Sigma^{R,t}_x$ at $F_\J(\alpha',\omega')$, which can only hold for $\alpha'=0$ and $g\in g_\J \cdot G_y$ by the choice of $\U_\J$, 
yielding \eqref{eq:9.6.2017a}. If $\kappa <n-1$, note that we can choose the coordinates $\alpha'$ in $\U_\J$ in such a way that 
 \bq
\label{eq:SGD}
\kappa(x)-\kappa(g_\J \cdot y) \perp V_{\omega'} \qquad \text{for all $\omega'\in N_{g_\J \cdot y} \O_y\cap  \U_\J$},
 \eq
showing the inclusion ''$\supset$'' in \eqref{eq:9.6.2017a}. To see the converse, assume as in the proof of Lemma \ref{lem:21.04.2015} (b) that for each $n \in \N$  there is a point $y_n\in W_n \cap Y \cap S_x$ approaching $x$ and a smooth curve 
\bqn 
(-\eps_n,\eps_n) \ni t \longmapsto (\alpha'_n(t),g_n(t)) \in \Crit \, \Phi^\J_{x,y_n,\omega'}
\eqn
parametrized such that $\norm{\dot \alpha_n'(t)}=1$. Since the vectors  $\kappa(x)-\kappa(g_n(t) \cdot y_n)$ are approximately normal to $T_{g_n(t)\cdot y_n}\O_{y_n}$,  the curves $\mklm{g_n(t) \cdot y_n\mid \, t \in (-\eps_n,\eps_n)}\subset Y$ must converge  to $x$ as  $n \to \infty$, while $\eps_n \to 0$.  Similarly, due to the compactness of $\Sigma_x^{R,t}$ the curves 
\bqn 
\mklm{F_\J(\alpha'_n(t), \omega')\mid \, t \in (-\eps_n,\eps_n)}\subset \Sigma_x^{R,t}
\eqn
converge to at least one $F_\J(\alpha'_\infty, \omega') \in \Sigma_x^{R,t}\cap N_x \O_x$, eventually after passing to a suitable convergent subsequence. By assumption, $G$ acts on $M$ with orbits of the same dimension $\kappa$, so that we can assume that all orbits in $Y$ are diffeomorphic. Consequently,  the more $y_n$ approaches $x$, the faster the direction of $\kappa(x)-\kappa(g_n(t) \cdot y_n)$ changes in orbital direction as $t \in (-\eps_n,\eps_n)$ varies. Thus, the Gaussian curvature of $\Sigma_x^{R,t}$ at $F_\J(\alpha'_\infty, \omega')$   cannot stay bounded in $\alpha'$-direction,  and  we have shown that for sufficiently small $Y$ we locally have 
$\Crit \, \Phi^\J_{x,y,\omega'} \simeq  G_y$, obtaining the inclusion ''$\subset$'' in \eqref{eq:9.6.2017a}. In any of the cases $y \in \O_x$ or $y \not \in \O_x$ we see that $\Crit \, \Phi^\J_{x,y,\omega'}$ is given by \eqref{eq:9.6.2017a} and, in particular, a smooth submanifold  of codimension $2\kappa$. 
To see that it is clean, note that with the notation and  the arguments in the proof of the previous theorem the transversal Hessian of $\Phi^\J_{x,y,\omega'}$ is constant on its critical set, and can be represented by the $(2\kappa\times 2\kappa)$--matrix
\bqn
 \left ( \begin{array}{c | c} \vspace{3mm} \norm{x-g_\J \cdot y} \mathrm{diag}(\rho_1,\dots,\rho_\kappa) & \Big \langle{\frac{\gd}{\gd\alpha'_i}}_{|\omega'},\widetilde X_{j,g_\J\cdot y}\Big \rangle   \\ \hline \\  \Big \langle {\frac{\gd}{\gd\alpha'_j}}_{|\omega'},\widetilde X_{i,g_\J\cdot y } \Big \rangle &  -\frac 12 \Big ({\widetilde X_{i,g_\J\cdot y} \big (\langle  \widetilde X_j,\omega'}\rangle\big )+ {\widetilde X_{j,g_\J\cdot y} \big (\langle  \widetilde X_i,\omega'}\rangle \big ) \Big ) 
 \end{array} \right  )
\eqn
with $\omega'=F_\J(0,\omega')$, where  the vector fields $\{\widetilde X_1, \dots, \widetilde X_\kappa\}$ constitute an orthonormal basis of $T_{g_\J\cdot y} \O_y$ at $g_\J\cdot y$, while the vector fields $\{\widetilde X_{\kappa+1}, \dots, \widetilde X_d\}$ vanish at $g_\J \cdot y$. As in \eqref{eq:8.6.2017} one then computes
\begin{align}
\label{eq:9.6.2017}
\begin{split}
\det \mathrm{Trans \, Hess} \,  {\Phi^\J_{x,y,\omega'}}=& c_0+\norm{x-g_\J \cdot y} c_1+ \dots + \norm{x-g_\J \cdot y}^\kappa c_\kappa 
\end{split}
\end{align}
with $c_i \in \R$ and $c_0\not=0$ uniformly in $x$ and $y$, showing the cleanness of $\Crit \, \Phi^\J_{x,y,\omega'}$ for sufficiently small $Y$. Now, applying Theorem \ref{thm:SP} to the inner integral in \eqref{eq:9.8.2017} we obtain
\begin{align*}
I^\J_{x,y}(\mu)= \mu^{-\kappa} \int_{N_{g_\J \cdot y} \O_y\cap \, \U_\J}  e^{i\mu \, ^0\Phi^\J_{x,y}(\omega')} \left [  \sum_{k_1=0}^{\tilde N_1-1} \mathcal{Q}_{\J,k_1}(x,y,\omega') \mu^{-k_1}+\mathcal{R}_{\J,\tilde N_1}(x,y,\omega',\mu) \right ] \d \omega',
\end{align*}
where  $^0\Phi^\J_{x,y}(\omega'):=\eklm{\kappa(x)-\kappa(g_\J\cdot y), \omega'}$ stands for the constant value of $\Phi^\J_{x,y,\omega'}$  on its critical set, and the coefficients and the remainder satisfy the usual estimates. Note that they are bounded from above by negative powers of the determinant of the transversal Hessian of $\Phi^\J_{x,y,\omega'}$. But since by \eqref{eq:9.6.2017} the determinant  of $\mathrm{Trans \, Hess} \,  {\Phi^\J_{x,y,\omega'}}$ is uniformly bounded away from zero in $x$ and $y$, the coefficients $\mathcal{Q}_{\J,k_1}(x,y,\omega')$ and the remainder $\mathcal{R}_{\J,\tilde N_1}(x,y,\omega',\mu)$ must be uniformly bounded in $x$ and $y$. Furthermore, they are bounded by derivatives with respect to $(\alpha',g)$ of the amplitude $a_\J$ up to order $2k_1$ and $2\tilde N_1+\kappa+1$, or, if the amplitude factorizes according to  $a(x,y,\omega,g)=a_1(x,y,\omega) \, a_2(x,y,g)$, by derivatives up to order $2 \tilde N_1 + \lfloor \kappa/2+1\rfloor$, respectively, compare Remark \ref{rem:25.07.2017}. If $y \in \O_x$, we have $^0\Phi^\J_{x,y}(\omega')=0$ for all $\omega'$, and we recover the asymptotic expansion for $I_{x,y}(\mu)$ derived in Theorem \ref{thm:12.05.2015} (a). Let us therefore assume that $y \not\in \O_x$. If $\kappa=n-1$, the set $N_{g_\J \cdot y} \O_y\cap \, \U_\J$ consists only of $\omega_\J$, and we are done. If  $\kappa<n-1$, in order to apply the stationary phase principle  to the integral over $\omega'$, observe that 
\bqn
\Crit \, ^0\Phi^\J_{x,y}=\mklm{\omega' \in N_{g_\J \cdot y} \O_y\cap \, \U_\J\mid \, \kappa(x)-\kappa(g_\J\cdot y) \in N_{\omega'} \Sigma^{R,t}_x}
=\mklm{\omega_\J},
\eqn
since  $\kappa(x)-\kappa(g_\J\cdot y) \perp V_{\omega'}$ by \eqref{eq:SGD}, so that 
\begin{align*}
\grad_{\omega'} \, ^0\Phi^\J_{x,y}=0 \quad &\Longleftrightarrow \quad   
\kappa(x)-\kappa(g_\J\cdot y) \in N_{\omega'} \Sigma^{R,t}_x.
\end{align*} 
Further,  the Hessian of  $ ^0\Phi^\J_{x,y}$ at the critical point $\omega_\J$ is given by
\bq
\label{eq:9.6.2017b}
\mathrm{Hess } \, ^0\Phi^\J_{x,y}(\omega_\J)= -\norm{\kappa(x)-\kappa( g_\J \cdot y) } \, \mathrm{diag}(\rho_{\kappa+1},\dots,\rho_{n-1})_{|\omega_\J}.
\eq
Consequently, $\omega_\J$ is a non-degenerate critical point  due to the strict convexity of  $\Sigma^{R,t}_x$. In order to find an interpolation formula we proceed as in Appendix \ref{appendix}, and write
\bqn 
e^{i\mu \, ^0\Phi^\J_{x,y}(\omega')}=e^{i(\norm{\kappa(x)-\kappa(g_\J\cdot y)} \mu+1) \, ^0\Psi^\J_{x,y}(\omega')}\, e^{-i  \, ^0\Psi^\J_{x,y}(\omega')}, \quad \Psi_{x,y}(\omega,g):=\frac{\Phi_{x,y}(\omega,g)}{\norm{\kappa(x)-\kappa(g\cdot y)}}.
\eqn
Theorem \ref{thm:SP} then implies for $\kappa < n-1$ with $^0\Psi^\J_{x,y}$ as phase function and $\mu\norm{\kappa(x)-\kappa(g_\J \cdot y)}+1$ as asymptotic parameter the asymptotic expansion
\begin{gather*}
\int_{N_{g_\J \cdot y} \O_y\cap \, \U_\J}   \, e^{i\mu \, ^0\Phi^\J_{x,y}(\omega')}  \mathcal{Q}_{\J,k_1}(x,y,\omega') \d \omega'=(\mu\norm{\kappa(x)-\kappa(g_\J \cdot y)}+1)^{-(n-\kappa-1)/2} e^{i\mu \, ^0\Phi^\J_{x,y} }\\ 
\cdot \left [ \sum_{k_2=0}^{\tilde N_2-1} (\mu\norm{\kappa(x)-\kappa(g_\J \cdot y)}+1)^{-k_2} \mathcal{Q}_{\J,k_1,k_2}(x,y,\omega_\J)+\mathcal{R}_{\J,k_1,\tilde N_2}(x,y,\mu) \right ],
\end{gather*}
where $ ^0\Phi^\J_{x,y}$ denotes the constant value of $^0\Phi^\J_{x,y}(\omega')$ at $\omega_\J$, which equals the constant value of $\Phi_{x,y}$ at $(\omega_\J,g_\J)\in \J$. Note that the coefficients $\mathcal{Q}_{\J,k_1,k_2}(x,y,\omega_\J)$ and the remainder  $\mathcal{R}_{\J,k_1,\tilde N_2}(x,y,\mu)$ are uniformly bounded in $x,y$ in view of \eqref{eq:9.6.2017b}, and are not bounded by additional derivatives with respect to $g$.  Treating the remainders $\mathcal{R}_{\J,\tilde N_1}$ alike, the theorem follows.
\end{proof}

To close this section, let us still consider the case of a finite group $G$. For this, one has to examine the asymptotic behavior of  oscillatory integrals of the form 
\bq
\label{eq:23.06.2016}
I_{z}(\nu):=  \int_{\Sigma} e^{i\nu \Psi_z(\omega)} a ( \omega) \d \Sigma (\omega)  , \qquad  z \in S^{n-1}, \quad \nu \to + \infty,
\eq
where  $\Sigma \subset \R^n$ denotes a strictly convex $\Cinft$-hypersurface, $d\Sigma$ the induced volume density, and $\Phi_z$ the phase function 
$
\Psi_z(\omega):=  \eklm{z, \omega},
$
 while   $a \in \CT(\Sigma)$ is an amplitude that might depend on $\nu$  and other parameters. 
  
 \begin{lem}
\label{thm:24.08.2016} 
For every   $ \tilde N \in \N$ one has the asymptotic formula
\begin{align*}
  I_{z}(\nu)&=  \sum_{{\omega_0 \in \Crit \Psi_z}} \frac{e^{i\nu\Psi_z(\omega_0)}  }{(\det (\nu\,   \mathrm{II}_{\omega_0}/2\pi i ))^{1/2}} \left [ \sum_{j=0}^{\tilde N-1} \mathcal{Q}_{j}(a,\Psi_z,\omega_0) \nu^{-j}  +\mathcal{R}_{\tilde N}(a,\Psi_z,\omega_0,\nu)\right ]
\end{align*}
as $\mu \to +\infty$, where the critical set of $\Psi_z$ is given by 
\bqn 
\Crit \, \Psi_z=\Big \{\omega \in \Sigma \mid   z \in N_\omega \Sigma\Big \},
\eqn
and only consists of non-degenerate, isolated points, while $\mathrm{II}$ denotes the second fundamental form of $\Sigma$. The coefficients  and the remainder satisfy the bounds
\begin{align*}
\begin{split}
|\mathcal{Q}_j| \leq C_{j}  \sup _{l \leq 2j}|D^l a(\omega_0)|, \qquad 
|\mathcal{R}_{\tilde N} | \leq \tilde C_{\tilde N} \sup _{l \leq \lfloor (n-1)/2 +1\rfloor+ 2\tilde N}\norm{D^l a}_{\infty, \Sigma}  \mu^{- \tilde N} 
\end{split}
\end{align*}
for suitable constants $C_{j}, \tilde C_{\tilde N}>0$ independent of $z$ and $\nu$, where $D^l$ denotes a differential operator on $\Sigma$ of order $l$.  In particular, 
\bqn
\mathcal{Q}_{0}(a,\Psi_z,\omega_0)=a(\omega_0).
\eqn
\end{lem}
 
\begin{proof}
The statement of the proposition is essentially known \cite[Theorem 7.7.14]{hoermanderI}, but for completeness, we include a proof here. 
Consider a local parametrization 
\bq
\label{eq:param}
F:\R^{n-1} \supset U \, \longrightarrow \, \Sigma\subset \R^n , \quad \xi \longmapsto F(\xi)=\omega,
\eq
of the hypersurface $\Sigma$. If we compute the derivatives of $\Psi_z$ with respect to this parametrization and set them equal to zero, we arrive at the conditions $\eklm{z, \gd F/\gd \xi_i}=0$ for $i=1,\dots, n-1$, implying  that $z$ must be normal to $\Sigma$ at $\omega$. Thus, $\Crit \, \Psi_z=\{\omega \in \Sigma\mid  z \in N_\omega \Sigma\}$.   Since $\Sigma$ is strictly convex,  the Gauss map 
$
\Ncal: \Sigma \ni \omega \longmapsto \Ncal(\omega) \in N_\omega \Sigma
$
 is a global diffeomorphism, so that for each $\tilde z \in S^{n-1}$ there is a unique $\omega_{\tilde z} \in \Sigma$ such that $\tilde z=\Ncal(\omega_{\tilde z})$.
Consequently,  $\omega \in \Crit \, \Psi_z$ is locally uniquely determined by the condition $\Ncal(\omega) =\pm \, \Ncal(\omega_{z})$, so that $\omega$ is an isolated point. In order to see  that $\Crit \,  \Psi_z$ consists of non-degenerate points, note that with respect to the parametrization  \eqref{eq:param} of $\Sigma$  the Hessian of $\Psi_z$ at a critical point $\omega$ is  given by the matrix
\bq
\label{eq:27.06.2016}
\mathrm{Hess} \, \Psi_z(\omega)\equiv \left ( \eklm{z, \frac{\gd ^2 F}{\gd \xi_i \gd \xi_j}(F^{-1}(\omega))}_{1\leq i,j \leq n-1}   \right  ).
\eq
Since $z\in N_\omega\Sigma$, $\mathrm{Hess} \, \Psi_z(\omega)$ corresponds to the second fundamental form $\mathrm{II}$ of $\Sigma$, compare  
\cite[Chapter VII, Section 3]{kobayashi-nomizuII}.
 Because $\Sigma$ is strictly convex,  the eigenvalues of $\mathrm{II}$  at  $\omega$, which are given by the principal curvatures of $\Sigma$ at that point, are all non-zero. Therefore, the determinant of $\mathrm{Hess} \, \Psi_z(\omega)$ is non-zero, and $\omega$ must be a non-degenerate critical point. In conclusion,   $\Psi_z$ has a clean critical set, so that the asymptotic formula for $I_z(\nu)$, together with the estimates for $\mathcal{Q}_j$ and  the remainder, follow directly  by applying Theorem \ref{thm:SP} to $I_z(\nu)$. 
\end{proof}

We now have the following 

\begin{proposition}
\label{prop:10.08.2017} 
Consider the integrals $I_{x,y}(\mu)$  defined in \eqref{eq:03.05.2015}. Assume that $G$ is finite,  and that the co-spheres $S_x^\ast M$ are strictly convex.  Then, for arbitrary $\tilde N \in \N$   one has the asymptotic formula
\begin{gather*}
 I_{x,y}(\mu)\\
 = \sum_{\stackrel{g \in G,\, \omega_0 \in }{\Crit \, \Phi_{x,y,g}}}   \frac{e^{i \mu  \Phi_{ x,y,g}(\omega_0)}}{(\mu \norm{ \kappa(x)-\kappa(g \cdot y)}+1 )^{\frac{n-1}2}} \left [ \sum_{k=0}^{\tilde N-1}  \frac {   \mathcal{Q}_{\omega_0,k} (x,y,g)}{ (\mu \norm{ \kappa(x)-\kappa(g \cdot y)}+1)^{k}}+ \mathcal{R}_{\omega_0,\tilde N}(x,y,g,\mu) \right ]
\end{gather*}
as $\mu \to +\infty$, where  $\Phi_{x,y,g}(\omega):=\Phi_{x,y}(\omega,g)$.   The coefficients  and the remainder term 
\bqn
\mathcal{R}_{\omega_0, \tilde N}(x,y,g,\mu)=O\Big ((\mu \norm{ \kappa(x)-\kappa(g \cdot y)}+1)^{-\tilde N}\Big )
\eqn
 are explicitly given and depend smoothly on $R,t$. Furthermore, they are uniformly bounded in $x$ and $y$. 
\end{proposition}

\begin{proof}
For finite $G$, the integral \eqref{eq:03.05.2015} reads
\bqn
I_{x,y}(\mu):=\sum_{g \in G} I_{x,y,g}(\mu), \qquad I_{x,y,g}(\mu):= \int_{\Sigma^{R,t}_x} e^{i\mu \Phi_{x,y,g}(\omega)} a (x,y, \omega,g) \d \Sigma^{R,t}_x (\omega).
\eqn
Assume that $x\not=g\cdot y$. Writing
\bqn 
e^{i\mu \, \Phi_{x,y,g}(\omega)}=e^{i(\norm{\kappa(x)-\kappa(g\cdot y)} \mu+1) \, \Psi_{x,y,g}(\omega)}\, e^{-i  \, \Psi_{x,y,g}(\omega)}, \quad \Psi_{x,y,g}(\omega):=\frac{\Phi_{x,y,g}(\omega)}{\norm{\kappa(x)-\kappa(g\cdot y)}},
\eqn
 the previous lemma implies with $\nu=\norm{\kappa(x)-\kappa(g\cdot y)} \mu+1$ as asymptotic parameter and $\Psi_{x,y,g}$ as phase function the expansion
\begin{align*}
I_{x,y,g}(\mu)&=  \sum_{{\omega_0 \in \Crit \Psi_{x,y,g}}} \frac{e^{i\nu\Psi_{x,y,g}(\omega_0)}  }{(\det (\nu\,   \mathrm{II}_{\omega_0}/2\pi i ))^{1/2}} \Big [ \sum_{j=0}^{\tilde N-1} \mathcal{Q}_{j}(a (x,y,\cdot,g)e^{-i  \, \Psi_{x,y,g}},\Psi_{x,y,g},\omega_0) \, \nu^{-j}  \\ 
&+\mathcal{R}_{\tilde N}(a(x,y,\cdot,g) e^{-i  \, \Psi_{x,y,g}},\Psi_{x,y,g},\omega_0,\nu)\Big ]
\end{align*}
as $\mu \to +\infty$, where 
\bqn 
\Crit \, \Psi_{x,y,g}=\Big \{\omega \in \Sigma^{R,t}_x \mid   \kappa(x)-\kappa(g\cdot y) \in N_\omega \Sigma^{R,t}_x\Big \}
\eqn
and all expressions are uniformly bounded in $x$ and $y$. If $x=g\cdot y$, 
\bqn 
I_{x,y,g}(\mu)= \int_{\Sigma^{R,t}_x}  a (x,y, \omega,g) \d \Sigma^{R,t}_x (\omega),
\eqn
and the assertion of the proposition follows by setting $\Crit \, \Phi_{x,y,g}:= \Sigma^{R,t}_x$ and replacing the sum over $\Crit \, \Phi_{x,y,g}$ by an integral over $\Sigma^{R,t}_x$ in this case.
\end{proof}

\section{The equivariant local Weyl law}

Let us now come back to our initial question of finding an asymptotic description of the equivariant spectral function. With the notation of the previous sections we have 

\begin{proposition} [\bf Point-wise  asymptotics for the kernel of the equivariant approximate projection]
\label{thm:kernelasymp}
For any  fixed $x \in M$, $\gamma \in \widehat G$,  and $\tilde N\in \N$ one has as $\mu \to +\infty$
\begin{align}
\label{eq:13.05.2015}
\begin{split}
K_{\widetilde \chi_\mu \circ \Pi_\gamma}(x,x)&=\sum_{j\geq 0, \, e_j \in \L^2_\gamma(M)} \rho(\mu-\mu_j) |{e_j(x)}|^2 \\ & = \Big (\frac{\mu}{2\pi}\Big )^{n-\dim \mathcal{O}_x-1} \frac{d_\gamma}{2\pi} \left [\sum_{k=0}^{\tilde N-1} \Lcal_k(x,\gamma) \mu^{-k}+ \mathcal{R}_{\tilde N}(x,\gamma) \right ]
\end{split}
\end{align}
with known coefficients  and  remainder  that depend smoothly  on $x \in M_\mathrm{prin}$. If $G$ is continuous, they  satisfy the bounds
\bqn 
|\mathcal{L}_k(x,\gamma)| \leq C_{k,x} \sup_{l \leq 2k} \norm{D^l \gamma}_\infty, \qquad |\mathcal{R}_{\tilde N}(x,\gamma)| \leq \tilde C_{\tilde N,x} \sup_{l \leq 2\tilde N+\lfloor \dim \O_x/2+1\rfloor} \norm{D^l \gamma}_\infty \mu^{-\tilde N},
\eqn
where $D^l$ denotes a differential operator on $G$ of order $l$, and the constants $C_{k,x}$, $\tilde C_{\tilde N,x}$ are uniformly bounded in $x$ if $M= M_\mathrm{prin} \cup M_\mathrm{except}$; if $G$ is finite, similar bounds hold with $l=0$. In particular, the leading coefficient is given by
\begin{align*}
\Lcal_0(x,\gamma) =  \hat \rho(0) [{\pi_\gamma}_{|G_x}:\1] \, \mbox{vol} \, [( \Omega \cap S_x^\ast M)/G],
\end{align*}
where $S^\ast M:=\mklm{(x,\xi) \in T^\ast M\mid p(x,\xi)=1}$, which for finite $G$ simply reads 
\bqn 
\hat \rho(0) \sum_{g \in G_x} \overline{\gamma(g)}  \, \mbox{vol} \, [( S_x^\ast M)/G].
\eqn
If $\mu \to -\infty$, the function $K_{\widetilde \chi_\mu \circ \Pi_\gamma}(x,x)$ is rapidly decreasing in $\mu$. 
\end{proposition}
\begin{proof}
Let the notation be as in  Corollary  \ref{cor:12.05.2015}, and $R,t \in \R$, $x \in Y_\iota$ be fixed.  If $G$ is continuous, one deduces as a direct consequence of Theorem \ref{thm:12.05.2015} (a)  for any $\tilde N\in \N$
\bqn 
 \gd_{R,t}^\beta I^\gamma_\iota(\mu, R, t, x,x)= (2\pi/\mu)^{\dim \mathcal{O}_x}\sum_{k=0}^{\tilde N-1} \mathcal{Q}^k_{\iota,\beta}(R,t,x,\gamma) \mu^{-k}+ O_{R,t,x,\gamma}(\mu^{-\dim \mathcal{O}_x-\tilde N}),
\eqn
where the coefficients  and the remainder term are explicitly given by distributions depending smoothly on $R,t$, and $x\in Y \cap M_\mathrm{prin}$.  Furthermore, both the coefficients $\mathcal{Q}^k_{\iota,\beta}(R,t,x,\gamma)$  and the remainder are bounded by expressions involving derivatives of $\gamma$ up to order $2k$ and $2\tilde N +\lfloor \dim \O_x/2+1\rfloor$, respectively, which are uniformly bounded in $x$ if $M=M_\mathrm{prin} \cup M_\mathrm{except}$.  Note  that  $\Phi_{\iota,x}$ vanishes on its critical set $\Crit_{R,t} \, \Phi_{\iota,x}:=(\Omega \cap \Sigma^{R,t}_{\iota,x}) \times G_x$ no matter what values  $R$ and $t$ take. Otherwise differentiation with respect to $R$ and $t$ of the factor $e^{i\mu\psi_0}$ in \eqref{eq:07.10.2015} with $\psi_0\equiv\Phi_{\iota,x|\Crit_{R,t} \, \Phi_{\iota,x}}$  would yield additional positive powers of $\mu$. Furthermore,  $a_\iota \in \mathrm{S}^0_{phg}$ is a classical symbol of order $0$, so that 
\bqn 
\big | \gd ^\alpha_\omega a_\iota ( t, \kappa _\iota (x), \mu \omega) \big | =|\mu|^{|\alpha|} \big | (\gd^\alpha_\omega a _\iota)( t, \kappa_\iota(x), \mu \omega) \big | \leq C |\omega|^{-|\alpha|}.
\eqn
Consequently, the  dependence of the amplitude on $\mu$ in \eqref{eq:02.05.2015} does not interfer with the asymptotics, compare \cite[Proposition 1.2.4]{duistermaatFIO}. Corollary \ref{cor:12.05.2015} then  implies the asymptotic expansion \eqref{eq:13.05.2015} with  
\bqn
\Lcal_0(x,\gamma)=   \sum _\iota {f_\iota(x)  \hat \rho(0)  } \int_{\Crit_{1,0} \, \Phi_{\iota,x} } \frac {    \overline{\gamma(g)}  }{|\det   \, \Phi''_{\iota,x} (\omega,g)_{N_{(\omega, g)}\Crit_{1,0} \, \Phi_{\iota,x} }|^{1/2}} \d(\Crit_{1,0} \, \Phi_{\iota,x})(\omega,g),
\eqn
since  $\alpha(q( x, \omega))=1$ on $\Sigma_{\iota,x}^{1,0}$ and $J_\iota(g,x)=1$ for $g \in G_x$. 
In order to compute $\Lcal_0(x,\gamma)$, let us note that for any $x \in Y_\iota$ and  smooth, compactly supported function $f$ on $\Omega  \cap \Sigma^{R,t}_{\iota,x}$ one has the formula
\begin{align*}
\int_{\Crit_{R,t} \, \Phi_{\iota,x} }&\frac{\overline {\gamma(g)}  f(x,\omega)}{|\det  \, \Phi_{\iota,x}'' (\omega,g)_{|N_{(\omega,g)} \Crit_{R,t} \, \Phi_{\iota,x} } |^{1/2}} d(\Crit_{R,t} \, \Phi_{\iota,x})( \omega,g) \\ &
=[{\pi_\gamma}_{|G_x}:\1]\int_{\Omega\cap \Sigma^{R,t}_{\iota,x}} f (x, \omega) \frac{d( \Omega\cap \Sigma^{R,t}_{\iota,x})(\omega)}{\mbox{vol }\mathcal{O}_{(x,\omega)}},
\end{align*}
where we took into account that $\int_{G_x} \overline{\gamma(g)} \d G_x(g)=[{\pi_\gamma}_{|G_x}:\1]$,  compare \cite[Lemma 7]{cassanas-ramacher09}, \cite[Proof of Theorem 9.5]{ramacher10}, and \cite[Section 3.3.2]{cassanas}, the map $\Crit_{R,t} \, \Phi_{\iota,x}\to  \Omega\cap \Sigma^{R,t}_{\iota,x}$ being a submersion. As a consequence of this, we obtain for $\Lcal_0(x)$ the expression
\begin{align*}
\Lcal_0(x,\gamma)&=   \hat \rho(0) [{\pi_\gamma}_{|G_x}:\1] \sum_\iota f_\iota(x) \int_{\Omega\cap \Sigma^{1,0}_{\iota, x}}  \frac{d( \Omega\cap \Sigma^{1,0}_{\iota,x})(\omega)}{\mbox{vol }\mathcal{O}_{(x,\omega)}} = \hat \rho(0)  [{\pi_\gamma}_{|G_x}:\1] \, \mbox{vol} \, [(\Omega \cap S^\ast_x M)/G].
\end{align*}
The case when $G$ is finite can be deduced from Proposition \ref{prop:10.08.2017} in an analogous way, since then $\Omega=T^\ast M$. 

\end{proof}

\begin{rem}
\label{rem:12.7.2017} Note that, if $M= M_\mathrm{prin} \cup M_\mathrm{except}$,   the previous proposition and  the Cauchy-Schwarz inequality imply  for $\tilde N=0$ with $ \kappa :=\dim G/K$    the estimate
\begin{align*}
\begin{split}
K_{\widetilde \chi_\mu \circ \Pi_\gamma}(x,y)&\leq  \sqrt{K_{\widetilde \chi_\mu \circ \Pi_\gamma}(x,x)} \sqrt{K_{\widetilde \chi_\mu \circ \Pi_\gamma}(y,y)}\ll\, \mu^{n-\kappa-1} \, d_\gamma \, \sup_{l \leq  \lfloor \kappa/2+1\rfloor} \norm{D^l \gamma}_\infty,
\end{split}
\end{align*}
uniformly in $x$ and $y$, $\rho\in \S(\R,\R_+)$ being a positive function. 
\end{rem}

Using a standard Tauberian argument, we can now deduce from Proposition \ref{thm:kernelasymp} our first main result.

\begin{thm}[\bf Equivariant local Weyl law]  
\label{thm:main}
Let $M$ be  a closed connected Riemannian manifold $M$ of dimension $n$ carrying an isometric and effective action of a compact   Lie group $G$, and $P_0$ a $G$-invariant elliptic classical pseudodifferential operator on $M$ 
of degree $m$. Let  $p(x,\xi)$ be its  principal symbol, and assume that $P_0$    is positive and symmetric. Denote its  unique self-adjoint extension by  $P$, and for a given $\gamma \in \widehat G$ let $e_\gamma(x,y,\lambda)$ be its reduced spectral  function. Further, let $\Jbb:T^\ast M \to \g^\ast$ be the momentum map of the $G$-action on $M$, and put $\Omega:=\Jbb^{-1}(\mklm{0})$.  Then, for fixed $x \in M$ one has
 \bq
 \label{eq:29.10.2015}
\left |e_\gamma(x,x,\lambda)-\frac{d_\gamma [\pi_{\gamma|G_x}:\1]}{(2\pi)^{n-{\kappa_x}}}  \lambda^{\frac{n-\kappa_x}{m}} \int_{\mklm{\xi\mid \, (x,\xi) \in \Omega, \, p(x,\xi)< 1}} \frac{ \d \xi}{\vol \O_{(x,\xi)}} \right | \leq C_{x,\gamma} \, \lambda^{\frac{n-{\kappa_x}-1}{m}}
\eq
as $\lambda \to +\infty$,  where $\kappa_x:=\dim \O_x$,  $d_\gamma$ denotes the dimension of an irreducible $G$-representation $\pi_\gamma$ belonging to $\gamma$ and $ [\pi_{\gamma|G_x}:\1]$ the multiplicity of the trivial representation  in the restriction of $\pi_\gamma$ to the isotropy group $G_x$ of $x$. If $G$ is continuous,  
\bq
\label{eq:4.6.2017}
C_{x,\gamma}=O_x\Big (d_\gamma \sup_{l \leq \lfloor \dim \O_x/2+3\rfloor} \norm{D^l \gamma}_\infty\Big )
\eq
 is a constant that depends smoothly on $x\in M_\mathrm{prin}$ and is uniformly bounded in $x$ if   $M= M_\mathrm{prin} \cup M_\mathrm{except}$; if $G$ is finite,  $C_{x,\gamma}=O(d_\gamma \norm{\gamma}_\infty )$.
\end{thm}
\begin{proof}
This follows directly  by taking $\tilde N=1$ in  \eqref{eq:13.05.2015} and integrating with respect to $\mu$ from $-\infty$ to $\sqrt[m]{\lambda}$ with  the arguments given in \cite[Proof of Eq. (2.25)]{duistermaat-guillemin75}.
\end{proof}

\begin{rem}
\label{rem:23.04.2017}
\hspace{0cm}
\begin{enumerate}
\item Note that in view of \eqref{eq:Ann} the integral in the leading term can also be written as
\bqn 
 \lambda^{\frac{n-\kappa_x}m}\int_{\mklm{\xi\mid \, (x,\xi) \in \Omega, \, p(x,\xi)< 1 }} \frac{ \d \xi}{\vol \O_{(x,\xi)}}=  \int_{\mklm{\xi\mid \, (x,\xi) \in \Omega, \, p(x,\xi)< \lambda^{1/m}}} \frac{ \d \xi}{\vol \O_{(x,\xi)}}.
 \eqn
\item 
The formula \eqref{eq:29.10.2015} is shown by proving first the estimate
\bq
\label{eq:21.06.2016}
\left | e_\gamma(x,x,\lambda+1)-e_\gamma(x,x,\lambda) \right | \leq C_{x,\gamma}  \, \lambda^{\frac{n-\kappa_x-1}m}, \qquad x \in M,
\eq
compare \cite[Lemma 2.3]{duistermaat-guillemin75}. Since $e_\gamma(x,y,\lambda+1)-e_\gamma(x,y,\lambda)$ is the kernel of a positive operator, one immediately infers from this with the Cauchy-Schwarz inequality the bound 
\bqn 
\left | e_\gamma(x,y,\lambda+1)-e_\gamma(x,y,\lambda) \right | \leq \sqrt{C_{x,\gamma} \lambda^{\frac{n-\kappa_x-1}m}} \sqrt{C_{y,\gamma} \lambda^{\frac{n-\kappa_y-1}m}}, \qquad x, y \in M. 
\eqn
From this, it is not difficult to deduce a corresponding equivariant local Weyl law for $e_\gamma(x,y,\lambda)$ in a neighborhood of the diagonal, see \cite[pp. 210]{hoermander68} or \cite[Section 21]{shubin}.\end{enumerate} 
\end{rem}

\begin{rem}
\label{rem:22.5.2017}
In case that $G$ is a connected compact semisimple Lie group, the bound \eqref{eq:4.6.2017} of the previous theorem can be rephrased using the Cartan-Weyl classification  of unitary irreducible representations of $G$. In fact, let $\g$ be its Lie algebra,  and $T\subset G$ a maximal torus with Lie algebra $\t$. Denote by $\g_\C$ and $\t_\C$  the complexifications of $\g$ and $\t$, respectively. Then $\t_\C$ is a Cartan subalgebra of $\g_\C$, and we write  $\Sigma(\g_\C,\t_\C)$ for the corresponding system  of roots and $\Sigma^+$ for a set of positive roots. Since any element in $G$ is conjugated to an element of $T$, a character $\gamma \in \widehat G$ is fully determined by its restriction to $T$. Now, as a consequence of the Cartan-Weyl classification of irreducible finite-dimensional representations of reductive Lie algebras over $\C$ one has the  identification 
\bqn
\widehat G\, \simeq\, \mklm{\Lambda \in \t_\C^\ast\mid \text{ $\Lambda$ is  dominant integral and $T$-integral}},
\eqn 
compare \cite{wallach}, and we write   $\Lambda_\gamma\in \t_\C^\ast$ for the {highest weight} corresponding to  $\gamma \in \widehat G$ given by  this isomorphism. Weyl's dimension formula then implies that $d_\gamma=O\big (|\Lambda_\gamma|  ^{|\Sigma^+|}\big )$, while from Weyl's character formula one infers that  if $D^l$  is a differential operator on $G$ of order $l$, 
\bqn
 \norm{D^l \gamma}_\infty = O\big (|\Lambda_\gamma|  ^{l+|\Sigma^+|}\big ), \qquad |\Lambda_\gamma|\to \infty,
\eqn
compare \cite[Eq.\ (3.5)]{ramacher10add}. Consequently, the bound \eqref{eq:4.6.2017} can be rewritten as
\[
C_{x,\gamma}=O_x\big (|\Lambda_\gamma|  ^{\lfloor \dim \O_x/2+3\rfloor +2|\Sigma^+|}\big ).
\]
\end{rem}

As a first consequence of Theorem \ref{thm:main}, let us note that the estimate \eqref{eq:21.06.2016} is equivalent to the following bound for spectral clusters. 
\begin{cor}  [\bf Point-wise bounds for isotypic spectral clusters] 
\label{cor:21.06.2016}
In the situation of Theorem \ref{thm:main} we have
\bqn 
\sum_{\stackrel{\lambda_j \in (\lambda,\lambda+1],}{ e_j \in \L^2_\gamma(M)}} |e_j(x)|^2 \leq C_{x, \gamma} \, \lambda^{\frac{n-\kappa_x-1}m}, \qquad x \in M,
\eqn
where $\mklm{e_j}$ denotes an orthonormal basis of $\L^2(M)$ consisting of eigenfunctions of $P$ with eigenvalues $\mklm{\lambda_j}$. 
\end{cor}
\qed

A further implication of Theorem \ref{thm:main} is the following  Kuznecov sum formula for periods of $G$-orbits, which generalizes the classical Kuznecov formula for periods of closed geodesics \cite{zelditch92}.

\begin{cor}[\bf Generalized Kuznecov sum formula for periods of $G$-orbits]
\label{cor:kuznecov} In the setting of Theorem \ref{thm:main} we have
 \bqn
\left |\sum_{\lambda_j \leq \lambda} \left | \int_G e_j(g^{-1} \cdot x) \d g \right |^2-\frac{\vol G_x}{(2\pi)^{n-{\kappa_x}}}  \lambda^{\frac{n-\kappa_x}{m}} \int_{\mklm{\xi\mid \, (x,\xi) \in \Omega, \, p(x,\xi)< 1}} \frac{ \d \xi}{\vol \O_{(x,\xi)}} \right | \leq C_{x} \, \lambda^{\frac{n-{\kappa_x}-1}{m}}
\eqn
for some constant $C_x>0$ depending on $x$.
\end{cor}
\begin{proof}
Let  $\gamma=\gamma_\text{triv}$ correspond to the trivial representation. Then
\bqn 
e_{\gamma_\text{triv}}(x,x,\lambda) = \sum_{\lambda_j \leq \lambda, \, e_j \in \L^2_{\gamma_\text{triv}}(M) } |e_j(x)|^2=\sum_{\lambda_j \leq \lambda} \left | \int_G e_j(g^{-1} \cdot x) \d g \right |^2,
\eqn 
and the assertion follows from the previous theorem with $d_{\gamma_\text{triv}}=1$ and 

$$[{\pi_\gamma}_{\text{triv}|G_x}:\1]=\int_{G_x} \overline{\gamma_\text{triv}(g)} \d G_x(g)=\vol G_x.$$
\end{proof}

In case that $\widetilde M:=M/G$ is an orbifold  we essentially recover the description of the spectral function of a Riemannian orbifold given by Stanhope and Uribe in \cite{stanhope-uribe}. More precisely, we  infer

\begin{cor}[\bf Local Weyl law for Riemannian orbifolds]
\label{cor:orbifold} In the situation of Theorem \ref{thm:main}, assume that $G$ acts on $M$ with finite isotropy groups.  Then, for fixed $x \in M$ and $\gamma \in \widehat G$ the asymptotic formula \eqref{eq:29.10.2015} holds with $n-\kappa_x\equiv n-\kappa$ being equal to  the dimension of $\widetilde M$. Moreover,   let $\gamma_\text{triv}$ be the trivial representation. Then $e_{\gamma_\text{triv}} (x,x,\lambda)$ is $G$-invariant, and descends to a function on $\widetilde M \times \widetilde M$ satisfying
\bqn 
\left |e_{\gamma_\text{triv}} (\tilde x,\tilde x,\lambda)-\frac{|G_{\tilde x}|}{(2\pi)^{\dim \widetilde M}} \lambda ^{\frac {\dim \widetilde M} m} \vol (S^\ast_{\widetilde p,\tilde x}(\widetilde M)) \right | \leq C_{\tilde x} \lambda^{\frac{\dim \widetilde M-1}{m}}, \qquad \tilde x \in \widetilde M, 
\eqn
where   $(G_{\tilde x})$ denotes the isotropy type of $\tilde x:=G \cdot x $, $|G_{\tilde x}|$ its cardinality, while $S^\ast_{\widetilde p,\tilde x}(\widetilde M)$ equals the 
fiber over $\tilde x$ of the orbifold bundle $S^\ast_{\widetilde p}(\widetilde M):=\mklm{(\tilde x,\xi) \in T^\ast \widetilde M\mid \widetilde p(\tilde x,\xi)=1}$, $\widetilde p$ being the function on $\widetilde M$ induced by $p$.
\end{cor}
\begin{proof} The first assertion is clear, since all $G$-orbits on $M$ have the same dimension $\kappa$, so that no singular orbits are present. To see the second note that since $G_x$ is finite,  one computes
\bqn 
[{\pi_{\gamma_\text{triv}}}_{|G_x}:\1]=\int_{G_x} \overline{\gamma_\text{triv}(g)} dG_x(g)=\sum_{l=1}^{|G_x|} 1 =|G_x|,
\eqn
$dG_x$ being the counting measure. For the volume factor, see \cite[Remark 6.2]{kuester-ramacher15}.
\end{proof}

\begin{ex}
\label{ex:torus}
Let us consider the case where $M=T^2\subset \R^3$ is the standard $2$-torus on which $G=\SO(2)$ acts by rotations around the symmetry axis. Then all orbits are $1$-dimensional and of principal type, and Theorem \ref{thm:main} yields with the identification $\Z\simeq \widehat{\SO(2)}$ for the reduced spectral function of the Laplace-Beltrami operator 
 \bqn
e_m(x,x,\lambda)-\frac{1}{2\pi} \sqrt \lambda  \int_{\mklm{\xi\mid \, (x,\xi) \in \Omega, \, p(x,\xi)< 1}} \frac{ \d \xi}{\vol \O_{(x,\xi)}} =O(1+|m|^3),  \qquad m \in \Z,
\eqn
uniformly in $x \in T^2$, the irreducible characters of $\SO(2)$ being given by the exponentials $\theta \mapsto e^{im\theta}$, $\theta \in [0,2\pi)\simeq \SO(2)$, $m \in \Z$. 
\end{ex}

\begin{ex}
\label{ex:cocpt}
Consider a connected semisimple Lie group $G$ with finite center  and Lie algebra $\g$, together with   a discrete  co-compact subgroup  $\Gamma$. In particular, $\Gamma$ might have torsion, meaning that there are non-trivial elements of $\Gamma$  conjugate in $G$  to an element of $K$. Let $K$ be a maximal compact subgroup of $G$, and  choose a left-invariant metric on $G$ given by an $\Ad(K)$-invariant bilinear form on  $\g$.
The quotient  $M:=\Gamma\backslash G$ is a compact manifold without boundary,  and has a  Riemannian structure induced by the one of $G$.  Furthermore, 
$K$ acts  on $\Gamma\backslash G$ from the right  in an isometric  and effective way, and  the isotropy group of a point $\Gamma g$ is conjugate 
to the finite group $gKg^{-1}\cap \Gamma$. Hence, all $K$-orbits in $\Gamma\backslash G$ are either principal or exceptional, $\Gamma \backslash G /K$ is an orbifold,  and Corollary \ref{cor:orbifold} applies. 
\end{ex}

\begin{ex}
\label{ex:sphere}
Let us now consider a case where singular orbits are present, and   $M=S^2\subset \R^3$ be the standard $2$-sphere on which $G=\SO(2)\subset \SO(3)$  acts by rotations around the $x_3$-axis with fixed points $x_N=(0,0,1)$ and $x_S=(0,0,-1)$. In this case the phase function of $I_x(\mu)$ reads $\Phi_x(\omega,g)=\eklm{ x-g\cdot x,\omega} $ with respect to standard coordinates in $\R^3$. For $x=x_N,x_S$ it simply vanishes, so that $I_x(\mu)$ is independent of $\mu$ in this case, which is consistent with the asymptotics 
\bqn 
I_x(\mu)= \begin{cases}  O(\mu^0), & x=x_N,x_S, \\ O(\mu^{-1}), & \text{otherwise,}\end{cases}
\eqn
implied by Theorem \ref{thm:12.05.2015}. Let us  now apply Theorem \ref{thm:main}  to  the Laplace-Beltrami operator $-\Delta$ on $S^2$, and notice for this that the orbit volume $\vol \O_{(x,\xi)}$ is of order $\sqrt{\xi_1^2+\xi_2^2}+\sqrt{x_1^2+x_2^2}$ for arbitrary $x$ and $\xi$.
By Theorem \ref{thm:main} and with the identification $\widehat{\SO(2)}\simeq \Z$   the reduced spectral function satisfies on $ S^2_\mathrm{prin}= S^2-\mklm{x_N,x_S}$ the estimate 
 \bq
 \label{eq:08.11.2015}
\left |e_m(x,x,\lambda)-\frac{  \sqrt\lambda}{2\pi}  \int_{\mklm{\xi\mid \, (x,\xi) \in \Omega, \, \norm{\xi}_x< 1}} \frac{ \d \xi}{\vol \O_{(x,\xi)}} \right | \leq C_x\,(1+|m|^3) , \qquad x \in S^2_\mathrm{prin}, \, m \in  \Z.
\eq
In this case, $\Omega \cap T^\ast_x(S^2)$ is $1$-dimensional; the integral in \eqref{eq:08.11.2015} is  finite, but as $S^2_\text{prin} \ni x \to x_N$ or $x_S$ the orbit volume becomes of order $\sqrt{\xi_1^2+\xi_2^2}$, so that the mentioned integral goes to infinity. On the other hand, for the fixed points  $x=x_N,x_S$ the space  $\Omega \cap T^\ast_xS^2=T^\ast _xS^2$ is $2$-dimensional and Theorem \ref{thm:main} yields
 \bq
 \label{eq:31.10.2015}
\left |e_m(x,x,\lambda)-\frac{  [\pi_{m|G}:\1]  }{(2\pi)^2} \lambda \int_{\mklm{\xi\mid  \, \norm{\xi}_x< 1}} \frac{ \d \xi}{\vol \O_{(x,\xi)}} \right | \leq C_x \, (1+|m|^3) \,  \sqrt \lambda, \qquad x=x_N,x_S, \, m \in \Z, 
\eq
where
\bqn 
[\pi_{m|G}:\1]=\begin{cases} 1, & m=0, \\ 0, & \text{otherwise.} \end{cases} 
\eqn
Thus, at the fixed points only  the trivial representation contributes to the main term in the asymptotic formula for the spectral function given by the local Weyl law \eqref{eq:1.1}. Further note that, though   for $x=x_N, x_S$ the orbit volume is proportional to $\sqrt{\xi_1^2+\xi_2^2}$,  its inverse is still locally integrable on $T_x^\ast S^2$, and the integral in \eqref{eq:31.10.2015}  certainly exists. Ultimately, the leading coefficient in \eqref{eq:08.11.2015} must blow up as $x$ approaches the fixed points in order to compensate for the fact that the leading power  changes abruptly  from  $\sqrt \lambda$ to $\lambda$ at the fixed points. Note that the remainder estimates in \eqref{eq:08.11.2015} and \eqref{eq:31.10.2015} are consistent with the asymptotics \eqref{eq:3.12.2015} for the spherical function $Y_{k,0}$. 

\end{ex}

\section{Equivariant $\L^p$-bounds of eigenfunctions for non-singular group actions}
\label{sec:equivLp}

 Let the notation be as in the previous sections. From the asymptotic formula for the equivariant spectral function proved in  Theorem \ref{thm:main} we already deduced in Corollary \ref{cor:21.06.2016} point-wise bounds for isotypic spectral clusters. Similarly,  one immediately obtains in the non-singular case the following equivariant $\L^\infty$-bounds for  eigenfunctions.

\begin{proposition}[\bf $\L^\infty$-bounds for isotypic spectral clusters]
\label{thm:bounds}
Assume that $G$ acts on $M$ with orbits of the same dimension $\kappa$, and denote by $\chi_\lambda$ the spectral projection onto the sum of eigenspaces of $P$ with eigenvalues in the interval  $(\lambda, \lambda+1]$. Then, for any $\gamma \in \widehat G$, 
\bq
\label{eq:5}
\norm{(\chi_\lambda\circ \Pi_\gamma) u}_{\L^\infty(M)} \leq C_{\gamma} (1+ \lambda)^{\frac{n-\kappa-1}{2m}} \norm{u}_{\L^2(M)}, \qquad u \in \L^2(M),
\eq
 where, if $G$ is continuous, 
 \bqn
C_{\gamma}=O\left (\sqrt {d_\gamma \sup_{l \leq \lfloor \kappa/2+3\rfloor } \norm{D^l \gamma}_\infty} \right ).
\eqn
If $G$ is finite, one simply has $C_{\gamma}=O(\sqrt {d_\gamma \norm{\gamma}_\infty})$.  In particular, we obtain
\bqn 
\norm{u}_{\L^\infty(M)} \ll C_\gamma  \, \lambda^{\frac{n-\kappa-1}{2m}}
\eqn
for any eigenfunction $u \in \L^2_\gamma(M)$ of $P$  with eigenvalue $\lambda$ satisfying  $\norm{u}_{\L^2}=1$.
\end{proposition}

\begin{proof}
The assertion is a direct consequence of Theorem \ref{thm:main}. In fact, standard arguments \cite[Eq. (3.2.6)]{sogge14} imply that
\begin{align*} 
\norm{\chi_\lambda\circ \Pi_\gamma}^2_{\L^2 \rightarrow \L^\infty} &=\Big [\sup_x \Big ( \, \intop_M |K_{ \chi_\lambda\circ \Pi_\gamma}(x,y)|^2 dM(y) \Big )^{1/2} \Big ]^2 \\ &= \sup_x  K_{ \chi_\lambda\circ \Pi_\gamma}(x,x)=\sup_x \Big [e_\gamma(x,x,\lambda+1)-e_\gamma(x,x,\lambda)\Big ].
\end{align*}
Since $M=M_\mathrm{prin} \cup M_\mathrm{except}$, the assertion follows from  \eqref{eq:29.10.2015}. 
\end{proof}

It is instructive to see how Proposition \ref{thm:bounds} can be deduced directly  from Proposition \ref{thm:kernelasymp} by transferring the arguments given in \cite[pp. 50]{sogge14} to the equivariant setting.  By duality, the estimate \eqref{eq:5} is equivalent to 
\bq
\label{eq:2} 
\norm{(\chi_\lambda \circ \Pi_\gamma) u}_{\L^2(M)} \leq C_{\gamma} (1+ \lambda)^{\frac{n-\kappa-1}{2m}} \norm{u}_{\L^1(M)}.
\eq
In order to show the latter estimate, one considers again a Schwartz function $\rho \in \S(\R,\R_+)$ satisfying $\rho(0)=1$ and $\supp \hat \rho\in (-\delta/2,\delta/2)$ for a given $\delta>0$. If $\widetilde \chi_\lambda$ denotes the corresponding approximate spectral projection, one then shows that  \eqref{eq:2} is implied by 
\bq
\label{eq:3}
\norm{( \widetilde \chi_\lambda\circ \Pi_\gamma ) u}_{\L^2(M)} \leq C_{\gamma} (1+ \lambda)^{\frac{n-\kappa-1}{2m}} \norm{u}_{\L^1(M)}.
\eq
Thus, one is left with the task of proving \eqref{eq:3}. Now, the $\L^1\to\L^2$ operator norm can be estimated according to 
\begin{align*}
\begin{split}
\norm{ \widetilde \chi_\lambda \circ \Pi_\gamma }^2_{\L^1\rightarrow \L^2} &=\sup_{y \in M} \int_M | K_{ \widetilde \chi_\lambda \circ \Pi_\gamma }(x,y) |^2 \d M(x)\\
&= \sup_{y \in M} \sum_{j\geq 0, e_j \in \L^2_\gamma(M) } [\rho(\lambda-\lambda_j)]^2 | e_j(y)|^2 \leq  \norm{\rho}_{\L^\infty(\R)} \sup_{y \in M} K_{ \widetilde \chi_\lambda \circ \Pi_\gamma }(y,y).
\end{split}
\end{align*}
Hence, everything is shown, since by Proposition \ref{thm:kernelasymp} we have  the uniform bound 
\bqn
|K_{ \widetilde \chi_\lambda\circ \Pi_\gamma}(y,y)| \ll d_\gamma \sup_{l \leq \lfloor \kappa/2+1\rfloor} \norm{D^l \gamma}_\infty (1+\lambda)^{\frac{n-\kappa-1}m}, \qquad  y \in M=M_\mathrm{prin} \cup M_\mathrm{except},
\eqn
with $l=0$ for $G$ finite, 
and we obtain again \eqref{eq:5} with the slightly better estimate
\bq
\label{eq:1.6.2017}
C_{\gamma}=O\left (\sqrt {d_\gamma \sup_{l \leq \lfloor \kappa/2+1\rfloor} \norm{D^l \gamma}_\infty} \right ).
\eq

 \begin{rem}
If $G$ is a connected compact semisimple Lie group, the bound \eqref{eq:1.6.2017} can be rewritten in terms of the highest weight $\Lambda_\gamma\in \t^\ast_\C$ of $\gamma \in \widehat G$, and we obtain
\bqn 
C_{\gamma} = O \Big (\sqrt{|\Lambda_\gamma|^{2|\Sigma^+|+\lfloor \kappa/2+1\rfloor} }\Big ),
\eqn
compare Remark \ref{rem:22.5.2017}.
\end{rem}

\begin{ex}
In the situation of Example \ref{ex:torus}, where $M=T^2\subset \R^3$ is the standard $2$-torus on which $G=\SO(2)$ acts by rotations, Proposition \ref{thm:bounds} and \eqref{eq:1.6.2017} imply the bound
\bqn 
\norm{u}_{\L^\infty(T^2)} =O\left (\sqrt{1+|m|}\right ) , \qquad u \in \L^2_m(T^2), \, \norm{u}_{\L^2}=1,
\eqn
for any eigenfunction of $P$  in a specific isotypic component, which in case of the Laplace-Beltrami operator $\Delta$ are well-known. Indeed, via  the identification 
\bqn 
\R^2/\Z^2 \stackrel{\simeq} \longrightarrow T^2 \simeq S^1 \times S^1, (x_1,x_2) \, \longmapsto \, (e^{2\pi i x_1}, e^{2\pi i x_2}),
\eqn
the  standard orthonormal basis of $\Delta$ is given by $\mklm{e^{2\pi i k_1 x_1}e^{2\pi i k_2 x_2}\mid (k_1,k_2) \in \Z^2}$, showing that the above bound is sharp in the eigenvalue but not in the isotypic aspect. 
\end{ex}

In what follows, we shall derive refined $\L^p$-bounds for isotypic spectral clusters using complex interpolation techniques. For this, we shall need  the additional assumption that the co-spheres $S_x^\ast M$ are strictly convex. In essence, the proof is an elaboration of arguments from \cite{seeger-sogge} applied to the equivariant setting. While for the proof of the $\L^\infty$-bounds in the previous proposition it was sufficient to consider the asymptotic behaviour of  the integrals $I_{x,y}(\mu)$ in case that $x=y$, the proof of $\L^p$-estimates actually requires  estimates for the  integrals  $I_{x,y}(\mu)$ in a neighborhood of the diagonal, making things significantly more involved. This leads us to our second main result.

\begin{thm}[\bf $\L^p$-bounds for isotypic spectral clusters]
\label{thm:20.02.2016}
Let $M$ be  a closed connected Riemannian manifold $M$ of dimension $n$ on which a compact   Lie group $G$ acts effectively and isometrically with orbits of the same dimension $\kappa$. Further, let $P$ be the unique self-adjoint extension of a $G$-invariant elliptic positive symmetric classical pseudodifferential operator on $M$ 
of degree $m$, and assume that its principal symbol $p(x,\xi)$ is such that the co-spheres $S_x^\ast M:=\mklm{(x,\xi) \in T^\ast M\mid \, p(x,\xi)=1}$ are strictly convex. Denote by $\chi_\lambda$ the spectral projection onto the sum of eigenspaces of $P$ with eigenvalues in the interval  $(\lambda, \lambda+1]$, and by $\Pi_\gamma$ the projection onto the isotypic component $\L^2_\gamma(M)$, where  $\gamma \in \widehat G$. Then,  for $u \in \L^2(M)$ 
\bq
\label{eq:31.12.2015}
\norm{(\chi_\lambda \circ \Pi_\gamma) u}_{\L^q(M)} \leq \begin{cases} C_{\gamma} \,  \lambda^{\frac{\delta_{n-\kappa}(q)}{m}} \norm{u}_{\L^2(M)}, &  \frac{2(n-\kappa+1)}{n-\kappa-1} \leq q \leq \infty, \vspace{2mm} \\ C_{\gamma} \, \lambda^{\frac{(n-\kappa-1)(2-q')}{4m q'}} \norm{u}_{\L^2(M)}, &  2 \leq q \leq \frac{2(n-\kappa+1)}{n-\kappa-1}, \end{cases} 
\eq
 where $\frac 1q+\frac 1{q'}=1$,  
 \bqn 
 \delta_{n-\kappa}(q):=\max \left ( (n-\kappa) \left | \frac 12-\frac 1q \right| -\frac 12,0 \right ),
 \eqn
 and  the constant $C_{\gamma}>0$ satisfies the bound  \eqref{eq:1.6.2017} if $G$ is continuous, or $C_{\gamma}=O(\sqrt {d_\gamma \norm{\gamma}_\infty})$ in case that $G$ is finite.  In particular, 
\bqn 
\norm{u}_{\L^q(M)} \leq \begin{cases} C_\gamma \,  \lambda^{\frac{\delta_{n-\kappa}(q)}{m}}, &  \frac{2(n-\kappa+1)}{n-\kappa-1} \leq q \leq \infty, \vspace{2mm} \\C_\gamma  \, \lambda^{\frac{(n-\kappa-1)(2-q')}{4m q'}}, &  2 \leq q \leq \frac{2(n-\kappa+1)}{n-\kappa-1}, \end{cases} 
\eqn
for any eigenfunction $u \in \L^2_\gamma(M)$ of $P$  with eigenvalue $\lambda$    satisfying $\norm{u}_{\L^2}=1$.
\end{thm}
\begin{proof}
By duality, \eqref{eq:31.12.2015} is equivalent to
\bq
\label{eq:31.12.2015a}
\norm{(\chi_\mu \circ \Pi_\gamma) u}_{\L^2(M)} \leq \begin{cases} C_{\gamma} \,  \mu^{\delta_{n-\kappa}(p)} \norm{u}_{\L^p(M)}, & 1 \leq p \leq  \frac{2(n-\kappa+1)}{n-\kappa+3},  \vspace{2mm} \\ C_{\gamma} \, \mu^{\frac{(n-\kappa-1)(2-p)}{4p}} \norm{u}_{\L^p(M)}, &  \frac{2(n-\kappa+1)}{n-\kappa+3}\leq p \leq 2, \end{cases} 
\eq
where  $\chi_\mu$ denotes the spectral projection onto the sum of eigenspaces of $Q:=\sqrt[m]{P}$ with eigenvalues  in the interval  $(\mu, \mu+1]$, $\mu=\sqrt[m]{\lambda}$.  The case $p=1$ follows from the equivariant local Weyl law, and has already  been dealt with in \eqref{eq:2}. On the other hand, orthogonality arguments immediately imply
\bqn 
\norm{(\chi_\mu \circ \Pi_\gamma) u}_{\L^2(M)} \leq \norm{u}_{\L^2(M)}.
\eqn
By the Riesz interpolation theorem \cite[Chapter V, Theorem 1.3]{stein-weiss} it therefore suffices to prove \eqref{eq:31.12.2015a} in case that  $p=\frac{2(n-\kappa+1)}{n-\kappa+3}$, which can be inferred from  the corresponding bound
\bq
\label{eq:01.01.2016}
\norm{(\widetilde \chi_\mu \circ \Pi_\gamma) u}_{\L^2(M)} \leq C_{\gamma}\,  \mu^{\delta_{n-\kappa}(p)} \norm{u}_{\L^p(M)}, \qquad p=\frac{2(n-\kappa+1)}{n-\kappa+3},
\eq
for the approximate spectral projection $\widetilde \chi_\mu$ defined in \eqref{eq:2.1}. Now, by H\"older's inequality one computes
\begin{align*}
\norm{(\widetilde \chi_\mu \circ \Pi_\gamma) u}_{\L^2(M)}^2&=\int_M\left |\sum_{j\geq 0, \, e_j \in \L^2_\gamma(M)} \rho(\mu-\mu_j) E_ju (x) \right |^2 \d M(x)\\
&=\int_M \sum_{j\geq 0, \, e_j \in \L^2_\gamma(M)} \rho^2(\mu-\mu_j) E_ju (x) \overline{u(x)} \d M(x)\\ & \leq \norm{(\check \chi_\mu \circ \Pi_\gamma) u}_{\L^{p'}(M)} \norm{u}_{\L^p(M)},
\end{align*}
where $\frac 1p + \frac 1{p'}=1$, and we put $\check \chi_\mu u := \sum_{j=0}^\infty \rho^2(\mu-\mu_j) E_{j}u $ for $u \in \L^2(M)$. In order to see \eqref{eq:01.01.2016} it is therefore sufficient to prove 
\bq
\label{eq:01.01.2016a}
\norm{(\check \chi_\mu \circ \Pi_\gamma) u}_{\L^{p'}(M)} \leq C_{\gamma} \,  \mu^{2 \delta_{n-\kappa}(p)} \norm{u}_{\L^p(M)}, \qquad p=\frac{2(n-\kappa+1)}{n-\kappa+3}.
\eq
In order to show the latter, we shall use analytic interpolation \cite[Chapter V, Theorem 4.1]{stein-weiss}, and consider the analytic family of operators
\bqn 
\check \chi^z_\mu:= \frac {e^{z^2}} {2\pi} \int_\R\widehat{\rho^2}(t) \, e^{it\mu} \, (t-i\,0)^z \, U(t) \d t, \qquad z \in \C,
\eqn
where $(t-i0)^z$ denotes the distribution $\lim_{\eps \to 0^+} (t-i\eps )^z$. 
Clearly, $\check \chi^z_\mu=\check \chi_\mu$ if $z=0$, and since $2\delta_{n-\kappa}(2(n-\kappa +1)/(n-\kappa+3))=(n-\kappa-1)/(n-\kappa +1) $,  analytic interpolation theory implies that \eqref{eq:01.01.2016a} would follow if we were able to show that
\begin{align}
\label{eq:01.01.2016b}
\norm{(\check \chi_\mu^z \circ \Pi_\gamma) u}_{\L^2(M)} &\leq C_{\gamma} \, \norm{u}_{\L^2(M)}, & \Re z =-1,\\ 
\label{eq:01.01.2016c}
\norm{(\check \chi_\mu^z \circ \Pi_\gamma) u}_{\L^\infty(M)} &\leq C_{\gamma} \, \mu^{\frac{n-\kappa-1}{2}} \norm{u}_{\L^1(M)}, &  \Re z = \frac{n-\kappa-1}2. 
\end{align}
The crucial observation for the following estimates is that  the Fourier transform of the distribution $\tau_+^{z}/\Gamma(z+1)$ is given by the formula
\bq
\label{eq:t-i0}
 \int_\R e^{-it\tau}  \frac {\tau^{z}_+}{\Gamma(z+1)} \d \tau= e^{-i\pi (z+1) /2} (t-i\,0)^{-z-1}, \qquad z \in \C, 
\eq
where $\Gamma$ denotes the Gamma function, see \cite[Example 7.1.17]{hoermanderI}; in particular,   the singularity  of $\tau_+^{z}/\Gamma(z+1)$ at $\tau=0$ determines the asymptotic behaviour of $(t-i\0)^{-z-1}$ as $t \to \infty$, and viceversa. From this \eqref{eq:01.01.2016b} immediately  follows. The non-trivial bound to be proven is  \eqref{eq:01.01.2016c}, which would follow if we were able to show that the Schwartz kernel of $\check \chi_\mu^z \circ \Pi_\gamma$ satisfies
\begin{align}
\label{eq:15.01.2016}
|K_{\check \chi_\mu^z \circ \Pi_\gamma} (x,y)| \leq C_{\gamma}\,  \mu^{\frac{n-\kappa-1}2}, \qquad   \Re z = \frac{n-\kappa-1}2,
\end{align}
uniformly in $x, y \in M$.  Note that, in contrast, by Remark \ref{rem:12.7.2017} we  have the uniform bound $|K_{\check \chi_\mu\circ \Pi_\gamma} (x,y)| \leq C_{\gamma} \,  \mu^{n-\kappa-1}$. Furthermore, it is not possible to reduce the proof of \eqref{eq:15.01.2016} to the case $x=y$, since $\check \chi^z_\mu$ is not a positive operator, compare Remark \ref{rem:23.04.2017} (2). Now, it is clear from \eqref{eq:25.12.2015} that
\begin{gather*}
\begin{split}
K_{\check \chi_\mu^z \circ \Pi_\gamma}(x,y)
= \frac {\mu^n  d_\gamma \, e^{z^2}}{(2\pi)^{n+1}} \sum _\iota  \int_{\R}\int _{\R} e^{i\mu(t-Rt)} (t-i \, 0)^z \,  I^\gamma_\iota(\mu,R,t,x,y)  \d t\d R
\end{split}
\end{gather*}
where  $I^\gamma_\iota(\mu,R,t,x,y)$ is as in \eqref{eq:02.05.2015} with $\rho$ replaced by $\rho^2$. Due to the presence of the distribution $(t-i \,0)^z$ we cannot apply the stationary phase theorem to the $(R,t)$-integral. Instead, we shall apply the stationary phase principle to the  integrals $I^\gamma_\iota(\mu,R,t,x,y)$ first, and then use   \eqref{eq:t-i0} to deal with the $(R,t)$-integral. Let us first consider the case of a continuous group $G$. If $x \not\in Y_\iota$ or $\O_y\cap Y_\iota=\emptyset$, $I^\gamma_\iota(\mu,R,t,x,y)=0$.  Otherwise, one deduces from Theorem \ref{thm:14.05.2017}  for fixed $R,t \in \R$, and  any $\tilde N_i \in \N$ the asymptotic expansion
\begin{gather*}
 I^\gamma_\iota(\mu, R, t, x,y)\\ = \sum_{\J \in \pi_0(\Crit \, \Phi_{\iota,x,y})}   \frac{ e^{i \mu  \Phi_{\iota, x,y}^\J(R,t)}}{\mu^\kappa (\mu \norm{ \kappa_\iota(x)-\kappa_\iota(g_\J \cdot y)}+1 )^{\frac{n-1-\kappa}2}} \left [\sum_{k_1,k_2 =0}^{\tilde N_1-1, \tilde N_2-1}  \frac {\mathcal{Q}^\gamma_{\iota,\J, k_1,k_2}(R,t,x,y)}{\mu^{k_1} (\mu \norm{ \kappa_\iota(x)-\kappa_\iota(g_\J \cdot y)}+1)^{k_2}}\right. \\  \left. + \mathcal{R}^\gamma_{\iota,\J, \tilde N_1,\tilde N_2}(R,t,x,y,\mu) \right ].
\end{gather*}
The coefficients $\mathcal{Q}^\gamma_{\iota,\J, k_1,k_2}(R,t,x,y)$  and the remainder term $$\mathcal{R}^\gamma_{\iota,\J, \tilde N_1,\tilde N_2}(R,t,x,y,\mu)=O_{R,t}(\mu^{-\tilde N_1} (\mu \norm{ \kappa_\iota(x)-\kappa_\iota(g_\J \cdot y)}+1 )^{-\tilde N_2})$$ are given by distributions depending smoothly on $R,t$ with support in the component $\J$ of $\Crit  \, \Phi_{\iota,x,y}$ and $\Sigma^{R,t}_{\iota,x} \times G$, respectively. Furthermore, they and their derivatives with respect to $R,t$ are uniformly bounded in $x$ and $y$ by derivatives of $\gamma$ up to order $2k_1$ and $2\tilde N_1+\lfloor \kappa/2+1\rfloor$, respectively, while  
$$\Phi_{\iota,x,y}^\J(R,t):= R \,  c_{x,g_\J\cdot y} (t)
$$
 denotes the constant value of $\Phi_{\iota,x,y}$ on $\J$.  If $y \in \O_x$ one has $x=g_\J \cdot y$, so that up to remainder terms the kernel $K_{\check \chi_\mu^z \circ \Pi_\gamma}(x,y)$ is given by a linear combination of terms of the form
\bq
\label{eq:17.01.2016}
 {\mu^{n-\kappa-k_1}  \, d_\gamma \, e^{z^2}}  \int_{\R}\int _{\R} e^{i\mu(t-Rt)} (t-i \, 0)^z \,  \mathcal{Q}^\gamma_{\iota,\J, k_1,k_2}(R,t,x,y)  \d t \d R, 
\eq
 and if $y \notin \O_x$,  up to remainder terms   by a linear combination  of terms of the form 
\begin{align}
\begin{split}
\label{eq:18.01.2016}
 \mu^{n-\kappa-k_1} &(\mu \norm{ \kappa_\iota(x)-\kappa_\iota(g_\J \cdot y)}+1 )^{-\frac{n-1-\kappa}2-k_2}  \, d_\gamma \, e^{z^2}  \\ \cdot  \int_{\R}\int _{\R} e^{i\mu(t-Rt)} &(t-i \, 0)^z \, e^{i\mu \Phi_{\iota, x,y}^\J(R,t)}  \mathcal{Q}^\gamma_{\iota,\J, k_1,k_2}(R,t,x,y)  \d t\d R.
\end{split}
\end{align}
Now, as a consequence of \eqref{eq:t-i0}, one has for any $f \in \CT(\R\times \R)$, that might depend on $\mu$ as a parameter, and $z \in \C$
\begin{align*}
\eklm{(t-i0)^z, e^{i\mu(1-R) t } f(R,t )}&=\frac{e^{-i\pi z/2}}{\Gamma(-z)}\eklm{\tau_+^{-z-1},\widehat{f(R,\cdot)}(\tau-\mu(1-R)) }.
\end{align*}
Let us consider first the case when $z=0,1,2, 3, \dots$, and write $-l:=-z-1$. Since  $\tau_+^{-l}/\Gamma(-l+1)=\delta^{(l-1)}_0$, compare \cite[(3.2.17)']{hoermanderI},  partial integration yields 
\begin{align*}
\begin{split}
\int_\R \int_\R  e^{i\mu(1-R) t } (t-i0)^z  f(R,t ) \d t \, \d R&=e^{-i\pi z/2} (-1)^{l-1}  \int_\R \int_\R (-it)^{l-1} e^{it\mu(1-R)} f(R,t) \d t \d R \\&=e^{-i\pi z/2} \mu^{-l+1}  \int_\R \int_\R  e^{it\mu (1-R)} (\gd_R^{l-1}f)(R,t) \d t \d R.
\end{split}
\end{align*}
The relevant integrals in  \eqref{eq:17.01.2016} and \eqref{eq:18.01.2016} therefore read
\bq
\label{eq:18.02.2016} 
e^{-i\pi z/2} \mu^{-l+1}  \int_\R \int_\R  e^{it\mu (1-R)} \gd_R^{l-1}\Big [e^{i\mu   \Phi_{\iota, x,y}^\J(R,t)}  \mathcal{Q}^\gamma_{\iota, \J,k_1,k_2}(R,t,x,y)\Big ] \d t \d R.
\eq
Since  similar considerations also hold for the remainder terms,  an application of  the classical stationary phase theorem \cite[Proposition 2.3]{grigis-sjoestrand} to the $(R,t)$-integral allows us to deduce for $z=0,1,2,3, \dots $ the uniform bound  \eqref{eq:15.01.2016}. Indeed, if $y \in \O_x$ the phase function in \eqref{eq:18.02.2016} simply reads $t(1-R)$, and the only critical point is $(R_0,t_0)=(1,0)$, which is non-degenerate, the determinant of  the Hessian being $-1$. If $y \not \in \O_x$, the phase function is given by $t(1-R)+ \Phi^\J_{\iota,x,y}(R, t)$, and a computation shows that the determinant of the matrix of its second derivatives is given by
\bq
\label{eq:20.06.2016}
-\left (-1+ c_{x,g_\J\cdot y}'(t)    \right )^2 \approx - (-1 \pm O(\norm{\kappa_\iota(x) - \kappa_\iota(g_\J \cdot y)}))^2 
\eq
since $c_{x,g \cdot y}(t)=\pm {\norm {\kappa_\iota(x) - \kappa_\iota( g \cdot y )}}/{\norm {\grad_\eta \zeta_\iota(t,\kappa_\iota(x),\omega) }}$.  By choosing the charts $Y_\iota$ sufficiently small so that $|\kappa_\iota(x) - \kappa_\iota(g_\J \cdot y)|1$ is small, we can therefore achieve that  in a  sufficiently small neighborhood of $(R,t)=(1,0)$, which is where $\mathcal{Q}^{\gamma}_{\iota, \J,k_1,k_2}(R,t,x,y)$ is supported, the phase function $t(1-R)+ \Phi^\J_{\iota,x,y}(R, t)$ has, if at all, only non-degenerate, hence isolated, critical points.  If we now apply the stationary phase theorem to the integral \eqref{eq:18.02.2016} with respect to the phase function $t(1-R)$ and  $t(1-R)+ \Phi^\J_{\iota,x,y}(R,t)$, respectively, treating the remainder terms alike, we obtain 
\bqn
|K_{\check \chi_\mu^z \circ \Pi_\gamma}(x,y)|\leq   C_{\gamma} \, \mu^{n-\kappa-z-1},  \qquad y \in \mathcal{O}_x, 
\eqn
as well as
\begin{align*}
|K_{\check \chi_\mu^z \circ \Pi_\gamma}(x,y)|&\leq C_{\gamma}  \, \mu^{n-\kappa-z-1} \big (\mu \norm{\kappa_\iota(x)-\kappa_\iota(g_\J \cdot y)}+1\big )^{-\frac{n-1-\kappa}2} \sum_{l'=0}^{l-1} \big (\mu \norm{\kappa_\iota(x)-\kappa_\iota(g_\J \cdot y)}\big )^{l'} \\
&\leq C_\gamma \, \mu^{n-\kappa-z-1}, \qquad y \notin \mathcal{O}_x, 
\end{align*}
yielding \eqref{eq:15.01.2016} for $z=0,1,2,3, \dots$. Next, let us turn to the case where  $z\not=0,1,2, 3, \dots$, and note that by  homogeneity of $\tau_+^z$  one has
\begin{align*}
\eklm{(t-i0)^z, e^{i\mu(1-R) t } f(R,t )}&= \frac{e^{-i\pi z/2}}{\Gamma(-z)}\eklm{\tau_+^{-z-1}, \mu^{-1} \widehat{f(R,{\cdot}/\mu )}\Big (\tau/\mu -1+R\Big ) }\\
&=\frac{e^{-i\pi z/2}}{\Gamma(-z)} \mu^{-z-1} \eklm{\tau_+^{-z-1},\widehat{f (R,\cdot/ \mu )}(\tau-1+R) },
\end{align*}
compare \cite[(3.2.7)]{hoermanderI}. By definition of $\tau_+^{-z-1}$ and partial integration one computes   
\begin{align*}
\label{eq:21.01.2016}
\begin{split}
-z(-z+1) & \dots  (-z-1+l) (-1)^l \int_\R \int_\R \tau_+^{-z-1} \widehat{f (R,\cdot/ \mu )}(\tau-1+R) \d \tau  \d R \\ 
&=
 \int_\R \int_\R \tau_+^{-z-1+l}\gd_\tau^l \Big [\widehat{f (R,\cdot/ \mu )}(\tau-1+R)\Big ] \d \tau \d R \\ 
 &= (-1)^l  \mu  \int_\R \tau_+^{-z-1+l} \left [  \int_\R \int_\R e^{-it\mu(\tau-1+R)} 
(\gd_R^l f )(R,t ) \d t  \d R  \right ]\d \tau,
\end{split}
\end{align*}
where $l > \Re z$ is a sufficiently large positive integer, so that $ \tau_+^{-z-1+l}$ becomes locally integrable. Note that we have, as we may,  interchanged the integrals over $\tau$ and $R$, while the integrals over $\tau$ and $t $ cannot be interchanged. As a consequence, the relevant integrals in \eqref{eq:17.01.2016} and \eqref{eq:18.01.2016} are given by  linear combinations of terms of the form
\bq
\label{eq:20.02.2016}
\mu^{-z} \int_\R \tau_+^{-z-1+l} \left [  \int_\R \int_\R e^{-it\mu(\tau-1+R)} 
\gd_R^l \big [e^{i\mu \Phi_{\iota, x,y}^\J(R,t)}  \mathcal{Q}^\gamma_{\iota,\J, k_1,k_2}(R,t,x,y)\big ]   \d t  \d R  \right ]\d \tau.
\eq
Again, let us examine the $(R,t)$-integral by means of the stationary phase. If $y \in \O_x$, the phase function is given by $t(\tau-1+R)$, the only critical point is $(R_0,t_0)=(1-\tau,0)$, and we obtain for \eqref{eq:20.02.2016} the estimate 
\bqn 
2\pi \mu^{-z-1} \int_\R \tau_+^{-z-1+l}  \Big [ (\gd_R^l   \mathcal{Q}^{\gamma}_{\iota,\J,k_1,k_2})(1-\tau,0,x,y) +O_{\gamma,\tau}(\mu^{-1})\Big ]  \d \tau=O_\gamma(\mu^{-\Re z-1})
\eqn
uniformly in $x,y$, 
the remainder $O_{\gamma,\tau}(\mu^{-1})$ being rapidly falling in $\tau$, since $\mathcal{Q}^{\gamma}_{\iota,\J,k_1,k_2}$ has compact $(R,t)$-support. Now, if $y \not \in \O_x$, the phase function reads
$
t(1-R)+ \Phi^\J_{\iota,x,y}(R,t)-t\tau$, and the determinant of the matrix of its second derivatives is again given by \eqref{eq:20.06.2016}. By the previous arguments, we can therefore assume that  in a  sufficiently small neighborhood of $(R,t)=(1,0)$ the phase function $ t(1-R)+ \Phi^\J_{\iota,x,y}(R,t)-t\tau$ has only one non-degenerate critical point $(R_0,t_0)$.
 It  satisfies the relations
\bqn 
t_0= c_{x,g_\J\cdot y}(t_0) \approx 0, \qquad R_0=\frac{1-\tau}{1-c_{x,g_\J\cdot y}'(t_0)}\approx 1-\tau,
\eqn
and at this point, the phase function takes the value
$
t_0(1-R_0)+ \Phi^\J_{\iota,x,y}(R_0,t_0)-t_0\tau=t_0(1-\tau)$. Taking into account that  for any $w \in \C$ with $\Re w>-1$ and $g \in \S(\R)$ one has
\bqn
\label{eq:24.02.2016}
\int_\R  e^{-i\mu\tau}  \tau_+^{w} g(\tau) \d \tau = O\big ((1+\mu)^{-\Re w-1})\big ) , \qquad \mu \geq 0, 
\eqn
compare \eqref{eq:t-i0}, we obtain for \eqref{eq:20.02.2016}  the bound
\begin{gather*}
2\pi \mu^{-z-1} \int_\R \tau_+^{-z-1+l}  e^{i\mu t_0(1-\tau)}  \sum_{l'+l''=l} c_{l',l''}
\\ \left [ (i\mu \, c_{x,g_\J\cdot y}(t_0))^{l'}    (\gd_R^{l''}   \mathcal{Q}^{\gamma}_{\iota,\J,k_1,k_2})(R_0,t_0,x,y)   +O_{\gamma,\tau}\big (\mu^{-1+l'}\norm{\kappa_\iota(x)-\kappa_\iota(g_\J \cdot y)}^{l'}\big )\right]  \d \tau\\ =O_\gamma\Big (\mu^{- \Re z-1}\big (1+\mu  \norm{\kappa_\iota(x)-\kappa_\iota(g_\J \cdot y)}\big ) ^{\Re z-l }\sum_{l'=0}^l (\mu \norm{\kappa_\iota(x)-\kappa_\iota(g_\J \cdot y)})^{l'}\Big )\\
=O_\gamma\Big (  \mu^{- \Re z-1}\big (1+\mu  \norm{\kappa_\iota(x)-\kappa_\iota(g_\J \cdot y)}\big ) ^{\Re z} \Big )
\end{gather*}
uniformly in $x,y$, where  the $c_{l',l''}$ are certain coefficients, and the remainder is rapidly falling in $\tau$.  Treating the remainder terms alike, we have shown \eqref{eq:15.01.2016} for $z\not=0,1,2,3,\dots$ as well. This completes the proof of Theorem \ref{thm:20.02.2016} in case that $G$ is continuous. The finite group case follows in an analogous way using Proposition \ref{prop:10.08.2017} instead of Theorem \ref{thm:14.05.2017}.
 \end{proof}

\begin{ex}
Let us resume Example \ref{ex:cocpt} of  a connected semisimple Lie group $G$ with finite center, discrete  co-compact subgroup  $\Gamma$, and   maximal compact subgroup $K$. The group  $K$ acts on $\Gamma \backslash G$ with orbits of principal and exceptional type, all orbits having the  dimension $\dim K$, and  we deduce from Proposition \ref{thm:bounds}  for each $\gamma \in \widehat K$ the estimate
\bqn 
\norm{u}_{\L^\infty(\Gamma \backslash G)} \leq C_{\gamma}\, \lambda^{\frac{\dim G/K -1}{2m}}, \qquad u \in \L^2_\gamma(\Gamma \backslash G), \, \norm{u}_{\L^2}=1,
\eqn
for any eigenfunction $u$ of a $K$-invariant elliptic positive symmetric classical pseudodifferential operator $P$ on $\Gamma \backslash G$ 
of degree $m$ with eigenvalue $\lambda$.  More generally,  with $\frac 1q+\frac 1{q'}=1$  and 
 \bqn 
 \delta(q):=\max \left ( \dim G/K \left | \frac 12-\frac 1q \right| -\frac 12,0 \right )
 \eqn
we have by Theorem \ref{thm:20.02.2016} the bound
\bqn 
\norm{u}_{\L^q(\Gamma \backslash G)} \leq \begin{cases} C_{\gamma} \,  \lambda^{\frac{\delta(q)}{m}}, &  \frac{2(\dim G/K+1)}{\dim G/K-1} \leq q \leq \infty, \vspace{2mm} \\ C_{\gamma} \, \lambda^{\frac{(\dim G/K-1)(2-q')}{4m q'}}, &  2 \leq q \leq \frac{2(\dim G/K+1)}{\dim G/K-1}, \end{cases} 
\eqn
provided that $P$ satisfies the strict convexity assumption in Theorem \ref{thm:20.02.2016}. In case that $\Gamma$ has no torsion, $\Gamma \backslash G/K$ is a locally symmetric space, and eigenfunctions of the Beltrami-Laplace operator on $\Gamma \backslash G/K$ correspond exactly to $K$-invariant eigenfunctions of the Beltrami-Laplace operator on $\Gamma \backslash G$,  the space $ \L^2(\Gamma \backslash G\slash K)\simeq \L^2(\Gamma \backslash G)^K$ being isomorphic to  the trivial isotypic component in the Peter-Weyl decomposition of $\L^2(\Gamma \backslash G)$.
Thus, our results  generalize  the classical $\L^p$-bounds on $\Gamma \backslash G/K $  to arbitrary $K$-types. 
\end{ex}

\section{The desingularization process}
\label{sec:DP}

As already noted, the asymptotic formula for the reduced spectral function $e_\gamma(x,x,\lambda)$ given in Theorem \ref{thm:main} depends in a highly non-smooth way on $x \in M$ if non-principal orbits are present. Moreover,  if $G$ is continuous, the mentioned formula  does not give a  precise description of the caustic behaviour of   $e_\gamma(x,x,\lambda)$ near singular orbits, leaving  it  unclear if the coefficients in the expansion of $e_\gamma(x,x,\lambda)$ are integrable in $x$, and how one could deduce from Theorem \ref{thm:main} asympotics for the equivariant spectral counting function $N_\gamma(\lambda):= \int_M e_\gamma(x,x,\lambda) \d M(x)$. In what follows, we shall therefore examine the case of a continuous group $G$ more closely. Our goal is to derive a description of $e_\gamma(x,x,\lambda)$ that  interpolates between the  asymptotics for different values of $x$, and in particular to characterize the  behaviour of the leading coefficient and the remainder term  in Theorem \ref{thm:main} as $x \in M_\mathrm{prin}$ approaches singular orbits. For this, we shall make use of resolution of singularities. As we shall see, the major difficulty resides in the fact that,  unless the Hamiltonian $G$-action on $T^\ast M$ is free, so that the corresponding momentum map becomes a submersion,  $\Omega$ and the critical set
  \eqref{eq:23.04.2015} of the phase function $\Phi$ are not  smooth manifolds. To overcome this difficulty, it was shown in \cite{ramacher10}   that by constructing a strong resolution of the set 
\bq
\label{eq:calN}
\Ncal:=\mklm{(x,g) \in M \times G \mid g\cdot x = x}
\eq
a partial desingularization 
\bq
\label{eq:220215}
\mathcal{Z}: \widetilde {\bf X} \rightarrow {\bf X}:= T^\ast M \times G 
\eq
 of the critical  set $\Crit \, \Phi$ can be achieved, and after applying the stationary phase theorem in the resolution space $\widetilde {\bf X}$, an asymptotic description of the integrals $I(\mu)$ defined in \eqref{eq:integral} can be  obtained, leading to an asymptotic formula for $N_\gamma(\lambda)$. 
 In the ensuing sections, we shall use  the partial desingularization \eqref{eq:220215}   to obtain an  asymptotic formula for the integrals $I_x(\mu)$ defined in \eqref{eq:11.9.2017}
  that allows us to describe  the caustic behaviour of the  coefficients $\mathcal{Q}_k(x,x)$   in Theorem \ref{thm:12.05.2015} (a) as one approaches singular orbits. One can deduce  from this  the asymptotic description of the integrals $I(\mu)$ given in \cite{ramacher10}, but the converse implication is more subtle and not straight-forward. For this reason, a careful re-examination of the results of \cite{ramacher10} is needed in order to obtain a precise description of the coefficients in the asymptotic formula for the integrals $I_x(\mu)$ and, ultimately, of the leading coefficient in the asymptotic formula for the equivariant spectral function. \\

Let $M$ be a closed connected Riemannian manifold and $G$ a  continuous compact Lie group  acting on $M$ by isometries. 
In what follows, we shall   recall the construction  of  the  partial desingularization \eqref{eq:220215}  of the critical set $
\Ccal:=\mklm{(x,\eta,g) \in ( \Omega \cap T^\ast M) \times G\mid g \in G_{(x,\eta)}} $ performed in \cite{ramacher10}.  The desingularization process presented here is exactly the same, only that we apply it now to the study of the integrals \eqref{eq:03.05.2015} instead of the integrals \eqref{eq:integral}. For details,  the reader is referred to \cite{ramacher10}. Consider the decomposition of $M$ into orbit types
 \bq
 \label{eq:2.19}
M=M(H_1) \, \dot \cup \, \cdots \, \dot \cup \, M(H_L),
\eq
where we suppose that  the isotropy types are numbered in such a way that $(H_i) \geq (H_j)$ implies $i \leq j$, $(H_L)$ being the principal isotropy type, see Figure \ref{fig:tree}.

\medskip

\begin{figure}[h!]
\begin{center}
\begin{tikzpicture}[node distance=1.4cm, auto]

\node (A00) {$H_{L}$}; 

\node (B0) [above left of=A00] {$H_{L-4}$}; 
\node (C0) [right of=B0] {$H_{L-3}$}; 
\node (D0) [right of=C0] {$H_{L-2}$}; 
\node (F0) [right of=D0] {$H_{L-1}$}; 

\node (D) [above left of=B0] {$H_{m-1}$}; 
\node (E) [right of=D] {$H_m$}; 
\node (F) [right of=E] {$H_{m+1}$}; 

\node (E1) [above left of =D] {$H_{i+2}$}; 
\node (F1) [right of=E1] {$H_{i+3}$}; 
\node (G1) [right of=F1] {$\cdots$};
\node (H1) [right of=G1] {$H_l$};
\node (I1) [right  of=H1] {$H_{l+1}$};

\node (D2) [above left of=E1] {$H_1$}; 
\node (E2) [right of=D2] {$H_2$}; 
\node (F2) [right of=E2] {$H_3$}; 
\node (G2) [right of=F2] {$\cdots$};
\node (H2) [right of=G2] {$H_{i-1}$};
\node (I2) [right of=H2] {$H_{i}$};
\node (J2) [right of=I2] {$H_{i+1}$};

\draw[-] (A00) to node {} (B0);
\draw[-] (A00) to node {} (C0);
\draw[-] (A00) to node {} (D0);
\draw[-] (A00) to node {} (F0);

\draw[-] (E1) to node {} (D2);
\draw[-] (E2) to node {} (E1);
\draw[-] (E2) to node {} (F1);

\draw[-, dashed] (D) to node {} (B0);
\draw[-, dashed] (E) to node {} (C0);
\draw[-, dashed] (F) to node {} (C0);

\draw[-] (F1) to node {} (F2);
\draw[-] (E2) to node {} (E1);

\draw[-] (H1) to node {} (H2);
\draw[-] (I2) to node {} (I1);
\draw[-] (J2) to node {} (I1);

\draw[-, dashed] (F) to node {} (H1);
\draw[-, dashed] (F) to node {} (I1);

\draw[-, dashed] (D) to node {} (E1);
\draw[-, dashed] (D) to node {} (F1);

\draw[-, dashed] (I1) to node {} (F0);

\end{tikzpicture}
\end{center}
\vspace{-.3cm}

\caption{\small An isotropy tree corresponding to the decomposition \eqref{eq:2.19}.  A line between two subgroups indicates partial ordering.}\label{fig:tree}
\end{figure}
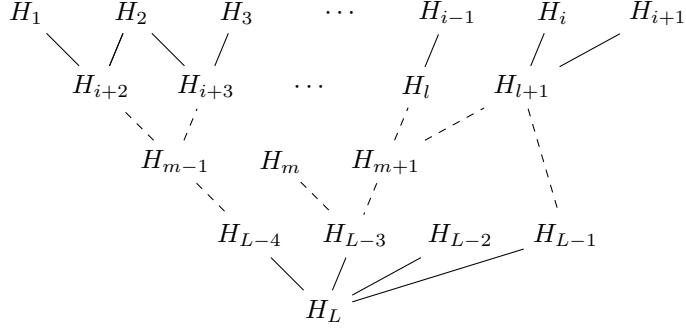

\smallskip

To construct  \eqref{eq:220215},  an iterative  process along the strata of the  $G$-action on $M$ is set up, where the centers of the blow-ups are successively chosen as isotropy bundles over unions of maximally singular orbits. For simplicity, one assumes that at each  step the union of maximally singular orbits is connected. 

\subsection*{Beginning of iteration} Let  $f_k:\nu_k\rightarrow M_k$  be an invariant tubular neighborhood of $M_k(H_k)$ in 
\bdm
M_k:=M-\bigcup_{i=1}^{k-1} f_i(\stackrel{\circ}{D}_{1/2}(\nu_i)), \qquad k=1, \dots, L, 
\edm 
a manifold with corners on which $G$ acts with the isotropy types $(H_k), (H_{k+1}), \dots, (H_L)$. Here $\nu_k$ denotes the normal G-vector bundle of $M_k(H_k)$, $\stackrel{\circ}{D}_{1/2}(\nu_i):=\mklm{  v \in \nu_i\mid \, \norm {v} <1/2}$, 
\bqn
f_k(p^{(k)},v^{(k)}):=(\exp_{p^{(k)}} \circ \,  \gamma ^{(k)})( v^{(k)}), \qquad  p^{(k)} \in M_k(H_k), \, v^{(k)} \in (\nu_k)_{p^{(k)}},
\eqn
 is an equivariant diffeomorphism given in terms of the exponential map, and
\bqn
\gamma^{(k)}(v^{(k)}):=\frac {F_k( p^{(k)})}{ (1+\|v^{(k)}\|^2 )^{1/2}} v^{(k)}, 
\eqn
where $F_k:M_k(H_k)\rightarrow \R$ is  a smooth, $G$-invariant, positive function, see \cite[p. 306]{bredon}. Let $S_k$ be the unit sphere bundle over $M_k(H_k)$, and put  $W_k:=f_k(\stackrel{\circ}{D_1}(\nu_k))$, $W_L: = \stackrel{\circ}{M}_L$, so that we obtain the open covering
\bq
\label{eq:872}
M= W_1 \cup \dots \cup W_L.
\eq
Fix an inner product on $\g$,  which induces a Riemannian structure on $G$, and consider for each $k$ and $p^{(k)}\in M_k(H_k)$  the decomposition
\bqn
T_eG\simeq  \g=\g_{p^{(k)}}\oplus \g_{p^{(k)}}^\perp, 
\eqn
where $\g_{p^{(k)}}\simeq  T_e G_{p^{(k)}}$ denotes the Lie algebra of the  stabilizer $G_{p^{(k)}}$ of $p^{(k)}$, and $\g_{p^{(k)}}^\perp$ its orthogonal complement with respect to the above Riemannian structure.  Now, introduce a partition of unity $ \{ \chi_k\}_{k=1, \dots, L}$ subordinated to the covering \eqref{eq:872}, and define
\bqn 
I_k(x,\mu)
:=   \chi_k (x) I_x(\mu)
\eqn
with  $I_x(\mu)$ as in   \eqref{eq:11.9.2017}.  By Theorem \ref{thm:12.05.2015} (a)  the asymptotic expansion for  $I_L(x,\mu)$ depends smoothly on $x \in W_L\cap Y$. Let us therefore turn to the case when $1 \leq k \leq  L-1$ and $W_k \cap Y \not= \emptyset$.   For fixed $k$ and $x=f_k(p^{(k)}, v^{(k)})\in W_k\cap Y$  Lemma \ref{lem:21.04.2015} (a) implies that 
 $$\Crit \, \Phi_x  =\mklm{(\omega,g) \in \Sigma^{R,t}_x \times G\mid (x,\omega) \in \Omega, \, g\cdot x=x} \subset  \Sigma^{R,t}_x\times G_{p^{(k)}}.$$
Up to non-stationary contributions, it will therefore suffice to evaluate the integrals $I_k(x,\mu)$ in a neighborhood of $G_{p^{(k)}}$. To this end, consider   the isotropy  bundle $\mathrm{Iso} \,M_k(H_k) \rightarrow M_k(H_k)$ over $M_k(H_k)$, as well as the canonical projection
\bqn 
\pi_k: W_k \rightarrow M_k(H_k), \qquad f_k(p^{(k)},v^{(k)}) \mapsto p^{(k)}.
\eqn
Further, let  
\bqn
\pi_k^\ast \, \mathrm{Iso}\,  M_k(H_k)=\mklm {(f_k(p^{(k)},v^{(k)}),h^{(k)})\in W_k \times G\mid h^{(k)} \in G_{p^{(k)}}}
\eqn
 be the induced bundle. Let  $U_k$ be  a sufficiently small tubular neighborhood of $\pi^\ast_k \mathrm{Iso} \, M_k(H_k)$ in $W_k \times G$, and note that the fiber of the normal bundle $N \, \pi^\ast_k \mathrm{Iso}\,  M_k(H_k)$ at a point $(f_k(p^{(k)},v^{(k)}),h^{(k)})$ may be identified with the fiber of the normal bundle to $G_{p^{(k)}}$ at the point $h^{(k)}$.  Consider  further an orthonormal basis $\{A_1(p^{(k)}), \dots, A_{d^{(k)}}(p^{(k)})\}$ of $\g_{p^{(k)}}^\perp$, and introduce  canonical coordinates of the second kind
\bq
\label{eq:922}
\R^{d^{(k)}} \times G_{ p^{(k)}} \ni (\alpha_1^{(k)}, \dots, \alpha_{d^{(k)}}^{(k)}, h^{(k)}) \, \longmapsto  \,  \e{\sum_i \alpha_i^{(k)} A_i(p^{(k)})} h^{(k)}
\eq
in a neighborhood of $G_{p^{(k)}}$,  see \cite[p. 146]{helgason78}. Denote by $b_\mu$  the amplitude $a$ multiplied by   a smooth cut-off-function with  support in $U_k$ which is equal to $1$ in a small neighborhood of  $\pi^\ast_k \mathrm{Iso} \, M_k(H_k)$. Taking into account the non-stationary phase theorem \cite[Theorem 7.7.1]{hoermanderI} one  computes 
\begin{align}\begin{split}\label{eq:19}
 I_k(x,\mu)
 &=  \chi_k(x) \int_{ G_{p^{(k)}} \times  \g^\perp_{p^{(k)}} \times {\Sigma^{R,t}_x}
} e^{i\mu \Phi_x}  b_\mu  \,  \d(\Sigma^{R,t}_x)( \omega) \, dA^{(k)} \, dh^{(k)} +O(\mu^{-\infty}),
\end{split}
\end{align}
where $dh^{(k)}, dA^{(k)}$ are suitable volume densities on the sets $G_{p^{(k)}}$ and $ \g_{p^{(k)}}^\perp\simeq  N_{h^{(k)}} G_{p^{(k)}} $, respectively, such that 
$ \d g \equiv   dA^{(k)} \, dh^{(k)}$, compare \cite[(5.4)]{ramacher10}, and the remainder estimate is uniform in $x$. 

 We shall now sucessively resolve the singularities of \eqref{eq:calN} in order to obtain a factorization of $\Phi_x$. 
Note that  by \cite[Eq. (5.1)]{ramacher10} 
\bqn 
\mathcal{N}  = \mathcal{N}_L  \cup \bigcup_{k=1}^{L-1} \mathcal{N}_k,
\eqn
where   $\mathcal{N}_k:=\mathcal{N} \cap U_k$, $\mathcal{N}_L:=\mathrm{Iso } \, W_L$,  $\mathrm{Iso } \, W_L \to W_L$ being the isotropy bundle over $W_L$. While   $\mathcal{N}_L$ is  a smooth submanifold, 
$\mathcal{N}_k$ is in general singular. In particular, if $\dim H_k\not=\dim H_L$,
 $\mathcal{N}_k$ has a maximal singular locus  given by $\mathrm{Iso} \, M_k(H_k)$. One then performs for each $k \in \mklm{1,\dots,L-1}$  a blow-up 
\bqn 
\zeta_k: B_{Z_k}( U_k) \longrightarrow U_k 
\eqn
 with center $Z_k:= \mathrm{Iso} \, M_k(H_k)\subset \mathcal{N}_k$, and by piecing these transformations together one obtains the global blow-up
 \bqn 
 \zeta^{(1)}: B_{Z^{(1)}}  \M \longrightarrow \M, \qquad Z^{(1)} :=\bigcup _{k=1}^{\stackrel{L-1}{\bullet }} Z_k,
 \eqn
 where we put $\M:=M \times G$, compare \cite[p.\ 56]{ramacher10}.  To get a local description, fix $k$,  let $\{v_1^{(k)},\dots  ,v_{c^{(k)}}^{(k)}\}$ be an orthonormal frame in $\nu_k$,  
and $(\theta_1^{(k)},\dots, \theta_{c^{(k)}}^{(k)})$ be coordinates in $\gamma^{(k)}((\nu_k)_{p^{(k)}})$. Similarly, consider the coordinates  $(\alpha_1^{(k)},\dots, \alpha_{d^{(k)}}^{(k)})$  introduced in \eqref{eq:922}. If one now covers $B_{Z_k}( U_k)$ with standard projective charts $\mklm{(\phi^\rho_k, \mathcal{O}^\rho_k)}$ one obtains in the so-called $\theta^{(k)}$-charts $\mklm{\mathcal{O}^\rho_k}_{1\leq \rho\leq c^{(k)}}$, in which the $\theta_\rho^{(k)}$-coordinate is non-zero, for $\zeta_k$ the local expressions
\begin{align}
\label{eq:21}\begin{split}
\zeta^\rho_k&=\zeta_k \circ (\phi^\rho_k)^{-1}: ( p^{(k)},\tau_k, \tilde v^{(k)},  A^{(k)}, h^{(k)})
 \mapsto \Big (\exp_{p^{(k)}}   \tau_k \tilde v^{(k)}, \e{\tau_k A^{(k)}} h^{(k)}\Big )=(x,g),
\end{split}
\end{align}
where 
$$ p^{(k)} \in M_k(H_k), \qquad   A^{(k)}\in \g^\perp_{p^{(k)}}, \qquad h^{(k)} \in G_{p^{(k)}}, \qquad \tilde v^{(k)}
  \in  \gamma^{(k)} \big (  ( S_k^+)_{p^{(k)}} \big ),
  $$
and 
$ S_k^+:=\mklm{v \in \nu_k\mid v := \sum s_i v_i^{(k)}, s_\rho>0,  \norm{v}=1}$, while $\tau_k \in (-1,1)$, see \cite[Eq. (5.6)]{ramacher10}. A similar description of $\zeta_k$ is given in the so-called  $\alpha^{(k)}$-charts $\mklm{\mathcal{O}^\rho_k}_{c^{(k)}+1 \leq \rho\leq c^{(k)}+d^{(k)}}$, in which the $\alpha_\rho^{(k)}$-coordinate does not vanish. By performing Taylor expansion at $\tau_k=0$ one can then show that  the phase function  \eqref{eq:phase}  factorizes  according to 
\bq
\label{eq:22}
 \Phi \circ (\id_\eta \otimes \zeta_k^\rho) =  \,^{(k)} \widetilde \Phi^{tot}=\tau_k \cdot  \,  ^{(k)} \phw, 
 \eq
$ \,^{(k)} \widetilde \Phi^{tot}$ and $  \,  ^{(k)} \phw $ being the \emph{total} and \emph{weak transform} of the phase function $\Phi$, respectively,  see \cite[Eqs. (5.8) and (5.9)]{ramacher10}. Since $\zeta_k$ is a real-analytic surjective proper map, which is a diffeomorphism on the complement of $\zeta_k^{-1}(Z_k)$,  we can lift the integral $I_k(x,\mu)$  along the restriction of $\zeta_k$ to the fiber over $\mklm{x} \times G$ to the resolution space $B_{Z_k}(U_k)$. To obtain local expressions,  introduce a compactly supported partition $\mklm{u^\rho_k}$ of unity subordinate to the covering $\mklm{\mathcal{O}^\rho_k}$, set $a_k^\rho := (u^\rho_k \circ (\phi_k^\rho)^{-1})  \cdot [ (b_\mu \chi_k) \circ ( \id_\omega \otimes \zeta_k^\rho)]$, and define for $x=\exp_{p^{(k)}}   \tau_k \tilde v^{(k)} \in W_k\cap Y$ and $1 \leq \rho \leq c^{(k)}$ the integrals
 \begin{gather*}
I_k^\rho(x,\mu):=|\tau_k|^{d^{(k)}}  \int_{ G_{p^{(k)}} \times  \g_{p^{(k)}}^\perp\times \Sigma^{R,t}_x
} e^{i\mu{\tau_k}\,  ^{(k)} \phw_{\tau_k, p^{(k)}, \tilde v^{(k)}}}a_k^\rho \,    \d(\Sigma^{R,t}_x)(\omega) \, dA^{(k)} \, dh^{(k)},
\end{gather*}
and for $c^{(k)} + 1 \leq \rho \leq c^{(k)}+d^{(k)}$ corresponding integrals $\widetilde I_k^\rho(x,\mu)$. Here $\phw_{\tau_k, p^{(k)}, \tilde v^{(k)}}$ denotes the weak transform regarded as a function of the variables $\omega, {A^{(k)},h^{(k)}}$, while ${\tau_k, p^{(k)},\tilde v^{(k)}}$ are considered as  parameters. Let us emphasize that the amplitudes $a_k^\rho$ are compactly supported. In view of \eqref{eq:19} we arrive for $x \in W_k$ at the decomposition 
\bqn 
I_k(x,\mu)=\sum_{\rho=1} ^{c^{(k)}}  I_k^\rho(x,\mu)+\sum_{\rho=c^{(k)}+1} ^{d^{(k)}}  \widetilde I_k^\rho(x, \mu)
\eqn 
up to terms of order $O(\mu^{-\infty})$, compare \cite[p.\ 57]{ramacher10}. As we shall see in Corollary \ref{cor:23.09.2015}, the weak transforms $\phw_{\tau_k, p^{(k)}, \tilde v^{(k)}}$ have no critical points in the $\alpha^{(k)}$-charts, which will imply that the integrals $  \widetilde I_k^\rho(x,\mu)$ contribute to $I(x,\mu)$  with terms of  order $O(\mu^{-\infty})$. If $G$ acts on $S_k$ only with isotropy type $(H_L)$,  we shall  see  in the next section  that in each of the $\theta^{(k)}$-charts  the weak transforms $^{(k)} \phw$ have clean critical sets, so that one can apply the stationary phase theorem  in order to obtain asymptotics for each of the $I_k^\rho(x,\mu)$. 
  But in general, $G$ will act on $S_k$ with singular orbit types, so that  neither $\mathcal{N}_k$ is   resolved, nor do the weak transforms $^{(k)} \widetilde \Phi^{wk}$ have clean critical sets, and we are forced to continue with the iteration.

\subsection*{Iteration step from $N-1$ to $N$} 
Denote by $\Lambda\leq L$  the maximal number of elements that a totally ordered subset of the set of isotropy types can have. Assume that  $2 \leq N <\Lambda$, and let   $\mklm{(H_{i_1}), \dots , (H_{i_{N}})}$ be a totally ordered subset of the set of isotropy types  such that $i_1 <  \dots < i_N<L$. Let 
 $f_{i_1}$,  $S_{i_1}$, as well as  $p^{(i_1)}\in M_{i_1}(H_{i_1})$ be defined as at the beginning of the iteration, and assume that $f_{i_1\dots i_{j}}$,  $S_{i_1\dots i_{j}},p^{(i_{j})},\dots$ have already been defined for $j<N$. 
For every fixed $p^{(i_{N-1})}$, denote by $ \gamma^{(i_{N-1})} ((S_{i_1\dots i_{N-1}})_{p^{(i_{N-1})}})_{i_{N}}$  the submanifold with corners of the closed $G_{p^{(i_{N-1})}}$-manifold 
$ \gamma^{(i_{N-1})}  ((S_{i_1\dots i_{N-1}})_{p^{(i_{N-1})}})$ from which all  orbit types  less than $G/H_{i_{N}}$ have been removed, and define $ \gamma^{(i_{N-1})} ((S_{i_1\dots i_{N-1}})_{p^{(i_{N-1})}})_{L}$ analogously. Consider  the invariant tubular neighborhood 
$$f_{i_1\dots i_{N}}:=\exp \circ \gamma^{(i_{N})}: \nu _{i_1\dots i_{N}} \rightarrow \gamma^{(i_{N-1})}  ((S_{i_1\dots i_{N-1}})_{p^{(i_{N-1})}})_{i_{N}}$$ of the set  $ \gamma^{(i_{N-1})} ((S_{i_1\dots i_{N-1}})_{p^{(i_{N-1})}})_{i_{N}}(H_{i_{N}})$, where $ \nu _{i_1\dots i_{N}}$ denotes its normal $G_{p^{(i_{N-1})}}$-vector bundle, and $\exp \circ \gamma ^{(i_N)}$ the corresponding equvariant diffeomorphism, and define $S_{i_1\dots i_{N}}$  as the sphere sub-bundle in $ \nu _{i_1\dots i_{N}}$, while  
 $$S_{i_1 \dots i_N}^+:=\mklm { v \in S_{i_1 \dots i_N}\mid  v=\sum s_i v_i^{(i_1 \dots i_N)}, \, s_{\rho_{i_N}}>0 }$$
for some $\rho_{i_N}$. 
Put 
\bqn
W_{i_1 \dots i_N}:= f_{i_1 \dots i_N}(\stackrel \circ D_1(\nu_{i_1 \dots i_N})), \qquad  W_{i_1 \dots i_{N-1}L}:=\mathrm{Int}(\gamma^{(i_{N-1})} ((S_{i_1\dots i_{N-1}})_{p^{(i_{N-1})}})_{L}),
\eqn
and denote the corresponding integrals in the decomposition of $I_{i_1 \dots i_{N-1}}^{\rho_{i_1} \dots \rho_{i_{N-1}}}(x, \mu)$ by $I_{i_1 \dots i_N}^{\rho_{i_1} \dots \rho_{i_{N-1}}}(x, \mu)$ and $I_{i_1 \dots i_{N-1}L}^{\rho_{i_1} \dots \rho_{i_{N-1}}}(x, \mu)$, respectively. Here we can assume that, modulo terms of order $O(\mu^{-\infty})$,  the $W_{i_1\dots i_N} \times G_{p^{(i_{N-1})}}$-support of the integrand in   $I_{i_1 \dots i_N}^{\rho_{i_1} \dots \rho_{i_{N-1}}}(\mu)$ is contained in a compactum of a tubular neighborhood  of the induced bundle $\pi^\ast _{i_1 \dots i_N} \mathrm{Iso} \,\gamma^{(i_{N-1})} ((S_{i_1\dots i_{N-1}})_{p^{(i_{N-1})}})_{i_{N}}(H_{i_{N}})$, where $\pi _{i_1 \dots i_N}:W_{i_1 \dots i_N}\rightarrow \gamma^{(i_{N-1})} ((S_{i_1\dots i_{N-1}})_{p^{(i_{N-1})}})_{i_{N}}(H_{i_{N}})$ denotes the canonical projection.
For  a given point $p^{(i_{N})}\in \gamma^{(i_{N-1})}   ((S^+_{i_1\dots i_{N-1}})_{p^{(i_{N-1})}})_{i_{N}}(H_{i_{N}})$,  consider further  the decomposition
\bqn 
 \g_{p^{(i_{N-1})}}= \g_{p^{(i_N)}}\oplus \g_{p^{(i_N)}}^\perp, 
\eqn
and set $d^{(i_N)}:=\dim \g_{p^(i_N)}^\perp$, $ e^{(i_N)}:=\dim \g_{p^(i_N)}$. This yields  the decomposition
\begin{gather}
\label{eq:gdecomp}
\g = \g_{p^{(i_1)}} \oplus \g_{p^{(i_1)}}^\perp =(\g_{p^{(i_2)}}\oplus \g_{p^{(i_2)}}^\perp) \oplus \g_{p^{(i_1)}}^\perp =\dots = \g_{p^{(i_N)}}\oplus \g_{p^{(i_N)}}^\perp \oplus \cdots \oplus \g_{p^{(i_1)}}^\perp.
\end{gather}
Denote by  $\mklm{ A_r^{(i_N)}(p^{(i_1)},\dots,p^{(i_N)})}$ an orthonormal basis of $\g_{p^(i_N)}^\perp$, and let $\left (\alpha^{(i_N)}_1,\dots,  \alpha^{(i_N)}_{d^{(i_N)}}\right )$ be corresponding coordinates. Further,  let   $\mklm{v_1^{(i_1\dots i_N)}, \dots ,v_{c^{(i_N)}}^{(i_1\dots i_N)} }$ be  an orthonormal frame in $\nu_{i_1\dots  i_N}$, and $(\theta^{(i_N)}_1,\dots,  \theta^{(i_N)}_{c^{(i_N)}})$  corresponding coordinates.
Now, let the blow-up $\zeta^{(1)}$ be defined as in the beginning of the iteration, and assume that the blow-ups $\zeta^{(j)}$ have already been defined for $j<N$. Put $\widetilde \M^{(j)} := B_{Z^{(j)}}( \widetilde \M^{(j-1)})$, $\widetilde \M ^{(0)}:= \M=M \times G$, and consider the blow-up 
 \bq
 \label{eq:montrans}
 \zeta^{(N)}: B_{Z^{(N)}}( \widetilde \M^{(N-1)} )\rightarrow  \widetilde \M^{(N-1)}, \qquad Z^{(N)}:= \bigcup_{i_1 < \dots < i_N<L}^\bullet Z_{i_1\dots i_N},
 \eq
 where the union is over all totally ordered subsets $\mklm{(H_{i_1}), \dots , (H_{i_{N}})}$ of $N$ elements with $i_1 < \dots < i_N < L$, and 
  \begin{gather*}
 Z_{i_1\dots i_N}\simeq \bigcup _{p^{(i_{1})},   \dots,  p^{(i_{N-1})}}   (-1,1)^{N-1} \times \mathrm{Iso} \, \gamma^{(i_{N-1})}((S_{i_1\dots i_{N-1}})_{p^{(i_{N-1})}})_{i_N} (H_{i_N})
 \end{gather*}
 are the  possible maximal singular loci of $(\zeta^{(1)} \circ \dots \circ \zeta^{(N-1)})^{-1}(\mathcal{N})$, compare \cite[Eq.\ (5.14)]{ramacher10}.  Denote   by  $ \zeta^{\rho_{i_1}}_{i_1} \circ \dots \circ  \zeta^{\rho_{i_1}\dots \rho_{i_N}}_{i_1\dots i_N}$ a local realization of the sequence of blow-ups $ \zeta^{(1)} \circ \dots \circ \zeta^{(N)}$  corresponding to   the totally ordered subset    $\mklm{(H_{i_1}), \dots, (H_{i_{N}})}$  in a set of charts labeled by the indices $\rho_{i_1},\dots ,\rho_{i_N}$. 
 As a consequence, we obtain  local factorizations of  the phase function according to 
 \bqn
\Phi \circ ( (\zeta_{i_1}^{\rho_{i_1}} \circ\dots \circ  \zeta_{{i_1} \dots {i_N}}^{\rho_{i_1} \dots \rho_{i_N}})\otimes \id_\eta)=\, ^{(i_1\dots i_N)} \widetilde \Phi^{tot}=\tau_{i_1} \cdots \tau_{i_N} \, ^{(i_1\dots i_N)}\widetilde \Phi^ {wk},
\eqn
see \cite[pp. 67]{ramacher10}. Assume now that the indices $\rho_{i_1},\dots ,\rho_{i_N}$ correspond to a set of $(\theta^{({i_1})},\dots ,\theta^{({i_N})})$-charts. Then $\zeta^{\rho_{i_1}}_{i_1} \circ \dots \circ  \zeta^{\rho_{i_1}\dots \rho_{i_N}}_{i_1\dots i_N}$ is explicitly given by 
\begin{gather*} 
(\tau_{i_1}, \dots, \tau_{i_N}, p^{(i_1)}, \dots, p^{(i_N)}, \tilde v ^{(i_N)}, A^{(i_1)}, \dots, A^{(i_N)}, h^{(i_N)}) \quad \longmapsto \quad (x_{i_1\dots i_N}^{\rho_{i_1}\dots \rho_{i_N}}, g_{i_1\dots i_N}^{\rho_{i_1}\dots \rho_{i_N}}) = (x,g),
\end{gather*}
where we set
\begin{align*}
\begin{split}
x_{i_j\dots i_N}^{\rho_{i_j}\dots \rho_{i_N}}&:=\exp_{p^{(i_j)}}[\tau_{i_j} \exp_{p^(i_{j+1})}[\tau_{i_{j+1}}\exp_{p^(i_{j+2})}[\dots [ \tau_{i_{N-2}}\exp_{p^(i_{N-1})}[\tau_{i_{N-1}}\exp_{p^{(i_N)}}[ \tau_{i_N} \tilde v ^{(i_N)}]]] \dots ]]], \\
g_{i_j\dots i_N}^{\rho_{i_j}\dots \rho_{i_N}}&:=\e{\tau_{i_j} \cdots \tau_{i_N}A^{(i_j)}}\e{\tau_{i_{j+1}} \cdots \tau_{i_N}A^{(i_{j+1})}}\cdots\e{\tau_{i_{N-1}}  \tau_{i_N}A^{(i_{N-1})}} \e{\tau_{i_N} A^{(i_N)}}h^{(i_N)}.
\end{split}
\end{align*}
In this situation we define
\begin{align}
\label{eq:N}
\begin{split}
 & \qquad\qquad  I_{i_1\dots i_N}^{\rho_{i_1}\dots \rho_{i_N}}(x, \mu) := \prod_{j=1}^N |\tau_{i_j}|^{\sum_{r=1}^j d^{(i_r)}}  \int_{ \widetilde \Xbf_{i_1\dots i_N}^{\rho_{i_1}\dots \rho_{i_N}}\times \Sigma^{R,t}_x }   \\ & \cdot e^{i\mu {\tau_1 \dots \tau_N} \, ^{(i_1\dots i_N)} \widetilde \Phi^{wk}_{\tau_{i_j}, p^{(i_j)},\tilde v^{(i_N)}}}   a_{i_1\dots i_N}^{\rho_{i_1}\dots \rho_{i_N}}   \d \omega  \d A^{(i_1)} \dots  \d A^{(i_N)}  \d h^{(i_N)},
\end{split}
\end{align}
compare \cite[Eq.\ (5.15)]{ramacher10}, where
\begin{itemize}
\item 
 $\widetilde \Xbf_{i_1\dots i_N}^{\rho_{i_1}\dots \rho_{i_N}}:=G_{p^{(i_{N})}}\times  \g_{p^{(i_{N})}}^\perp \times \cdots \times \g_{p^{(i_{1})}}^\perp$, 
 \item  $ ^{(i_1\dots i_N)} \widetilde \Phi^{wk}_{\tau_{i_j}, p^{(i_j)},\tilde v^{(i_N)}}$ denotes the weak transform regarded as a function on $\widetilde \Xbf_{i_1\dots i_N}^{\rho_{i_1}\dots \rho_{i_N}}\times \Sigma^{R,t}_x$, while the $\tau_{i_j}, p^{(i_j)},\tilde v^{(i_N)}$ are regarded as parameters,
\item the $a_{i_1\dots i_N}^{\rho_{i_1}\dots \rho_{i_N}} $
are amplitudes with compact support in a system of $(\theta^{(i_1)}, \dots, \theta^{(i_N)})$-charts labeled  by the indices $ \rho_{i_1}, \dots, \rho_{i_N}$,
\item $dA^{(i_1)},\dots, dA^{(i_N)}, dh^{(i_N)}$ are suitable measures on $ \g_{p^{(i_{1})}}^\perp, \dots,   \g_{p^{(i_{N})}}^\perp$, and $G_{p^{(i_{N})}}$,  respectively.
\end{itemize}
Similarly, one defines analogous integrals $\widetilde I_{i_1\dots i_N}^{\rho_{i_1}\dots \rho_{i_N}}(x, \mu)$ in the $(\theta^{(i_1)}, \dots, \theta^{(i_{N-1})}, \alpha^{(i_N)} )$-charts. As we shall see, $I_x(\mu)$ will be  given by a sum involving integrals of the type  $I_{i_1\dots i_N}^{\rho_{i_1}\dots \rho_{i_N}}(x, \mu)$, compare   \eqref{eq:65}.

Now, for each $p^{(i_{N-1})}$, the isotropy group $G_{p^{(i_{N-1})}}$ acts on $ \gamma^{(i_{N-1})} ((S_{i_1\dots i_{N-1}})_{p^{(i_{N-1})}})_{i_{N}}$ by the isotropy types $(H_{i_N}), \dots, (H_L)$. The types occuring in $W_{i_1 \dots i_N}$ constitute a subset of these, and   $G_{p^{(i_{N-1})}}$ acts on  the sphere bundle $S_{i_1\dots i_N}$ over the submanifold $\gamma^{(i_{N-1})}((S_{i_1\dots i_{N-1}})_{p^{(i_{N-1})}})_{i_N} (H_{i_N})\subset W_{i_1 \dots i_N}$  with one type less.

 \medskip

\subsection*{End of iteration}

After  $N=\Lambda-1$ steps, the end of the iteration is reached, yielding a strong desingularization  of $\mathcal{N}$, see \cite[Theorem 5.1]{ramacher10}, and a factorization of the phase function $\Phi_x$  that will allow us to interpolate between the different asymptotics for  the integrals $I_x(\mu)$ described in Theorem \ref{thm:12.05.2015} (a).

\section{The singular equivariant local Weyl law. Caustics  and concentration of eigenfunctions}
\label{sec:5}

We are now ready to give an asymptotic formula for the integrals \eqref{eq:N} that will result in a corresponding description of the integrals \eqref{eq:03.05.2015} in case that  $x=y$. With the notation as before, consider for  fixed $1 \leq N \leq \Lambda-1$ a maximal, totally ordered subset $\mklm{(H_{i_1}),\dots, (H_{i_N})}$ of non-principal isoptropy types in the sense that if there is an isotropy type $(H_{i_{N+1}})$ with $i_N < i_{N+1}$ such that $\mklm{(H_{i_1}),\dots, (H_{i_{N+1}})}$ is a totally ordered subset, then $(H_{i_{N+1}})=(H_L)$. Assign to each such subset   the sequence  of consecutive local blow-ups  
\bqn 
\mathcal{Z}^{\rho_{i_1}\dots \rho_{i_N}}_{i_1\dots i_N}:=  (\zeta^{\rho_{i_1}}_{i_1} \circ \dots \circ  \zeta^{\rho_{i_1}\dots \rho_{i_N}}_{i_1\dots i_N} \circ (\delta_{i_1\dots i_N}\otimes \id) )\otimes \id_\eta
\eqn
where $\delta_{i_1\dots i_N}$ denotes the sequence of local quadratic transformations
 \begin{align*}
\delta_{i_1\dots i_N}: (\sigma_{i_1}, \dots \sigma_{i_N}) &\mapsto \sigma_{i_1}( 1, \sigma_{i_2}, \dots, \sigma_{i_N})= (\sigma_{i_1}', \dots ,\sigma_{i_N}')\mapsto \sigma_{i_2}'(\sigma_{i_1}',1,\dots, \sigma_{i_N}')= (\sigma_{i_1}'', \dots, \sigma_{i_N}'')\\
 &\mapsto \sigma_{i_3}''(\sigma_{i_1}'',\sigma_{i_2}'', 1,\dots, \sigma_{i_N}'')= \cdots \mapsto \dots = (\tau_{i_1}, \dots ,\tau_{i_N}).
\end{align*}
  The global morphism induced by the local transformations $\mathcal{Z}_{i_1\dots i_N}^{\rho_{i_1}\dots \rho_{i_N}}$  is then denoted by 
  \bqn
\mathcal{Z}: \widetilde {\bf X} \rightarrow {\bf X}:= T^\ast M \times G, 
\eqn
and constitutes a partial desingularization of the critical set $\Ccal$, see \cite[Section 9]{ramacher10}.  Pulling  the phase function \eqref{eq:phase} back along   the maps  $\mathcal{Z}^{\rho_{i_1}\dots \rho_{i_N}}_{i_1\dots i_N}$ then yields the local factorization
\bqn
\Phi \circ \mathcal{Z}^{\rho_{i_1}\dots \rho_{i_N}}_{i_1\dots i_N} =\, ^{(i_1\dots i_N)} \widetilde \Phi^{tot}=\tau_{i_1}(\sigma) \dots \tau_{i_N}(\sigma) \, ^{(i_1\dots i_N)}\widetilde \Phi^ {wk}, 
\eqn
 where the  $\tau_{i_j}$ are monomials in the desingularization parameters $\sigma_{i_1},\dots, \sigma_{i_N}$. The principal result in \cite{ramacher10} is 
 
 \begin{thm}
 \label{thm:maincrelle}
In any of the $(\theta^{({i_1})},\dots ,\theta^{({i_N})})$-charts, the critical sets of the weak transforms $\, ^{(i_1\dots i_{N})}\widetilde \Phi^ {wk}$ are smooth submanifolds  in the resolution space of codimension $2\kappa$, and the Hessians $\mathrm{Hess } \, ^{(i_1\dots i_{N})}\widetilde \Phi^ {wk}$ are transversally non-degenerate. In other words, the weak transforms $\, ^{(i_1\dots i_{N})}\widetilde \Phi^ {wk}$ have clean critical sets in the mentioned charts. On the other hand,   the weak transforms $\, ^{(i_1\dots i_{N})}\widetilde \Phi^ {wk}$ have no critical points  in any of  the $(\theta^{({i_1})},\dots , \theta^{({i_{N-1}})}, \alpha^{({i_N})})$-charts.  
 \end{thm}
 \begin{proof}
 See \cite[Theorems 6.1 and 7.2, as well as p.\ 90]{ramacher10}.
 \end{proof}

In order to prove Theorem \ref{thm:maincrelle} for the $(\theta^{({i_1})},\dots ,\theta^{({i_N})})$-charts one  first shows that 
\bqn 
\gd _{\eta, \alpha^{(i_1)}, \dots, \alpha^{(i_N)}, h^{(i_N)}} \, ^{(i_1\dots i_N)}\widetilde \Phi^ {wk}=0 \quad \Longrightarrow \quad \gd _{\sigma_{i_1}, \dots, \sigma_{i_N}, p^{(i_1)}, \dots, p^{(i_N)},\tilde v^{(i_N)}} \, ^{(i_1\dots i_N)}\widetilde \Phi^ {wk}=0,
\eqn
see\cite[p.\ 80]{ramacher10}.    If therefore 
$$
\, ^{(i_1\dots i_N)}\widetilde \Phi^ {wk}_{\sigma_{i_j}, p^{(i_j)},\tilde v^{(i_N)}}(\alpha^{(i_j)},  h^{(i_N)},\eta)
$$ 
denotes the weak transform of $\Phi$ regarded as a function of the variables $(\alpha^{(i_1)},\dots, \alpha^{(i_N)}, h^{(i_N)},\eta)$ alone, while the variables $(\sigma_{i_1},\dots,\sigma_{i_N}, p^{(i_1)},\dots, p^{(i_N)},\tilde v^{(i_N)})$ are kept fixed at constant values, its critical set is given by the transversal intersection
\bqn 
\Crit \big ( \, ^{(i_1\dots i_N)}\widetilde \Phi^ {wk}_{\sigma_{i_j}, p^{(i_j)},\tilde v^{(i_N)}}\big )=\Crit \big ( \, ^{(i_1\dots i_N)}\widetilde \Phi^ {wk}\big )  \cap \mklm{\sigma_{i_j}, p^{(i_j)},\tilde v^{(i_N)} = \, \, \text{constant}}.
\eqn
In fact, $\Crit \big ( \, ^{(i_1\dots i_N)}\widetilde \Phi^ {wk}\big )$ turns out to be a fibre bundle \cite[p.\ 78]{ramacher10}, and the critical set of the phase function $\, ^{(i_1\dots i_N)}\widetilde \Phi^ {wk}_{\sigma_{i_j}, p^{(i_j)},\tilde v^{(i_N)}}$ is equal to the fiber over $(\sigma_{i_j}, p^{(i_j)},\tilde v^{(i_N)})$ of this bundle,  in particular  being a smooth submanifold. Furthermore, \cite[Lemma 7.1]{ramacher10}   implies that  the transversal Hessian  of $\, ^{(i_1\dots i_N)}\widetilde \Phi^ {wk}$ is non-degenerate iff the transversal Hessian of  $\, ^{(i_1\dots i_N)}\widetilde \Phi^ {wk}_{\sigma_{i_j}, p^{(i_j)},\tilde v^{(i_N)}}$ is non-degenerate, the latter fact being proved in \cite[Proposition 7.4]{ramacher10} for the critical case $\sigma_{i_1} \cdots \sigma_{i_N}=0$. Thus, we arrive at

\begin{cor}
\label{cor:23.09.2015}
In any of the $(\theta^{({i_1})},\dots ,\theta^{({i_N})})$-charts, the weak transforms $\, ^{(i_1\dots i_N)}\widetilde \Phi^ {wk}_{\sigma_{i_j}, p^{(i_j)},\tilde v^{(i_N)}}$  have clean critical sets of codimension $2\kappa$ as functions on $\widetilde \Xbf_{i_1\dots i_N}^{\rho_{i_1}\dots \rho_{i_N}}\times \Sigma^{R,t}_x$.  They do not have  critical points  in  the $(\theta^{({i_1})},\dots , \theta^{({i_{N-1}})}, \alpha^{({i_N})})$-charts.
\end{cor}
\begin{proof}
The assertion is a direct consequence of the foregoing explanations and transversality arguments like those given in \cite[Section 7]{ramacher10}.
\end{proof}

From this we immediately deduce

\begin{proposition}
\label{prop:18.08.2016}
For every $\tilde N\in \N$, $\eps> 0$, any  $(\theta^{({i_1})},\dots ,\theta^{({i_N})})$-chart labeled by the indices $\rho_{i_1},\dots,  \rho_{i_N}$, and $x=x_{i_1\dots i_N}^{\rho_{i_1}\dots \rho_{i_N}}$ in $Y$ (or in $Y \cap M_\mathrm{prin}$ and $\eps\geq 0$)  one has the asymptotic formula
\bqn
  I_{i_1\dots i_N}^{\rho_{i_1}\dots \rho_{i_N}}(x, \mu)= \prod_{j=1}^N |\tau_{i_j}|^{\dim G-\dim H_{i_j}} \left (\sum_{k=0}^{\tilde N-1}  \frac {\,^k \mathcal{Q}_{i_1\dots i_N}^{\rho_{i_1}\dots \rho_{i_N}}(x)}{(\mu |\tau_{i_1} \cdots \tau_{i_N}|+\eps)^{\kappa+k}}  +\mathcal{R}_{\tilde N}(x,\mu)\right ),
\eqn
where  the $^k \mathcal{Q}_{i_1\dots i_N}^{\rho_{i_1}\dots \rho_{i_N}}(x)$ and $\mathcal{R}_{\tilde N}(x,\mu)$ are explicitly known coefficients that are uniformly bounded in  $x$ by $(A^{(i_j)},h^{(i_N)})$-derivatives of the amplitude $ a_{i_1\dots i_N}^{\rho_{i_1}\dots \rho_{i_N}}$ up to order $2k$ and $2\tilde N+\kappa +1$, respectively, and 
\bqn 
\mathcal{R}_{\tilde N}(x,\mu)=O\big ((\mu |\tau_{i_1} \cdots \tau_{i_N}|+\eps)^{-\kappa -\tilde N}\big ).
\eqn
In particular, with $\widetilde \Phi^{wk}:= \,^{(i_1\dots i_N)} \widetilde \Phi^{wk}_{\sigma_{i_j}, p^{(i_j)},\tilde v^{(i_N)}}$ we have 
\bqn 
^0 \mathcal{Q}_{i_1\dots i_N}^{\rho_{i_1}\dots \rho_{i_N}}(x)=(2 \pi)^\kappa \int_{\mathrm{Crit} \, \widetilde \Phi^{wk}} \frac{   a_{i_1\dots i_N}^{\rho_{i_1}\dots \rho_{i_N}} }{\big | \det \mathrm{Hess} \, \widetilde \Phi^{wk}_{|N \mathrm{Crit} \, \widetilde \Phi^{wk} } \big |^{1/2}}. 
\eqn
 If the amplitude $a$  factorizes according to $a(x,y,\omega,g)=a_1(x,y,\omega) \, a_2(x,y,g)$, the remainder can also be estimated by derivatives of $a_{i_1\dots i_N}^{\rho_{i_1}\dots \rho_{i_N}}$ with respect to $(A^{(i_j)},h^{(i_N)})$  up to order  $2\tilde N+\lfloor \kappa/2+1\rfloor$.
\end{proposition}
\begin{proof}
By definition we have  $d^{(i_r)}=\dim H_{i_{r-1}}-\dim H_{i_r}$ with $H_{i_{0}}:=G$. Consequently, $\sum_{r=1}^j d^{(i_r)}=\dim G-\dim H_{i_j}$. By Corollary \ref{cor:23.09.2015} we can apply Theorem \ref{thm:SP} and the final remarks in the appendix  to the integral \eqref{eq:N} with asymptotic parameter $\mu |\tau_{i_1}(\sigma) \cdots \tau_{i_N}(\sigma)|+\eps$, yielding the assertion, since $e^{-i\eps \, \widetilde \Phi^{wk}} =1$ on $\mathrm{Crit} \, \widetilde \Phi^{wk}$. In particular,  \eqref{eq:25.07.2017} implies that the coefficients and the remainder in the expansion are uniformly bounded in $x$, since $\det \mathrm{Hess}\,  \widetilde \Phi^{wk}$ is uniformly bounded away from zero with respect to the parameters $\sigma_{i_j}, p^{(i_j)},\tilde v^{(i_N)}$. Furthermore, $\mathrm{Crit} \, \widetilde \Phi^{wk}$ is given as a Cartesian product of $G_{\tilde v^{(i_N)}}$ with a certain subspace in $T^\ast_xM$ intersected with $\Sigma^{R,t}_x$, compare \cite[p.\ 78]{ramacher10}, so that Remark \ref{rem:25.07.2017} applies.
\end{proof}

\begin{proposition} In any  $(\theta^{({i_1})},\dots ,\theta^{({i_{N-1}})}, \alpha^{({i_N})})$-chart labeled by the indices $\rho_{i_1},\dots,  \rho_{i_N}$ one has
\label{prop:nonstat}
\bqn 
 \widetilde  I_{i_1\dots i_N}^{\rho_{i_1}\dots \rho_{i_N}}(x, \mu)= O(\mu^{-\infty})
\eqn
uniformly in $x$. 
\end{proposition}
\begin{proof}
This is an immediate consequence of the previous corollary and  the non-stationary phase principle \cite[Theorem 7.7.1]{hoermanderI}. 
\end{proof}

Now,  let us consider the oscillatory integral $I_x(\mu) $ introduced in \eqref{eq:11.9.2017}. Transforming it under the global morphism  $\mathcal{Z}$  we obtain with our previous notation for $x=x_{i_1\dots i_N}^{\rho_{i_1}\dots \rho_{i_N}}$  the decomposition
\begin{align}
\label{eq:65}
I_x(\mu)&=\sum _{N=1}^{\Lambda-1} \Big ( \sum_{\stackrel{i_1<\dots< i_{N-1}<L}{ \rho_{i_1}, \dots ,\rho_{i_{N-1}}}}   I_{i_1\dots i_{N-1} L}^{\rho_{i_1} \dots \rho_{i_{N-1}}}(x, \mu)+ \sum_{\stackrel{i_1<\dots< i_{N}}{ \rho_{i_1}, \dots ,\rho_{i_{N}}}} I_{i_1\dots i_N}^{\rho_{i_1} \dots \rho_{i_{N}}}(x, \mu) \Big ) + \mathcal{R}(x,\mu),
\end{align}
where  the first multiple  sum is one over arbitrary  totally ordered subsets of non-principal isotropy types and corresponding charts, while the second multiple sum is one over arbitrary maximal   totally ordered subsets of non-principal isotropy types and corresponding charts,  and $\mathcal{R}(\mu,x)$ denotes the non-stationary contributions of order $O(\mu^{-\infty})$ 
that arise by localizing the relevant integrals to tubular neighborhoods of the relevant critical sets, or  correspond to integrals  over charts of the resolution spaces where the
weak transforms of the phase functions do not have critical points, compare \cite[Eq. (9.1)]{ramacher10}. Here $I_{i_1\dots i_N}^{\rho_{i_1} \dots \rho_{i_{N}}}(x, \mu)=0$ unless $x=x_{i_1\dots i_N}^{\rho_{i_1}\dots \rho_{i_N}}$ lies in the corresponding chart, and similarly for $I_{i_1\dots i_{N-1} L}^{\rho_{i_1} \dots \rho_{i_{N-1}}}(x, \mu)$.  Since the latter integrals  have an analogous asymptotic description than the one given for the  integrals $I_{i_1\dots i_{N} }^{\rho_{i_1} \dots \rho_{i_{N}}}(x, \mu) $ in Proposition \ref{prop:18.08.2016} we arrive at

\begin{thm} 
\label{thm:31.10.2015}
For every $\tilde N\in \N$, $x \in Y$ and $\eps >0$ (or $x \in Y\cap M_\mathrm{prin}$ and $\eps \geq 0$) one has  
\begin{gather*} 
I_x(\mu)= \sum _{N=1}^{\Lambda-1}\, \sum_{\stackrel{i_1<\dots< i_{N-1}<L}{ \rho_{i_1}, \dots ,\rho_{i_{N-1}}}}  \prod_{l=1}^{N-1} |\tau_{i_l}|^{\dim G-\dim H_{i_l}} \\  
\cdot \left [    \sum_{k=0}^{\tilde N-1} \frac {^k\mathcal{P}_{i_1\dots i_{N-1}L }^{\rho_{i_1} \dots \rho_{i_{N-1}}}(x)}{(\mu |\tau_{i_1} \cdots \tau_{i_{N-1}}|+\eps)^{\kappa+k}}  +O\big ((\mu |\tau_{i_1} \cdots \tau_{i_{N-1}}|+\eps)^{-\kappa -\tilde N}\big )  \right ]\\ 
 +  \sum _{N=1}^{\Lambda-1}\, \sum_{\stackrel{i_1<\dots< i_{N}}{ \rho_{i_1}, \dots ,\rho_{i_{N}}}}  \prod_{l=1}^{N} |\tau_{i_l}|^{\dim G-\dim H_{i_l}}   \left [  \sum_{k=0}^{\tilde N-1} \frac {^k\mathcal{Q}_{i_1\dots i_{N} }^{\rho_{i_1} \dots \rho_{i_{N}}}(x)}{(\mu |\tau_{i_1} \cdots \tau_{i_N}|+\eps)^{\kappa+k}} +O\big ((\mu |\tau_{i_1} \cdots \tau_{i_N}|+\eps)^{-\kappa -\tilde N}\big )  \right ]
 \end{gather*}
 up to terms of order $O(\mu^{-\infty})$, where the multiple sums run over arbitrary totally ordered subsets  and arbitrary maximal  totally ordered subsets of non-principal isotropy types, respectively. Furthermore, all coefficients and remainders are given  explicitly in terms of distributions on the resolution space, and   are  uniformly bounded in  $x$ by $G$-derivatives of the corresponding amplitudes up to order $2k$ and $ 2\tilde N+\kappa +1$, respectively.  If the amplitude $a$  factorizes according to $a(x,y,\omega,g)=a_1(x,y,\omega) \, a_2(x,y,g)$, the remainder can also be estimated by $G$-derivatives 
   up to order  $2\tilde N+\lfloor \kappa/2+1\rfloor$.
\end{thm}
\qed 

\medskip

Theorem \ref{thm:31.10.2015}  gives a simultaneous description of the competing asymptotics $\mu  \to +\infty $ and $\tau_{i_j}\to 0$, and for $\eps >0$ interpolates between the different asymptotics  in Theorem \ref{thm:12.05.2015} (a). For $\eps=0$, it  yields a  description of the singular behaviour of the coefficients in the expansion of $I_x(\mu)$ in Theorem \ref{thm:12.05.2015} (a) as $x\in M_\mathrm{prin}$ approaches singular orbits.  Note that the factors $ |\tau_{i_l}|^{\dim G-\dim H_{i_l}} $  in the expansion of Theorem  \ref{thm:31.10.2015} reflect the fact that the coefficients  become more singular as the dimension of the stabilizer groups $H_{i_l}$ become large, that is, as one approaches more and more singular orbits, answering for the different asymptotics in Theorem \ref{thm:12.05.2015} (a) given by the  exponents $\kappa_x=\dim \mathcal{O}_x$.\footnote{Indeed, assume that   $M_\mathrm{prin} \ni x_{i_1\dots i_N}^{\rho_{i_1}\dots \rho_{i_N}} \to y \in M(H_{i_q})$ in such a way  that  the index  $\tau_{i_q}$ goes to zero with rate $\tau_{i_q}\approx \mu^{-1} \to 0$. Then,  if $\kappa=\dim G$,
\bqn 
\prod_{l=1}^N \frac { |\tau_{i_l}|^{\dim G-\dim H_{i_l}}} {(\mu |\tau_{i_1} \cdots \tau_{i_N}|)^{\kappa}}= \prod_{l=1}^N \frac{|\tau_{i_l}|^{-\dim H_{i_l}}}{\mu^\kappa}\approx O(\mu^{-\dim G+\dim H_{i_q}})=O(\mu^{-\dim \mathcal{O}_{y}}).
\eqn
}
For an exceptional orbit of type $(H_{i_l})$ one has $\dim G - \dim H_{i_l}=\kappa$, so that the corresponding factors $|\tau_{i_l}|^\kappa$ cancel each other, in concordance with Theorem \ref{thm:12.05.2015} (a), by which the summands in the expansion of $I_x(\mu) $ in Theorem \ref{thm:31.10.2015} must stay bounded as one approaches exceptional orbits. Besides, note that the terms with $k\geq 1$ involve derivatives with respect to $g$  that give rise to additional positive powers in the desingularization parameters.
In the same way that Theorem \ref{thm:main} was deduced from Theorem \ref{thm:12.05.2015} (a), the previous theorem allows us to derive  the asymptotic formula for the reduced spectral function we were looking for. First, one deduces

\begin{proposition}[\bf Singular point-wise asymptotics for the kernel of the equivariant approximate projection]
\label{prop:15.11.2015}
For    arbitrary integers $\tilde N_1,\tilde N_2=0,1,2,3, \dots$, fixed $\gamma \in \widehat G$, $x \in M$ and $\eps > 0$ (or $x \in M_\mathrm{prin}\cup M_\mathrm{except}$ and $\eps\geq 0$)  one has for $\mu \to + \infty$ the asymptotic expansion
\begin{gather*}
  K_{ \widetilde \chi_\mu \circ \Pi_\gamma } (x,x) =  \frac{\mu^{n-1}d_\gamma}{(2\pi)^{n-\kappa}}
 \sum_{j=0}^{\tilde N_1-1}  \mu^{-j}  \sum _{N=1}^{\Lambda-1} \, \sum_{i_1<\dots< i_{N}} \prod_{l=1}^{N} |\tau_{i_l}|^{\dim G-\dim H_{i_l}}\\ 
\cdot \left [ \sum_{k=0}^{\tilde N_2-1}  \frac {   \mathcal{L}^{j,k}_{i_1\dots i_{N}}(x,\gamma)}{(\mu |\tau_{i_1} \cdots \tau_{i_{N}}|+\eps)^{\kappa+k}}  +O_\gamma\big ((\mu |\tau_{i_1} \cdots \tau_{i_{N}}|+\eps)^{-\kappa -\tilde N_2}\big )  \right ] 
\end{gather*}
up to terms of order $O(\mu^{n-\tilde N_1-1})$, where the multiple sum runs over all possible totally ordered subsets $\mklm{(H_{i_1}),\dots, (H_{i_N})}$ of singular isotropy types, and all coefficients and remainders are explicitly given by distributions on the resolution space bounded uniformly in $x$ by derivatives of $\gamma$ up to order $2k$ and $2\tilde N_2+\lfloor \kappa/2 + 1\rfloor$, respectively. For $\mu \to -\infty$, the function $K_{\widetilde \chi_\mu \circ \Pi_\gamma}(x,x)$ is rapidly decreasing in $\mu$. 
\end{proposition}
\begin{proof}
The assertion  follows from Corollary \ref{cor:12.05.2015} by applying Theorem \ref{thm:31.10.2015} to the integrals \eqref{eq:02.05.2015},  combing the coefficients $^k\mathcal{P}_{i_1\dots i_{N-1}L }^{\rho_{i_1} \dots \rho_{i_{N-1}}}(x)$ and $^k\mathcal{Q}_{i_1\dots i_{N-1} }^{\rho_{i_1} \dots \rho_{i_{N-1}}}(x)$ in the expansion of $I_x(\mu) $, and collecting  the terms from different charts corresponding to the same subset of isotropy types. Then, one merges the contributions from exceptional and principal isotropy types, taking into account that by Theorem \ref{thm:12.05.2015} (a)  the summands in Theorem \ref{thm:31.10.2015} must stay bounded as one approaches exceptional orbits, all coefficients and remainders in the  expansions being smooth in $R,t$ and uniformly bounded in $x$ by derivatives of $\gamma$ up to order $2k$ and $2\tilde N_2+\lfloor \kappa/2 + 1\rfloor$, respectively.
\end{proof}

Using standard Tauberian arguments we obtain as our third main result

\begin{thm}[\bf Singular equivariant local Weyl law]
\label{thm:15.11.2015}
Let $M$ be  a closed connected Riemannian manifold $M$ of dimension $n$ with an isometric and effective action of a continuous compact   Lie group $G$ and $P_0$ a $G$-invariant elliptic classical pseudodifferential operator on $M$ 
of degree $m$. Let  $p(x,\xi)$ be its  principal symbol, and assume that $P_0$    is positive and symmetric. Denote its  unique self-adjoint extension by  $P$, and for a given $\gamma \in \widehat G$ let $e_\gamma(x,y,\lambda)$ be its reduced spectral counting function. Write 
$\kappa$ for the dimension of an $G$-orbit in $M$  of principal type and   $d_\gamma$ for  the dimension of an irreducible $G$-representation $\pi_\gamma$ of class $\gamma$. Then, for $x \in M_\mathrm{prin}\cup M_\mathrm{except}$ one has the asymptotic formula
 \begin{gather*}
\left |e_\gamma(x,x,\lambda)- \frac{d_\gamma \lambda^{\frac{n-\kappa}{m}}}{(2\pi)^{n-\kappa}} \sum_{N=1}^{\Lambda-1} \,  \sum_{i_1<\dots< i_{N}} \, \prod_{l=1}^{N}   |\tau_{i_l}|^{\dim G- \dim H_{i_l}-\kappa}  \mathcal{L}_{i_1\dots i_{N} }^{0,0}(x,\gamma)    \right | \\
\leq C_\gamma \, \lambda^{\frac{n-\kappa-1}m} \sum_{N=1}^{\Lambda-1}\, \sum_{i_1<\dots< i_{N}}   \prod_{l=1}^N  |\tau_{i_l}|^{\dim G- \dim H_{i_l}-\kappa-1} 
\end{gather*}
 as $\lambda \to +\infty$, where  the multiple sum runs over all possible totally ordered subsets $\mklm{(H_{i_1}),\dots, (H_{i_N})}$ of singular isotropy types, and the coefficients satisfy the bounds
$
 \mathcal{L}_{i_1\dots i_{N}}^{0,0}(x,\gamma) \ll \norm{\gamma}_\infty
 $ 
 uniformly in $x$, while 
  \bqn 
  C_\gamma \ll d_\gamma \sup_{l\leq \lfloor \kappa/2+3\rfloor} \norm{D^l \gamma}_\infty
  \eqn 
  is a constant independent of $x$ and $\lambda$,  the $D^l$ are differential operators on $G$ of order $l$, and the $\tau_{i_j}=\tau_{i_j}(x)$ parameters satisfying  $|\tau_{i_j}|\approx \dist (x, M(H_{i_j}))$.
\end{thm}
\begin{proof}
The assertion follows by integrating the expression for $K_{\Pi_\gamma  \circ \widetilde \chi_\mu} (x,x) $ in  Proposition \ref{prop:15.11.2015} with respect to $\mu$ from $-\infty$ to $\sqrt[m]\lambda$ for the values  $\eps=0$, $\tilde N_1=\kappa+1$, $\tilde N_2=1$ with the arguments given in the proof of Theorem \ref{thm:main}, noting that $ \dim G - \dim H_{i_l}-\kappa\leq 0$ for all $i_l$.
\end{proof}

\begin{rem}
Again, if $G$ is a connected compact semisimple Lie group, in terms of the highest weight $\Lambda_\gamma\in \t^\ast_\C$ of $\gamma \in \widehat G$ we have
$C_{\gamma}  \ll  |\Lambda_\gamma|  ^{2|\Sigma^+|+\lfloor \kappa/2 +3\rfloor}$,  compare Remark \ref{rem:22.5.2017}.
\end{rem}

As an immediate consequence this yields 

\begin{cor}[\bf Singular point-wise bounds for isotypic spectral clusters]
\label{cor:2.12.2015}
In the setting of Theorem \ref{thm:15.11.2015} we have 
\bqn 
\sum_{\stackrel{\lambda_j \in (\lambda,\lambda+1],}{ e_j \in \L^2_\gamma(M)}} |e_j(x)|^2  \leq \begin{cases}  C \, \lambda^{\frac{n-1}m}, & x\in M_\mathrm{sing}, \\
& \\
C_\gamma \,  \lambda^{\frac{n-\kappa-1}m} \sum\limits_{N=1}^{\Lambda-1}\, \sum\limits_{i_1<\dots< i_{N}}\prod\limits_{l=1}^N  |\tau_{i_l}|^{\dim G- \dim H_{i_l}-\kappa-1}, & x\in M-M_\mathrm{sing},  \end{cases}
\eqn
with $C>0$ independent of $\gamma$.  In particular, the bound holds for each individual $e_j \in \L^2_\gamma(M)$ with $\lambda_j \in (\lambda, \lambda+1]$. 
\end{cor}
\qed

We would like to remark that the expansion in Theorem \ref{thm:15.11.2015} is only meaningful if $\lambda$ is sufficiently large compared to the  desingularization parameters $\tau_{i_l}$, more precisely, if
\bqn 
\lambda^{1/m} \prod_l |\tau_{i_l}|>1
\eqn
for all possible combinations of the $\tau_{i_l}$. While  \eqref{eq:29.10.2015} describes the asymptotics of the equivariant spectral function for arbitrary, but  fixed $x\in M$, Theorem \ref{thm:15.11.2015} gives a uniform description of  the  behaviour of the coefficients as $x\in M_\mathrm{prin}$ approaches singular orbits.

 An asymptotic formula for $e_\gamma(x,x,\lambda)$ that  interpolates between the various asymptotic behaviours  in Theorem \ref{thm:main}, in the same way than Theorem \ref{thm:31.10.2015}   interpolates between the different asymptotics  in Theorem \ref{thm:12.05.2015} (a) can be obtained by integrating the expression for $K_{\Pi_\gamma  \circ \widetilde \chi_\mu} (x,x) $ in  Proposition \ref{prop:15.11.2015} with respect to $\mu$ from $-\infty$ to $\sqrt[m]\lambda$ for the values  $\eps=1$, $\tilde N_1=\kappa+1$, $\tilde N_2=1$ with the arguments given in the proof of Theorem \ref{thm:main}. This leads to expressions for  $e_\gamma(x,x,\lambda)$ which  involve the hypergeometric function, in the same way than the associated Legendre polynomials are given in terms of that function \cite[p.\ 188]{hobson}.

\begin{ex} 
\label{ex:7.12.2015}
To illustrate  the desingularization process and our results, let us resume  Example \ref{ex:sphere}, where we considered the action of $G=\SO(2)$ on the standard $2$-sphere $M=S^2\subset \R^3$ by rotations around the $x_3$-axis. The isotropy types are $H_1=\SO(2)$ and $H_2=\mklm{e}$, and the set of maximally singular orbits $M_1(H_1)=\mklm{x_N,x_S}$ is disconnected in this case. Instead of working with the covering \eqref{eq:872}, we can cover $S^2$ with the two charts $Y_1:=S^2-\mklm{x_N}$ and $Y_2:=S^2 - \mklm{x_S}$ by  introducing geodesic polar coordinates  $x=\exp_{x_S} (\tau_1 \tilde v)$ and $x=\exp_{x_N} (\tau_2 \tilde v)$ around the poles,  respectively, where  $\tilde v \in S^1$, and $\tau_i> 0$  equals the induced Riemannian distance of $x$ to the corresponding  pole. Note that $\g_{x_N}^\perp=\g_{x_S}^\perp=\mklm{0}$, so that it is not necessary to perform a blow-up in the group variables, and no additional $O(\mu^{-\infty})$-terms arise. After one iteration, the action is desingularized, and one obtains in agreement with Theorem \ref{thm:31.10.2015} for arbitrary $\tilde N \in \N$ and $\eps \geq 0$ the asymptotic formula
\bqn 
I_x(\mu)= \sum_{i=1,2} \left [ \sum_{k=0}^{\tilde N-1}  \, ^k\mathcal{Q}^i(x) \, (\mu \tau_i + \eps)^{-1-k} + O((\mu\tau_i +\eps)^{-1-\tilde N})\right ],
\eqn
all coefficients being bounded in $x$. In particular, setting $\eps=0$ one sees that the leading coefficient in Theorem \ref{thm:12.05.2015} (a) is given by
\bqn 
2 \pi \, \mathcal{Q}_0(x)= \frac 1 {\tau_1} \, ^0\mathcal{Q}^1(x)+ \frac 1 {\tau_2} \, ^0\mathcal{Q}^2(x), \qquad x \not=x_N,\, x_S,
\eqn
which describes its singular behaviour as one approaches the fixed points. 
This implies for the reduced spectral counting function of the Laplace-Beltrami operator $-\Delta$ on $S^2$ the asymptotics
 \bqn
 \label{eq:1599}
e_m(x,x,\lambda)-\frac{ \sqrt \lambda }{2 \pi}  \frac{ \mathcal{L}(x)}{\text{dist}(x,\mklm{x_N,x_S})}  \ll    \frac {1+|m|^3}{\text{dist}^2(x,\mklm{x_N,x_S})} , \qquad  m \in \Z, \, x \not=x_N,\, x_S,  
\eqn
$\mathcal{L}(x)$  being bounded  in $x$, provided that $\sqrt\lambda \, \text{dist}(x,\mklm{x_N,x_S})>1$,  in agreement with Theorem \ref{thm:15.11.2015}.  From this, we immediately deduce the following pointwise bounds for spherical harmonics. Let $Y_{k,m}$ be 
 the classical spherical functions with $ k \in \N, m \in \Z\simeq \widehat{\SO(2)}$, $|m| \leq l$ satisfying 
\bqn 
-\Delta Y_{k,m} = \lambda_k \, Y_{k,m}, \qquad \lambda_k =k(k+1).
\eqn
Then, from 
\bqn 
e_m(x,x,\lambda+1)-e_m(x,x,\lambda) = \sum_{ \lambda_k \in (\lambda,  \lambda+1]}  \big | Y_{k,m}(x)\big |^2
\eqn
 one  directly infers for fixed $m$  the point-wise  bounds
\bqn 
|Y_{k,m}(x)|^2 \ll  \begin{cases}   (1+|m|^3) \, \sqrt {\lambda_k}, & x=x_N,\, x_S, \\ (1+|m|^3) \, [\mathrm{dist} (x, \mklm{x_N,x_S})]^{-2} , & x \not=x_N,\, x_S,  \end{cases}
\eqn
as $k \to \infty$,  where we took into account the  bound \eqref{eq:31.10.2015}. In particular, this is consistent with \eqref{eq:3.12.2015}. Thus, spherical harmonics with fixed $m$ concentrate on the poles as $k$ becomes large. This fact is in accordance with the probability of finding a classical particle of zero angular momentum  near singular orbits and the shape of the corresponding equivariant quantum limits, see \cite[Section 9.2]{kuester-ramacher15}. Furthermore, if $\mathbf{c}$ denotes a closed geodesic on $S^2$ we obtain for the restriction of $Y_{k,m}$ to  $\mathbf{c}$ the $\L^\infty$-bounds
\bqn 
\norm{Y_{k,m|\mathbf{c}}}_{\infty}=\begin{cases} O_m( \lambda^{1/4}_k), & \text{if }x_N,x_S \in\mathbf{c},  \\ O_{m,\mathbf{c}}(1), & \text{otherwise}, \end{cases}
\eqn
as $k \to \infty$. The foregoing considerations can be immediately generalized to surfaces of revolution diffeomorphic to the $2$-sphere. 
\end{ex}

\section{Sharpness}
\label{sec:sharpness}

To conclude, we show that the obtained bounds are sharp and that, as in the classical case \cite{avacumovic, hoermander68} and \cite[Section 3.4]{sogge14}, they are already attained on the $2$-dimensional sphere. Denote by  $M=S^n$  the standard sphere in $\R^{n+1}$ endowed with the induced metric, and let $\Delta$ be the Laplace-Beltrami operator on $S^n$. The eigenvalues of $-\Delta$ are given by the numbers $ \lambda_k=k(k+n-1)$, where $k=0,1,2,3,\dots$  and the corresponding $d_k$-dimensional eigenspaces $\H_k$ are spanned by     the classical spherical functions $Y_{kl}$, $1 \leq l \leq d_k$, so that 
\bqn 
-\Delta Y_{kl} = \lambda_k \, Y_{kl}.
\eqn
The ${Y_{kl}}$ are orthonormal to each other, and by the spectral theorem we have the decomposition $
\L^2(M)= \bigoplus _{k=0}^\infty \H_k$. Now, let  $G\subset \SO(n)$ be a subgroup of the isotropy group of a point in $S^n\simeq \SO(n+1)/\SO(n)$, and 
\bqn 
\H_k=\bigoplus_{\gamma \in \widehat G} \H_k^\gamma
\eqn
be the decomposition of the eigenspace $\H_k$ into its isotypic components. It is clear that 
$
d_k=\sum_{\gamma \in \widehat G} m_\gamma(k) d_\gamma,
$
where $m_\gamma(k)$ denotes the multiplicity of $\pi_\gamma\in \gamma$ in $\H_k$. Let $\{Z_{kj}^\gamma\}\subset \mathrm{Span} \mklm{Y_{kl}}_{l=1}^{d_k}$ be an orthonormal basis of $\H_k^\gamma$ so that with $\mu=\mu_k-1$, $\mu_k=\sqrt{\lambda_k}$, 
\bqn 
K_{ \chi_{\mu}\circ \Pi_\gamma }(x,y)= \sum_{j=1}^{ m_\gamma(k) d_\gamma} Z_{kj}^\gamma(x) \overline{Z_{kj}^\gamma(y)},
\eqn
$ \chi_{\mu}\circ \Pi_\gamma $ being the projection onto $\H_k^\gamma$. By  Theorem  \ref{thm:main} we have  the bound
\bqn
|K_{ \chi_{\mu}\circ \Pi_\gamma}(x,x)| =|e_\gamma(x,x,\mu_k)-e_\gamma(x,x,\mu_k-1)| \leq C_{x,\gamma} \, \mu_k^{n-\kappa_x-1}, \qquad C_{x,\gamma}>0, \, x \in S^n,
\eqn
while the behaviour near singular orbits is described in  Theorem \ref{thm:15.11.2015}. 
We   now define for fixed $x\in S^n$ the {isotypic zonal eigenfunction}
\bqn
e_{\mu_k}^\gamma: S^n \ni y \longmapsto \sum_{j=1}^{ m_\gamma(k) d_\gamma} Z_{kj}^\gamma(x) \overline{Z_{kj}^\gamma(y)}\in \C,
\eqn
which is  an eigenfunction of $\sqrt{-\Delta}$ for the eigenvalue $\mu_k$ and satisfies
\bqn 
\norm{e_{\mu_k}^\gamma}_{\L^2}=\left (\sum_{j=0}^{ m_\gamma(k) d_\gamma} |Z_{kj}^\gamma(x)|^2 \right )^{1/2} =(K_{\chi_{\mu}\circ \Pi_\gamma}(x,x))^{1/2}.
\eqn
In order to examine the sharpness of the bounds obtained, we specialize  to  the case where $n=2$ and $G=\SO(2)$ acts by rotations around the symmetry axis through the poles. In this case, $\H_k^\gamma$, $\gamma \equiv m \in \Z$, $|m| \leq k$ is spanned by the spherical function
\begin{equation*}
Y_{k,m}(\phi,\theta)=\sqrt{\frac{2k+1}{4\pi}\frac{(k-m)!}{(k+m)!}} P_{k,m}(\cos \theta)e^{im\phi}, \qquad 0\leq \phi<2\pi, \, 0 \leq \theta < \pi, 
\end{equation*}
where  $P_{k,m}$ is the associated Legendre polynomial
\begin{equation*}
P_{k,m}(\alpha):=(-1)^m \left(1-\alpha^2\right)^{\frac{m}{2}}  \frac{d^{m}}{d\alpha^{m}} P_k(\alpha): =\frac{(-1)^m}{2^kk!}\left(1-\alpha^2\right)^{\frac{m}{2}}\frac{d^{k+m}}{d\alpha^{k+m}}\left(\alpha^2-1\right)^{k}.
\end{equation*}
Furthermore,  for the Legendre polynomials $P_k(\cos \theta)$ one has the asymptotics
\bq
\label{eq:27.11.2015}
P_k(\cos \theta)=\sqrt{\frac 2{ \pi k \sin \theta}} \cos \left ( \left (k+\frac 12 \right ) \theta -\frac \pi 4 \right ) +O\left (\frac 1{(k \sin \theta)^{3/2}}\right ),
\eq
where the remainder is uniform in $\theta$ on any interval  $[\eps,\pi - \eps]$ with $0<\eps$ small  \cite[p. 303]{hobson} 
\footnote{There is even an asymptotic expansion of $P_k(\cos \theta)$, provided that $k \sin \theta >1$.}. 
Thus,  in the special case where $m=0$ we see that  with  $\mu=\mu_k-1$ one has in the limit $k \to \infty $  
\bqn 
K_{ \chi_{\mu}\circ \Pi_\gamma }(x,x)= |Y_{k,0}(x)|^2={\frac{2k+1}{4\pi}} |P_{k,0}(\cos \theta)|^2 \approx\begin{cases} \sqrt{\lambda_k}, & x=x_N,x_S,\\
  \frac{1}{{\sin \theta}}\approx \frac 1 {\mathrm{dist}(x,\mklm{x_N,x_S})}, & x \in S^2-\mklm{x_N,x_S}, \end{cases}
\eqn
where $x_N$ and $x_S$ denote the  poles.  Consequently, we conclude that  the remainder estimates in  Theorems \ref{thm:main} and  \ref{thm:15.11.2015}   are sharp in the spectral parameter $\lambda$, but not optimal in the desingularization parameters $\tau_{i_j}$, since in the present case we have $\lambda\approx k^2$, $\sin \theta \approx \theta\approx \tau_{i_j}$, compare also Example \ref{ex:7.12.2015}.  Nevertheless, the estimate given in  Theorem \ref{thm:15.11.2015} qualitatively reflects the singular behaviour of $Y_{k,0}(x)$ as $x$ approaches the poles, and suggests that the asymptotic formula \eqref{eq:27.11.2015} should have a structural  explanation in terms of  caustics of oscillatory integrals.  On the other hand,  the  bound for $|Y_{k,0}(x)|$ implies similar bounds for $e^\gamma_{\mu_k}(y)=Y_{k,0}(x)\overline{Y_{k,0}(y)}$, and that for an eigenfunction $f\in \L^2(S^2)$ of $-\Delta$ belonging to a specific isotypic component with $\norm{f}_{\L^2}=1$ and eigenvalue $\lambda$  the estimate
\bqn 
|f(x)| \leq C_{x,\gamma} \, \lambda^{\frac{n-\kappa_x-1}4}, \qquad x \in S^2,
\eqn
 in Corollary \ref{cor:21.06.2016} cannot be improved in the eigenvalue aspect.  \\ 

To close, let us mention that in the  considered case $M=S^2$ and $G=\SO(2)$
the previous considerations imply for the equivariant counting function $N_{\gamma}(\lambda)$   of the Beltrami-Laplace operator with $\gamma \equiv m$ the estimate
\begin{align}
\label{eq:3.6.2017}
N_{\gamma }(\lambda) = d_\gamma \sum_{\lambda_k \leq \lambda} m_\gamma(k)=\sum_{k(k+1) \leq \lambda, \, |m| \leq k} 1\approx \sum_{|m| \leq k \leq \sqrt{\lambda}} 1\approx \sqrt{\lambda}-|m|,
\end{align}
as $\lambda \to +\infty $. From this one recovers the classical Weyl law
\[
N(\lambda)= \sum_{k(k+1) \leq \lambda} \dim \mathcal{H}_k =\sum_{\gamma \in \widehat G} N_{\gamma}(\lambda)\approx \sum_{|m| \leq \sqrt{\lambda}} (\sqrt{\lambda} -|m|)\approx (2\sqrt\lambda +1)\sqrt \lambda -2 \frac{\sqrt \lambda(\sqrt \lambda+1)}2=\lambda.
\]
The asymptotic formula  \eqref{eq:3.6.2017} implies  that  the equivariant Weyl law proved in \cite[Theorem 9.5] {ramacher10} is sharp up to a logarithmic factor in the remainder estimate, but  shows that  the remainder estimates in  Theorems \ref{thm:main} and  \ref{thm:15.11.2015}   are  not optimal in $\gamma \in \widehat G$. 

\appendix
\renewcommand*{\thesection}{\Alph{section}}

\section{Stationary phase principle and caustics}
\label{appendix}

Our analysis relies on the  generalized stationary phase principle, which we state below. Sketches of proofs can be found in  \cite[Theorem 3.3]{combescure-ralston-robert} and \cite[Theorem 2.12]{varadarajan97}. For a detailed proof, which includes explicit expressions for the coefficients and the remainder term in the stationary phase expansion, see \cite[Theorem 4.1 and Remark 4.2]{ramacher10}.

\begin{thm}[\bf Generalized stationary phase principle]
\label{thm:SP}
Consider   an  $n$-dimensional Riemannian manifold $\M$  with volume density $d\M$, a phase function  $\psi \in \Cinft(\M, \R)$,   and set
\bq
\label{eq:SPT}
I({\mu})=\int_\M e^{i\mu\psi(m)} a(m) \, \d\M(m), \qquad \mu >0,
\eq
where $a(m)\in \CT(\M)$ is an amplitude. In addition, assume that the critical set $$\mathcal C:=\Crit (\psi)=\mklm{m \in \M\mid \psi_\ast:T_m\M \rightarrow T_{\psi(m)}\R \text{ is zero}}$$ of the phase function $\psi$ is clean\footnote{That is, $\mathcal{C}$ is a smooth submanifold and the Hessian of $\psi$ is non-degenerate on $N_{m}\Ccal$ for all $m\in \Ccal$. In this case, we shall also say that the Hessian is \emph{transversally non-degenerate} or that the \emph{transversal Hessian} is non-degenerate at $m\in \Ccal$.}, meaning that  $\psi$ is a Morse--Bott function. Then, for all $\tilde N\in \N$ one has the asymptotic formula 
\bq
\label{eq:07.10.2015}
I(\mu) := e^{i\mu \psi_0}(2\pi/ \mu)^{\frac {n-p}{2}}\left [\sum_{r=0} ^{\tilde N-1} \mu^{-r} \mathcal{Q}_r (\psi,a)+ \mathcal{R}_{\tilde N}(\psi,a;\mu)\right ],
\eq
 where $p$ denotes the dimension of $\Ccal$,   $\psi_0$ is the constant value of $\psi$ on $\mathcal{C}$, and the expressions $\mathcal{Q}_r(\psi,a)$ and $\mathcal{R}_{\tilde N}(\psi,a;\mu)$ can be computed explicitly. 
 Furthermore,  there exist constants $C_{r,\psi}>0$ and  $\widetilde C_{\tilde N,\psi,\eps}>0$  such that 
\begin{align*}
|\mathcal{Q}_r(\psi;a)|&\leq  C_{r,\psi} \, \vol (\supp a \cap \mathcal{C}) \sup _{l\leq 2r} \norm{D^l a }_{\infty,\mathcal{C}}, \\
|\mathcal{R}_{\tilde N}(\psi,a;\mu) | &\leq \widetilde C_{\tilde N,\psi,\eps} \,  \mu^{-\tilde N}  \int_{\mathcal{C}} \, \sup_{l\leq 2\tilde N} \norm{D^l a }_{H^{(n-p)/2+\eps}(N_m\mathcal{C})} \d \sigma_{\mathcal{C}} (m),
\end{align*}
for any $\eps>0$, where $D^l$ are differential operators on $\M$ transversal to $\mathcal{C}$ of order $l$ independent of $\psi$, $H^s$ denotes  the $s$-th Sobolev space, and 
\bq
\label{eq:25.07.2017}
C_{r,\psi} \ll \sup_{m \in \mathcal{C} \cap \supp a} \norm {\Big ( \psi''(m)_{|N_m\mathcal{C}}\Big ) ^{-1}}^r \cdot |\det \psi''(m)_{|N_m\mathcal{C}}|^{-1/2}
\eq
with a similar bound for $\tilde C_{\tilde N,\psi,\eps}$. In particular, 
\bqn
\mathcal{Q}_0(\psi,a)= \int _{\mathcal{C}} \frac {a(m)}{|\det \psi''(m)_{|N_m\mathcal{C}}|^{1/2}} d\sigma_{\mathcal{C}}(m) e^{ i \frac\pi 4 \sigma_{\psi''}},
\eqn
where $d\sigma_{\mathcal{C}}$ stands for the induced volume density on $\mathcal{C}$ and 
 $\sigma_{\psi''}$ for  the constant value of the signature of the transversal Hessian $\psi''(m)_{|N_m\mathcal{C}}$ on $\mathcal{C}$.
\end{thm}
\qed 

\begin{rem}
\label{rem:25.07.2017}
In the setting of the previous theorem, suppose that $\M=\M_1\times \M_2$ is a product manifold, as well as  $\mathcal{C}=\mathcal{C}_1\times \mathcal{C}_2$, where $\mathcal{C}_i\subset \M_i$ are submanifolds of codimension $q_i$, and that the amplitude factorizes according to $a(m)=a_1(m_1) \, a_2(m_2)$,  $m=(m_1,m_2) \in \M$. Then, the remainder term can be estimated according to
\bqn
|\mathcal{R}_{\tilde N}(\psi,a;\mu) | \leq \widetilde C_{\tilde N,\psi,\eps} \,  \mu^{-\tilde N}  \prod_{i=1,2} \int_{\mathcal{C}_i} \, \sup_{l\leq 2\tilde N} \norm{D_i^l a_i }_{H^{q_i/2+\eps}(N_{m_i}\mathcal{C}_i)} \d \sigma_{\mathcal{C}_i} (m_i)
\eqn
for any $\eps>0$, the $D_i^l$ being differential operators on $\M_i$ transversal to $\mathcal{C}_i$ of order $l$. This allows one to estimate the remainder term by derivatives of the amplitudes $a_i$ of lower order. 
\end{rem}

\begin{rem}
As stated, the expansion \eqref{eq:07.10.2015} is valid for arbitrary $\mu>0$, though  the case of interest is when $\mu \to +\infty $, since then the error becomes smaller than the other terms. In essence, the point is that by Taylor's formula one has
\bqn 
\left | e^{it} - \sum_{k=0}^{N-1} \frac {(it)^k}{k!} \right | =O(|t|^N) \qquad \text{for arbitrary }\, t \in \R,
\eqn
no matter how large $|t|$ is, though the estimate is only meaningful for $|t| <1$. 
\end{rem}

One of the main concerns of  this paper is  extrapolating between stationary phase expansions of different orders. Thus, consider  an integral of the form \eqref{eq:SPT} with a clean critical set, let $\tau\geq 0$ be an additional parameter, and define the integral
\bqn 
I(\mu,\tau ) := \int_\M e^{i\mu\tau \psi(m)} a(m) \, \d\M(m).
\eqn
Depending on the value of $\tau$, it will exhibit different asymptotic behaviours in $\mu$. Indeed, for $\tau>0$ the integral $I(\mu,\tau )$ decreases with order $O(\mu^{-\frac{n-p}2})$, while for $\tau=0$ it is actually independent of $\mu$. This behaviour is reflected in the fact that  if we apply the previous theorem to the integral $I(\mu,\tau )$, either with $\mu \tau$ as asymptotic parameter, or with $\tau \psi$ as phase function, we would arrive at an expansion of the form \eqref{eq:07.10.2015} in which the coefficients in the expansion blow up as $\tau \to 0$ due to the  abrupt change of the  critical set of the phase function $ \tau \psi(m) $ when $\tau$ becomes zero. In general, if $\psi_\aleph\in \Cinft(\mathcal{M},\R)$ denotes a family of phase functions depending on a parameter $\aleph$ such that $\Crit (\psi_\aleph)$ is clean for generic values of $\aleph$, one understands by a \emph{caustic point} for this family  a parameter value $\aleph$ such that  $\Crit (\psi_\aleph)$ is not clean or where $\Crit (\psi_\aleph)$ changes drastically its dimension, compare  \cite{varadarajan97}. With this terminology,  in the situation above $\tau=0$ constitutes a caustic point. Nevertheless, it is possible to derive an adequate asymptotic expansion for $I(\mu,\tau)$ that smoothly interpolates between the different asymptotics, and takes into account the competing asymptotics $\mu \to +\infty $ and $\tau \to 0$,   based on the following simple idea.  Let $\epsilon\geq 0$ be   a fixed positive real number, and consider the integral
\bqn 
I_\eps({\mu}):=\int_\M e^{i\mu\psi(m)}   e^{-i\eps \psi(m)} a(m) \, \d\M(m).
\eqn
Clearly, $I(\mu)=I_\eps(\mu+\eps)$. Since $e^{-i\eps \psi}$ is independent of $\mu$, we can apply the previous theorem with $\mu+\eps$ as parameter, obtaining for each $\tilde N \in \N$ and each $\eps\geq 0$ the asymptotic formula
\bq
\label{eq:07.10.2015a}
I(\mu) = e^{i(\mu+\eps) \psi_0}\Big (\frac{2\pi}{\mu+\eps}\Big )^{\frac {n-p}{2}}\sum_{r=0} ^{\tilde N-1} (\mu+\eps)^{-r} \mathcal{Q}_r (\psi, e^{-i\eps \psi} a)+ \mathcal{R}_{\tilde N}(\psi, e^{-i\eps \psi} a; \mu+\eps).
\eq
Because 
\bqn 
\frac 1{\mu+\eps}=\frac 1 \mu \cdot \frac 1{1+\frac \eps\mu}=\frac 1 \mu \sum_{k=0}^\infty \Big ( \frac {-\eps}\mu \Big )^k= \frac 1 \mu -\frac \eps{\mu^2} + \frac{\eps^2}{\mu^3} - \cdots, \qquad \eps/\mu<1,
\eqn
the expansion \eqref{eq:07.10.2015} is consistent with the expansion \eqref{eq:07.10.2015a}, the respective corrections being of lower order. 
Now, if we apply  the previous argument to $I(\mu,\tau)=I(\mu \tau)$ we obtain
\bqn 
I(\mu,\tau ) = e^{i(\mu\tau +\eps) \psi_0}\Big (\frac{2\pi}{\mu\tau +\eps}\Big )^{\frac {n-p}{2}}\sum_{r=0} ^{\tilde N-1} (\mu\tau +\eps)^{-r} \mathcal{Q}_r (\psi, e^{-i\eps \psi} a)+ \mathcal{R}_{\tilde N}(\psi, e^{-i\eps \psi} a; \mu\tau +\eps)
\eqn
 as $\mu \to +\infty $. The formula is only meaningful for $\tau \mu +\eps>1$, and  simultaneously describes the asymptotic behaviour of $I(\mu,\tau)$  in the competing parameters $\tau$ and $\mu$. For $\eps>0$, it  interpolates between the asymptotics $O(\mu^{-\frac{n-p}2})$ and $O(\mu^0)$  in a smooth way; in fact, for $\tau=0$ it simply collapses to $ \int_\M  a \, \d\M$.


\begin{thebibliography}{10}

\bibitem{avacumovic}
V.G. Avacumovi\v{c}, \emph{{\"{U}ber die Eigenfunktionen auf geschlossenen
  Riemannschen Mannigfaltikeiten}}, Math. Z. \textbf{65} (1956), 327--344.

\bibitem{bredon}
G.~E. Bredon, \emph{Introduction to compact transformation groups}, Academic
  Press, New York, 1972, Pure and Applied Mathematics, Vol. 46.

\bibitem{bruening-heintze79}
J.~Br\"uning and E.~Heintze, \emph{Representations of compact {Lie groups} and
  elliptic operators}, Inv. math. \textbf{50} (1979), 169--203.

\bibitem{cassanas}
R.~Cassanas, \emph{Reduced {G}utzwiller formula with symmetry: {Case of a Lie
  group}}, J. Math. Pures Appl. \textbf{85} (2006), 719--742.

\bibitem{cassanas-ramacher09}
R.~Cassanas and P.~Ramacher, \emph{Reduced {Weyl} asymptotics for
  pseudodifferential operators on bounded domains {II}. {T}he compact group
  case}, J. Funct. Anal. \textbf{256} (2009), 91--128.

\bibitem{combescure-ralston-robert}
M.~Combescure, J.~Ralston, and D.~Robert, \emph{A proof of the {G}utzwiller
  semiclassical trace formula using coherent states decomposition}, Comm. Math.
  Phys. \textbf{202} (1999), 463--480.

\bibitem{donnelly78}
H.~Donnelly, \emph{{G-spaces, the asymptotic splitting of $L^2(M)$ into
  irreducibles}}, Math. Ann. \textbf{237} (1978), 23--40.

\bibitem{duistermaatFIO}
J.~J. Duistermaat, \emph{Fourier integral operators}, Birkh\"{a}user, 1996.

\bibitem{duistermaat-guillemin75}
J.~J. Duistermaat and V.W. Guillemin, \emph{The spectrum of positive elliptic
  operators and periodic bicharacteristics}, Inv. Math. \textbf{29} (1975),
  no.~3, 39--79.

\bibitem{grigis-sjoestrand}
A.~Grigis and J.~Sj{\"o}strand, \emph{Microlocal analysis for differential
  operators}, London Mathematical Society Lecture Note Series, vol. 196,
  Cambridge University Press, 1994.

\bibitem{helgason78}
S.~Helgason, \emph{Differential geometry, {Lie} groups, and symmetric spaces},
  American Mathematical Society, Providence Rhode Island, 2001.

\bibitem{hobson}
E.~W. Hobson, \emph{The theory of spherical and ellipsoidal harmonics},
  Cambridge University Press, 1931.

\bibitem{hoermander68}
L.~H\"ormander, \emph{The spectral function of an elliptic operator}, Acta
  Math. \textbf{121} (1968), 193--218.

\bibitem{hoermanderI}
L.~H{\"{o}}rmander, \emph{The analysis of linear partial differential
  operators}, vol.~I, Springer--Verlag, Berlin, Heidelberg, New York, 1983.

\bibitem{hoermanderIV}
\bysame, \emph{The analysis of linear partial differential operators}, vol.~IV,
  Springer--Verlag, Berlin, Heidelberg, New York, 1985.

\bibitem{iwaniec-sarnak95}
H.~{Iwaniec} and P.~{Sarnak}, \emph{{$L\sp \infty$ norms of eigenfunctions of
  arithmetic surfaces.}}, {Ann. Math. (2)} \textbf{141} (1995), no.~2,
  301--320.

\bibitem{kobayashi-nomizuII}
S.~Kobayashi and N.~Nomizu, \emph{Foundations of differential geometry},
  vol.~II, John Wiley \& Sons, INC., New York, 1969.

\bibitem{kuester-ramacher15}
B.~K\"uster and P.~Ramacher, \emph{Quantum ergodicity and symmetry reduction},
  J. Funct. Anal. \textbf{273} (2017), 41--124.

\bibitem{levitan52}
B.~M. Levitan, \emph{On the asymptoptic behavior of the spectral function of a
  self-adjoint differential equation of the second order}, Izv. Akad. Nauk SSSR
  Ser. Mat. \textbf{16} (1952), 325--352.

\bibitem{Marshall}
S.~Marshall, \emph{Sup norms of {Maass} forms on semisimple groups},
  arXiv:1405.7033 (2014).

\bibitem{marshall16}
S.~Marshall, \emph{{$L^p$} norms of higher rank eigenfunctions and bounds for
  spherical functions}, J. Eur. Math. Soc. (JEMS) \textbf{18} (2016), no.~7,
  1437--1493.

\bibitem{ramacher10}
P.~Ramacher, \emph{{Singular equivariant asymptotics and Weyl's law. On the
  distribution of eigenvalues of an invariant elliptic operator}}, J. reine
  angew. Math. \textbf{716} (2016), 29--101.

\bibitem{ramacher10add}
\bysame, \emph{{Addendum to ''Singular equivariant asymptotics and Weyl's
  law''}}, to be published in Crelle's Journal, DOI 10.1515/crelle-2017-0001,
  2017.

\bibitem{ramacher-wakatsuki17}
P.~Ramacher and S.~Wakatsuki, \emph{{Subconvex bounds for Hecke--Maass forms on
  compact arithmetic quotients of semisimple Lie groups}}, arXiv Preprint
  1703.06973, 2017.

\bibitem{sarnak_letter}
P.~Sarnak, \emph{Letter to {Morawetz}}, available at
  http://www.math.princeton.edu/sarnak/.

\bibitem{seeger-sogge}
A.~Seeger and C.~D. Sogge, \emph{Bounds for eigenfunctions of differential
  operators}, Indiana Univ. Math. J. \textbf{38} (1989), no.~3, 669--682.

\bibitem{shubin}
M.~A. Shubin, \emph{Pseudodifferential operators and spectral theory}, 2nd
  edition, Springer--Verlag, Berlin, Heidelberg, New York, 2001.

\bibitem{smith-sogge}
H.~F. {Smith} and C.~{Sogge}, \emph{{On the $L^p$ norm of spectral clusters for
  compact manifolds with boundary.}}, {Acta Math.} \textbf{198} (2007), no.~1,
  107--153 (English).

\bibitem{sogge88}
C.~Sogge, \emph{Concerning the {$\mathrm{L}^p$} norm of spectral clusters of
  second order elliptic operators on a compact manifold}, J. Funct. Anal.
  \textbf{77} (1988), 123--138.

\bibitem{sogge14}
\bysame, \emph{Hangzhou lectures on eigenfunctions of the {L}aplacian},
  Princeton University Press, 2014.

\bibitem{sogge-toth-zelditch}
C.~Sogge, J.~A. Toth, and S.~Zelditch, \emph{About the blowup of quasimodes on
  {Riemannian} manifolds}, J. Geom. Anal. \textbf{21} (2011), no.~1, 150--173.

\bibitem{stanhope-uribe}
E.~Stanhope and A.~Uribe, \emph{The spectral function of a {Riemannian}
  orbifold}, Ann. Glob. Anal. Geom. (2011), no.~40, 47--65.

\bibitem{stein-weiss}
E.~M. Stein and G.~Weiss, \emph{{Introduction to Fourier analysis on Euclidean
  space}}, Princeton University Press, Princeton, New Jersey, 1971.

\bibitem{varadarajan97}
V.~S. Varadarajan, \emph{The method of stationary phase and applications to
  geometry and analysis on {Lie} groups}, Algebraic and analytic methods in
  representation theory, Persp. Math., vol.~17, Academic Press, 1997,
  pp.~167--242.

\bibitem{wallach}
N.~R. Wallach, \emph{Real reductive groups}, vol.~I, Academic Press, Inc.,
  1988.

\bibitem{zelditch92}
S.~Zelditch, \emph{{Kuznecov sum formulae and Szego limit formulae on
  manifolds}}, Comm. PDE (1992), 221--260.

\bibitem{zelditch13}
\bysame, \emph{{Park City lectures on eigenfunctions}}, arXiv Preprint
  1310.7888, 2013.

\bibitem{zelditch08}
Steve {Zelditch}, \emph{{Local and global analysis of eigenfunctions on
  Riemannian manifolds.}}, {Handbook of geometric analysis. No. 1}, Somerville,
  MA: International Press; Beijing: Higher Education Press, 2008, pp.~545--658
  (English).

\end{thebibliography}

\providecommand{\bysame}{\leavevmode\hbox to3em{\hrulefill}\thinspace}
\providecommand{\MR}{\relax\ifhmode\unskip\space\fi MR }
\providecommand{\MRhref}[2]{%
  \href{http://www.ams.org/mathscinet-getitem?mr=#1}{#2}
}
\providecommand{\href}[2]{#2}

\end{document}